\documentclass[onefignum,onetabnum]{siamart171218}
\usepackage[margin=2.5cm]{geometry}
\usepackage{lineno,hyperref}
\modulolinenumbers[5]
\usepackage{lipsum}
\usepackage{graphicx}
\usepackage{epstopdf}
\usepackage{algorithmic}
\usepackage{algorithm}
\usepackage{wrapfig}
\usepackage[tight]{subfigure}
\usepackage{microtype}
\usepackage{enumerate}

\usepackage{amsfonts}
\usepackage[font=small,labelfont=bf]{caption}
\usepackage[english]{babel}
\usepackage{multirow}
\usepackage{wrapfig}
\usepackage{tabularx}
\usepackage[toc,page]{appendix}
\usepackage{marginnote}

\def\ie{{i.e. }}
\def\eg{{e.g. }}

\def\R{\mathbb{R}}

\def\u{\underline}

\DeclareMathOperator*{\argmin}{arg\,min}
\DeclareMathOperator*{\argmax}{arg\,max}

\graphicspath{{f/}}

\title{
Global Optimality in Separable Dictionary Learning \\with Applications to the Analysis of Diffusion MRI\thanks{Submitted to the editors Aug 23, 2019. First two authors contributed equally. \funding{B. Haeffele and R. Vidal were supported by the grant NSF 1618485.}}}

\author{Evan Schwab\thanks{Center for Imaging Science \& Department of Electrical and Computer Engineering, Johns Hopkins University.}
\and
Benjamin D. Haeffele\footnotemark[3] 
\and 
Ren\'{e} Vidal\thanks{Mathematical Institute for Data Science \& Department of Biomedical Engineering, Johns Hopkins University.}
\and
Nicolas Charon\thanks{Center for Imaging Science \& Department of Applied Mathematics and Statistics, Johns Hopkins University.}}

\begin{document}

%\begin{frontmatter}

\maketitle
%\title{Efficient Spatial-Angular Sparse Coding for HARDI}
%\title{Efficient Kronecker Optimization for Global Spatial-Angular Sparse Coding in Diffusion MRI}
%\title{Separable Spatial-Angular Dictionary Learning\\ for Diffusion MRI with Global Optimality}
%\title{Separable Dictionary Learning with Global Optimality\\
%and Applications to Diffusion MRI}
%\title{Separable Spatial-Angular Dictionary Learning with Global Optimality for Diffusion MRI}
%\title{Globally Optimal Spatial-Angular Dictionary Learning and Convolutional Sparse Coding for Diffusion MRI}


%\author{Evan Schwab, Benjamin Haeffele, Nicolas Charon and Ren\'{e} Vidal}%\fnref{myfootnote}}
%\address{Center for Imaging Science, Johns Hopkins University, Baltimore, MD USA}




\begin{abstract}
Sparse dictionary learning is a popular method for representing signals as linear combinations of a few elements from a dictionary that is learned from the data. In the classical setting, signals are represented as vectors and the dictionary learning problem is posed as a matrix factorization problem where the data matrix is approximately factorized into a dictionary matrix and a sparse matrix of coefficients. However, in many applications in computer vision and medical imaging, signals are better represented as matrices or tensors (e.g., images or videos). In such cases, instead of learning a large-scale dictionary tensor, it may be beneficial to exploit the multi-dimensional structure of the data to learn a more compact representation. One such approach is \emph{separable dictionary learning}, where one learns separate dictionaries for different dimensions of the data (e.g., spatial and temporal dimensions of a video). However, while there has been significant recent work on separable dictionary learning, typical formulations involve solving a non-convex optimization problem; thus guaranteeing global optimality remains a challenge. In this work, we propose a framework that builds upon recent developments in matrix factorization to provide theoretical and numerical guarantees of global optimality for separable dictionary learning. Specifically, we prove that local minima are guaranteed to be global when some dictionary atoms and the corresponding coefficients are zero. We also propose an algorithm to find such a globally optimal solution, which alternates between following local descent steps and checking a certificate for global optimality. We illustrate our approach on diffusion magnetic resonance imaging (dMRI) data, a medical imaging modality that measures water diffusion along multiple angular directions in every voxel of an MRI volume. State-of-the-art methods in dMRI either learn dictionaries only for the angular domain of the signals or in some cases learn spatial and angular dictionaries independently. In this work, we apply the proposed separable dictionary learning framework to learn spatial and angular dMRI dictionaries jointly and provide preliminary validation on denoising phantom and real dMRI brain data.

\end{abstract}
%\begin{AMS}
%  68Q25, 68R10, 68U05
%\end{AMS}
%\end{frontmatter}

\section{Introduction}
In signal processing, a well studied problem is that of reconstructing a signal from a set of noisy measurements by finding a representation of the signal in a chosen domain for which one can more easily process and analyze the data.  In the most general setting, one would like to represent (or approximate) a signal $y\in \mathbb{R}^N$ in terms of a dictionary $D \in \mathbb{R}^{N\times r}$ with $r$ dictionary atoms as 
\begin{equation}
\label{eq:rep}
y = Dw,
\end{equation}
where the coefficient vector $w \in \mathbb{R}^{r}$ is the representation of $y$ in terms of the dictionary $D$. This is in general an ill-posed problem and one typically imposes additional constraints on the reconstructed signal.  A common assumption is that $y$ is sparse with respect to the dictionary $D$, \ie $w$ has very few non-zero entries. This leads to the classical sparse coding problem:
\begin{equation}
\label{eq:SparseCoding}
\min_w \frac{1}{2}||Dw-y||_2^2 + \lambda ||w||_1,
%\min_c \ g(c) \ \ \textnormal{s.t.} \ \ \frac{1}{2}||Dc - s||_2^2 \leq \epsilon,
\end{equation}
where $\lambda>0$ is a trade-off parameter between the sparsity term $\|w\|_1$ and the data fidelity term $\|Dw-y\|_2^2$.
%where $\ell$ is a data fidelity term or loss function between the signal and the representation such as the usual Euclidean distance, $\ell(y,Dw) = \frac{1}{2}||Dw-y||_2^2$, and $g$ is a function to promote sparsity such as the $L_1$ norm, $g(w)=||w||_1$.

There are many signal processing applications of sparse coding including denoising \cite{Elad:TIP06}, super-resolution \cite{Jianchao:TIP10} and Compressed Sensing (CS) \cite{Candes:ACHA11}.  For example, the goal of CS is to minimize the number of samples needed to accurately reconstruct a signal in order to accelerate signal acquisition. The typical number of measurements needed is directly linked to the sparsity of the representation which is obviously dependent on the choice of dictionary $D$. For different types of signals there may be an array of known dictionaries that produce sparse representations (\eg Wavelets for natural or medical images \cite{Lustig:MRM07}).  However, prescribing a known dictionary for a new signal or data type may lead to suboptimal sparsity levels. 

\subsection{Dictionary learning}
\label{sec:dictionaryLearning}
To overcome this limitation, in sparse dictionary learning we learn a dictionary directly from the signal itself (or from a set of training examples).
%This problem is known as sparse dictionary learning and remains a very active research topic to date.
%%%Our goal in this work is to develop a new dictionary learning framework for dMRI which exploits the unique 6D spatial-angular data structure in order to provide sparser representations for applications of de-noising and compressed sensing.  In the next section, we review state-of-the-art dictionary learning methods for dMRI which have been confined only to the angular domain. We then present our separable spatial-angular viewpoint of dMRI and review the area of separable dictionary learning which could be applied to our problem.  In Section~\ref{sec:ProposedDictionaryLearning} we propose a novel separable dictionary learning method based on matrix factorization that improves upon the state of the art by guaranteeing global optimality.
%Then, with our learned dictionaries, we propose an efficient convolutional sparse coding framework in Section \ref{sec:ConvSparseCoding} to sparsely represent global dMRI data with experiments on phantom and real dMRI data.
Although many different methodologies exist, typical formulations assume that one is given a training set of $T$ signals $\{y_t\}_{t=1}^T$ that resemble the signals of interest, and the goal is to approximate each signal $y_t$ as a sparse linear combination of the atoms $D_i$ from a dictionary $D$. Therefore, one considers an optimization problem of the form:
%amounts to solving \eqref{eq:SparseCoding} jointly over $D$ and each coefficient vector $c_t$ with added constraints on the $N_D$ dictionary atoms $D_k$ to keep them from growing in magnitude, leading to an optimization problem of the form: 
\begin{equation}
\label{eq:DictionaryLearning}
\min_{D,\{w_t\}} \frac{1}{2} \sum_{t=1}^T ||Dw_t-y_t||_2^2 + \lambda ||w_t||_1 \ \ \textnormal{s.t.} \ \ ||D_i||_2 \leq 1 \ \ \textnormal{for} \ \ i=1,\dots, r,
\end{equation}
where the constraints $\|D_i\|_2 \leq 1$ are enforced to prevent an unbounded solution for $D$.  Letting $Y = [y_1,\dots, y_T]$ and $W = [w_1,\dots,w_T]$ the problem can be written more compactly in matrix from as:
\begin{equation}
\label{eq:DictionaryLearningMat}
\min_{D,W} \frac{1}{2}||DW-Y||_F^2 + \lambda ||W||_1 \ \ \textnormal{s.t.} \ \ ||D_i||_2 \leq 1 \ \ \textnormal{for} \ \ i=1,\dots, r,
\end{equation}
where $||X||_F = \sqrt{\sum_{i,j}|X_{i,j}|^2}$ is the Frobenius norm and $\|X\|_1 = \sum_{ij} |X_{ij}|$ is the $\ell_1$ norm.
%where $A = [a_1,\dots,a_T]$ and $S=[s_1,\dots,s_T]$ and $||X||_F = \sqrt{\sum_{i,j}|X_{i,j}|^2}$ is the Frobenius norm.  

Many algorithms for solving this dictionary learning problem (or its variants) have been proposed \cite{Engan1999,Kreutz-Delgado2003,Mairal2010,dohmatob2016learning}. One well-known method is KSVD \cite{Elad:TSP06}, which alternates between solving for the $w_t$'s while $D$ is fixed (sparse coding update), and updating $D_i$ one by one (dictionary learning update) via SVD decomposition. Once a dictionary is learned, it can be used in \eqref{eq:SparseCoding} for sparsely representing a new (test) signal. One important downside of the dictionary learning problem \eqref{eq:DictionaryLearning} is that the joint optimization over $D$ and $W$ results in a non-convex problem, and therefore guaranteeing a globally optimal solution is a difficult challenge. At best, optimization algorithms such as gradient descent may reach stationary points which can be either local minima or saddle points, providing sometimes suboptimal dictionary solutions.

\subsection{Separable dictionary learning}
\label{sec:SepDictionaryLearning}
Note that the classical dictionary learning problem in \eqref{eq:DictionaryLearning} is defined for vector valued signals $y_t \in \mathbb{R}^N$. However, there are many applications where signals have additional structure that we would like to preserve.
%While \eqref{eq:DictionaryLearning} is the classical setting for vector valued signals $y_t \in \mathbb{R}^N$, the problem of dictionary learning becomes increasingly more complex for signals with additional structure that we would like to preserve. For instance, for image data, instead of vectorizing an image and learning a dictionary, it may be useful to preserve the 2D structure of the image for computer vision tasks. 
%In this case, one can attempt to learn separate 1D dictionaries for each dimension of the image to reduce complexity.  
%Another increasingly important setting is one in which we have structured data that is very large-scale and high-dimensional (\eg 4D time-series data) in which learning a ``global" dictionary for the entire dataset may be computationally infeasible. 
%It may be more useful, for instance, to preserve the 2D structure of an image and learn two separate dictionaries for each dimension, than to vectorize the image and learn a single vector-valued dictionary. 
% RENE CHANGES
For instance, for image data, instead of vectorizing an image and learning a vector-valued dictionary, it may be useful to preserve the 2D structure of the image. Similarly, in the case of diffusion magnetic resonance imaging (dMRI) data, it may be useful to preserve the spatial-angular structure of the data. 
Just as for vector valued signals (1D tensor or 1-tensor), the dictionary is a matrix (2-tensor), for matrix valued signals the dictionary can be written as a 3-tensor and likewise for any n-tensor valued signal. However, the resulting \emph{tensor dictionary learning} problem may be computationally intractable for very large-scale high-dimensional datasets. 
To reduce computational complexity and memory requirements, one can assume that the dictionary elements are \emph{separable} along the individual dimensions of the data, which leads to a problem known as \emph{separable dictionary learning}.
%In this case, learning smaller-scale dictionaries for each individual dimension can be a more manageable problem, reducing computational complexity and memory requirements. Formally, this is known as \textit{separable} dictionary learning or \textit{tensor} dictionary learning due to the tensor structure of the data, (\eg 2D images are 2-tensors and 4D time-series are 4-tensors). 
%In this paper we will consider as an application uniquely structured, large-scale and high-dimensional diffusion magnetic resonance imaging (dMRI) data, used for analyzing neurological diseases such as Alzheimer's Disease.

To mathematically introduce the notation of separable dictionary learning, consider a 2D signal $S \in \mathbb{R}^{G \times V}$. Similarly to \eqref{eq:rep}, $S$ can be written as a linear combination of $r$ dictionary atoms $\Phi_k \in \mathbb{R}^{G\times V}$ as $S = \sum_{k=1}^r c_k \Phi_k$. Now, recall from \eqref{eq:DictionaryLearning} that in the case of vector-valued signals, dictionary learning utilizes $T$ training examples $y_t$, each with coefficients $w_t$, such that $y_t = Dw_t$. For a matrix-valued signal, with $T$ training examples $S_t$, each with coefficients $c_{k,t}$, $S_t = \sum_{k=1}^r c_{k,t} \Phi_k$. In this setting, the full dictionary and set of coefficients can be represented as tensors of size $G \times V \times r$ and $G \times V \times r \times T$, respectively, which can be prohibitively large to estimate in practical applications.

In the separable dictionary learning setting, we assume each atom $\Phi_k$ can be decomposed as the product of two atoms along the individual dimensions, \ie $\Phi_k = \Gamma_i \Psi_j^\top$, where $\Gamma_i \in \mathbb{R}^G$ and  $\Psi_j \in \mathbb{R}^V$. Therefore, we write the signal as a bilinear combination of separable dictionaries i.e. $S = \Gamma C \Psi^\top = \sum_{i,j} c_{ij} \Gamma_i \Psi_j^\top$, where $\Gamma \in \mathbb{R}^{G \times r_1}$ and $\Psi \in \mathbb{R}^{V \times r_2}$ are dictionaries for the first and second dimensions respectively, and $C\in\mathbb{R}^{r_1 \times r_2}$ is the set of joint coefficients between dictionaries. Then, for each training example $S_t$ represented by coefficients $C_t$, we have $S_t = \Gamma C_t \Psi^\top$.
With these expressions, the separable dictionary learning problem can be stated as follows:
\begin{equation}
\label{eq:SeparableDictionaryLearning}
\min_{\Gamma,\Psi,\{C_t\}} \frac{1}{2} \sum_{t=1}^T|| \Gamma C_t \Psi^\top -  S_t ||^2_F + \lambda ||C_t||_1 \ \ \textnormal{s.t.} \ \ ||\Gamma_i||_2\leq 1,  ||\Psi_j||_2\leq 1 \ \ \forall \ \ (i,j).
\end{equation}
%where each training example $S_t$ can be taken to be a set of dMRI brain volumes, or even smaller spatial patches, as is commonly considered in dictionary learning problems to ease computational complexity \cite{}.  

%Then, just as we were able to write the vectored case compactly as matrices $Y=DW$, we can write the matrix case compactly as a 3-tensors: $\u{S} = \u{C} \times_1 \Gamma \times_2 \Psi$, where $\u{S} \in \mathbb{R}^{G\times V\times T}$ and $\u{C} \in \mathbb{R}^{N_\Gamma \times N_\Psi \times T}$ are comprised of stacking slices $S_t$ and $C_t$ along the third dimension. Here the notation $\times_n$ stands for matrix multiplication on the $n^{th}$ dimension of the tensor $\u{C}$\footnote{We refer the reader to review tensor notation and decompositions in \cite{de2000multilinear}}.

The problem of learning separable dictionaries via \eqref{eq:SeparableDictionaryLearning} and its multidimensional tensor generalization for higher-dimensional signals has received significant attention in the literature. The work of \cite{Hawe:CVPR13,Zhang:ICSP16,Zhang:Neurocomputing17} solve variations of \eqref{eq:SeparableDictionaryLearning} using conjugate gradient methods over smooth manifolds.  On the other hand, \cite{Qi:CVPR16,Zubair:DSP13} use tensor factorization techniques in which one solves alternatively for each mode of the tensor using the classical vector-valued dictionary learning techniques after $n$-mode unfolding. However, this loses the computational gain of maintaining a tensor structure.  The work of \cite{Duan:ICPR12, Roemer:ICASSP14, Stevens:AIS17} use decompositions such as Tucker, Kruskal-Factor and tensor SVDs, while \cite{dantas2017learning} considers a dictionary as the sum of Kronecker products. 
%To provide some theoretical analysis for Kronecker and tensor structured data, \cite{Shakeri:ISIT16,Shakeri:arXiv16,Shakeri:ICASSP17} examine lower bounds on minimax risk. 
Finally, \cite{Bahri:ICCV17,ghassemi2017stark} propose low-rank variations of the separable dictionary learning problem.

As we recall for \eqref{eq:DictionaryLearning}, one key difficulty in dictionary learning is the lack of guarantees of global optimality due to the non-convexity of the joint optimization over the dictionary and coefficients. This issue is especially difficult for separable dictionary learning because the number of variables to jointly optimize over increases from two to three or more. To the best of our knowledge, none of the aforementioned work on separable dictionary learning come equipped with guarantees of global optimality, and so their solutions may correspond to a local minimum or saddle point and may also heavily depend on initialization.

\subsection{Paper contributions}
The main contribution of this work is a new framework for solving the separable dictionary learning in \eqref{eq:SeparableDictionaryLearning} with guarantees of global optimality. To do this, we build upon recent theoretical work on matrix factorization \cite{Bach:arXiv08,Haeffele:PAMI2019} which has been applied previously to provide theoretical guarantees for the original dictionary learning problem \eqref{eq:DictionaryLearning}.

Specifically, in Section~\ref{sec:Background}, we recall how classical dictionary learning \eqref{eq:DictionaryLearning} can be framed as a matrix factorization problem and make a quick summary of the global optimality results obtained in \cite{Haeffele:PAMI2019}. Then in Section~\ref{sec:proposed}, we consider a fairly general class of tensor factorization problems for which we obtain similar theoretical results of global optimality. In Section \ref{sec:glob_opt_dict_learning}, based on those results, we specialize the analysis to the case of separable dictionary learning in order to derive verifiable conditions for the global optimality of solutions. We also show that such formulations of separable dictionary learning can be equivalently understood as low-rank tensor factorizations of the data. Then, in Section~\ref{sec:algorithm}, we derive a novel algorithm to find optimal and compact separable dictionaries under our model, which we first experiment, in Section~\ref{sec:numerics}, on a small synthetic problem to illustrate the algorithm's convergence properties. Finally, Section~\ref{sec:dMRIbackground} provides a few preliminary results of the approach for learning separable spatial-angular dictionaries from diffusion magnetic resonance imaging data as well as some comparisons with other methods in basic denoising experiments. %{\color{red} Are we going to mention the equivalence with nuclear norm regularization?}

%The purpose of this work is to address this important issue by proposing a novel formulation and an algorithm for learning \textit{separable dictionaries} with guarantees of global optimality. Our approach follows the footsteps of previous works by the authors on matrix factorization \cite{Haeffele:arXiv2015-PosFactor,Haeffele:ICML14,Haef} and separable sparse coding for dMRI \cite{Schwab:MIA-ArXiv17}. Because of the unique separable structure defined over a product of a spatial and angular diffusion domain, dMRI is a well-suited application for separable dictionary learning and one for which prior work has been restricted to using more classical methods only on the angular domain.

%Alternatively, the work of \cite{Soltani:BITNM16,Jiang:17a,Jiang:17b,Zhang:15,Zhang:IJCAI16} choose to factor the tensor data as a linear combination of tensors and solves this convex problem using variations of ADMM, or tensor SVD.  This type of tensor combination does not fit our model of sparse 

%Unlike the state of the art, we propose a novel separable dictionary learning framework with global optimality guarantees of convergence which borrows its results from general matrix factorization.

%\section{Separable Dictionary Learning with Global Optimality}
%\label{sec:ProposedDictionaryLearning}
%\subsection{Separable Dictionary Learning as Matrix Factorization}

%One of the difficulties of dictionary learning is the non-convexity with respect to dictionaries. 

\section{Background}
\label{sec:Background}

\subsection{Dictionary learning as matrix factorization}
\label{sec:matrixfactorization}

The general problem of matrix factorization is concerned with finding factors $D$ and $W$, such that a data matrix $Y$ can be approximated by a matrix $X=DW$. Naturally, the dictionary learning problem \eqref{eq:DictionaryLearning} can be thought of in this way. In \cite{Bach:arXiv08,Haeffele:arXiv2015-PosFactor,Haeffele:ICML14,Haeffele:PAMI2019} the authors develop a general matrix factorization framework for a number of applications including the dictionary learning problem.
%    In what follows, we propose an alternative formulation of \eqref{eq:SeparableDictionaryLearning} for which we are guaranteed that a local minimum becomes a global minimum.  
The key insight is an equivalence relation between the non-convex factorized problem with respect to the factors $D$ and $W$ and a convex problem with respect to $X$, which allows one to obtain guarantees of global optimality for $(D,W)$.

%For separable dictionary learning we are interested in the extension to three factors, $X=\Gamma C \Psi^\top$.  Furthermore, we require joint factorizations for each training example, namely $X_t = \Gamma C_t \Psi^\top$ for all $t$. To build up to our proposed extension, we first review the main results of two-factor matrix factorization.

%\subsection{Matrix Factorization}
First, the non-convex matrix factorization problem can be written as:
\begin{equation}
\label{eq:fUV}
\min_{D,W} \ell(Y,DW) + \lambda \Theta(D,W),
\end{equation}
where $\ell$ is a data fidelity term or loss that measures the error between the original signal $Y$ and the reconstruction $X=DW$, and $\Theta$ is a regularizer on the factors $D$ and $W$ which promotes particular properties relevant to the problem. For the dictionary learning problem \eqref{eq:DictionaryLearning}, $\ell(Y,DW) = \frac{1}{2}||DW-Y||_F^2$. Furthermore it can be shown that the constraints $||D_i||_2\leq 1$ can be combined with the sparsity term $||W||_1$ to get $\Theta(D,W) = \sum_{i=1}^r ||D_i||_2||W^\top_i||_1$, where $W^\top_i \in \mathbb{R}^T$ is the $i^{th}$ row of $W$. Then \eqref{eq:fUV} is an equivalent problem formulation of \eqref{eq:DictionaryLearning}. The goal of recasting the dictionary learning problem as a matrix factorization problem is to relate \eqref{eq:fUV}, which is non-convex with respect to $D$ and $W$, to a convex problem with respect to $X$.

\subsection{Global optimality for matrix factorization}
\label{sec:globalMatrixFactorization}

To derive conditions for the global optimality, \cite{Haeffele:PAMI2019} first impose the regularizer to be of the specific form $\Theta(D,W) = \sum_i^r \theta(D_i,W^\top_i)$ where $\theta$ is a rank-1 regularizer that must satisfy the following properties:
\begin{definition}[from \cite{Haeffele:PAMI2019}]
\label{def:theta}
A function $\theta : \mathbb{R}^N \times \mathbb{R}^T \rightarrow \mathbb{R}_+ \cup \infty$ is said to be a \textbf{rank-1 regularizer} if
\begin{enumerate}
\item $\theta(u,v)$ is positively homogeneous with degree 2, \ie $\theta(\alpha u,\alpha v) = \alpha^2\theta(u,v) \ \forall \alpha \geq 0, \ \forall (u,v)$.
\item $\theta(u,v)$ is positive semi-definite, \ie $\theta(0,0) = 0$ and $\theta(u,v) \geq 0 \ \forall (u,v)$.
\item For any sequence $(u_n,v_n)$ such that $||u_n v_n^\top|| \rightarrow \infty$, we have that $\theta(u_n, v_n) \rightarrow \infty$.
\end{enumerate}
\end{definition}

It is easy to show that the choice of $\theta(u,v) = ||u||_2||v||_1$ fits this definition. Another example of $\theta$ satisfying Definition~\ref{def:theta} that can be used for dictionary learning is $\theta(u,v) = \|u\|_2 (\|v\|_2 + \alpha \|v\|_1)$ which promotes limiting the number of columns in $u$ and $v$ and also sparsity in $v$ as analyzed in \cite{Bach:arXiv08}.
%(See \cite{Haeffele:PAMI2019} for a list of examples of $\theta$ for various problem types.) 

Now, in order to connect the non-convex problem in \eqref{eq:fUV} with a convex problem with respect to matrix $X$, we introduce a related regularizer $\Omega_\theta(X)$ which depends on $\theta$:
%. With rank-1 regularizer $\theta$ defined, we can introduce  the definition of $\Omega_\theta$ (hereafter renamed with $\theta$ to show dependence of the rank-1 regularizer and since the infimum is taken also over rank $r$) is:
\begin{definition}[from \cite{Haeffele:PAMI2019}] 
\label{def:Omega}
Given a rank-1 regularizer $\theta$ that satisfies the conditions of Definition~\ref{def:theta}, the \textbf{matrix factorization regularizer} $\Omega_\theta: \mathbb{R}^{N\times T} \rightarrow \mathbb{R}_+ \cup \infty$ is defined as:
\begin{equation}
\label{eq:Omega}
\Omega_\theta(X) \equiv \inf_{r\in \mathbb{N}_+} \inf_{D,W} \sum_{i=1}^r \theta(D_i,W^\top_i) \ \ \textnormal{s.t.} \ \ DW = X.
\end{equation}
If the infimum is achieved for some $D, W$ and $r$ then we say that $DW$ is an optimal factorization of $X$.
\end{definition}
It is important to note that the number of dictionary atoms $r$ becomes an important variable for finding an optimal matrix factorization in this definition. As a motivating example for the origin of $\Omega_\theta$, when $\theta(D_i,W^\top_i) = ||D_i||_2||W^\top_i||_2$, $\Omega_\theta(X)$ becomes the variational definition of the nuclear norm:
\begin{equation}
||X||_* \equiv \inf_{r\in \mathbb{N}_+} \inf_{D,W} \sum_{i=1}^r ||D_i||_2 ||W^\top_i||_2 \ \ \textnormal{s.t.} \ \ DW = X.
\end{equation}

From the results of \cite{Haeffele:PAMI2019}, $\Omega_\theta(X)$ is a gauge function (and even a norm if $\theta$ is symmetric, \ie $\theta(-u,v)=\theta(u,v)$ or $\theta(u,-v)=\theta(u,v)$ for all $u,v$), which leads to the new convex optimization problem with respect to $X$:
\begin{equation}
\label{eq:FUV}
\min_X \ell(Y,X) + \lambda \Omega_{\theta}(X).
\end{equation}

Since \eqref{eq:FUV} is convex, a local minimum $\hat{X}$ is guaranteed to be global.  The question answered in \cite{Haeffele:PAMI2019} is then how to relate a local minimum $(\tilde{D},\tilde{W})$ of the non-convex \eqref{eq:fUV} to a global minimum of the convex \eqref{eq:FUV} and when, if ever, we can say something about a global minimum $(\hat{D},\hat{W})$ of \eqref{eq:fUV}.
%To overcome the lack of global guarantees for the non-convex problem \eqref{eq:fUV}, the aim is to link \eqref{eq:fUV} to the closely related convex problem with respect to $X=UW^\top$:
%\begin{equation}
%\label{eq:FUV}
%\min_X \ell(S,X) + \lambda \Omega_{\Theta}(X),
%\end{equation}
%where $\Omega_{\Theta}(X)$ is a regularizer which can be seen as an infimum of $\Theta$ over $U$ and $W$ such that $X=UW^\top$.
%\subsection{Global Optimality for Matrix Factorization and Dictionary Learning}
%\label{sec:globalMatrixFactorization}
First, it is evident that \eqref{eq:FUV} provides a global lower bound of \eqref{eq:fUV} because $\Omega_\theta$ is the infimum of $\Theta$ and $\ell(Y,X) = \ell(Y,DW)$.
%The authors prove that for a rank-1 regularizer $\theta$, $\Omega_\theta$ is a norm on $X$, which means that for an $\ell$ that is convex with respect to the second argument, \eqref{eq:FUV} is a convex problem. Therefore, using convex optimization, a global minimum $\hat{X}$ can be reached if it exists.
The main result is then that under certain conditions local minima $(\tilde{D},\tilde{W})$ of the non-convex \eqref{eq:fUV} are optimal factorizations of $X$, such that $\hat{X} = \tilde{D}\tilde{W}$. In other words, given a local solution $(\tilde{D},\tilde{W})$ to \eqref{eq:fUV}, we can write a matrix $X = \tilde{D}\tilde{W}$ and under certain conditions, it turns out that the matrix $X$ is a global minimum of \eqref{eq:FUV}, \ie $X\equiv\hat{X}$. Therefore, $(\tilde{D},\tilde{W})$ is in fact a global minimum of \eqref{eq:fUV}, $(\hat{D},\hat{W})$. We restate this main theorem of \cite{Haeffele:PAMI2019} here: 
\begin{theorem}
\label{thm:ben}
[from \cite{Haeffele:PAMI2019}] Given a function $\ell(S, X)$ that is convex and once differentiable w.r.t. $X$, a rank-1
regularizer $\theta$ that satisfies the conditions in Definition~\ref{def:theta}, with constants $r\in \mathbb{N}_+$, and $\lambda>0$, local minima $(\tilde{D},\tilde{W})$ of \eqref{eq:fUV} are globally optimal if $(\tilde{D}_i,\tilde{W}_i^\top) = (0,0)$ for some $i\in [r]$.  Moreover, $\hat{X} = \tilde{D}\tilde{W}$ is a global minima of \eqref{eq:FUV} and $\tilde{D}\tilde{W}$ is an optimal factorization of $\hat{X}$.
\end{theorem}

Since $\theta$ is general, this matrix factorization can be applied to many problems such as low-rank, non-negative matrix factorization, sparse PCA as well as the desired dictionary learning.  However, one important downside for the application of dictionary learning is that the choices of $\theta$ stated above are not well suited to checking the criteria of Theorem~\ref{thm:ben} in practice. In particular verifying if a point is stationary or a local minimum remains difficult. Therefore, finding globally optimal solutions for classical dictionary learning still remains a challenging problem. In the next section we will extend the results of \cite{Haeffele:PAMI2019} for the more complex structured \textit{separable} dictionary learning. %{\color{red} (I might delete this.  We have the exact same computational difficulty here, just we're really doing low-rank regularization, which is tractable.) Moreover, given a certain choice of regularizer, we do not run into the same issues as for the classical dictionary learning problem when finding globally optimal solutions in practice.}

\section{Global optimality for tensor factorization}
\label{sec:proposed}

In this section we will extend the framework and theories of global optimality developed for matrix factorization \cite{Haeffele:ICML14,Haeffele:PAMI2019} to the more general case of \textit{tensor} factorization. Then, just as classical dictionary learning was cast as a matrix factorization problem, in Section \ref{sec:glob_opt_dict_learning} we will show that separable dictionary learning is a particular tensor factorization problem.

\subsection{Tensor factorization problem and notation}
\label{sec:tensorfactorization}
Similar to matrix factorization, tensor factorization is concerned with finding factors that decompose an $n$-tensor $\u{S} \in \mathbb{R}^{N_1\times N_2 \times \cdots \times N_n}$, where we use an underlined capital letter to denote a tensor. There are two main types of tensor decompositions: rank-1 decomposition, where each factor $f_i \in \mathbb{R}^{N_i}$ is a vector such that $\u{X} = f_1 \otimes f_2 \otimes \cdots \otimes f_n$ and $\otimes$ denotes the tensor product, and the Tucker decomposition, in which there is a core $n$-tensor $\u{C} \in \mathbb{R}^{r_1\times r_2 \cdots \times r_n}$ and matrix factors $F_i \in \mathbb{R}^{N_i\times r_i}$ such that $\u{X} = \u{C} \times_1 F_1 \times_2 F_2 \dots \times_n F_n$, where $\times_n$ stands for matrix multiplication along the $n^{th}$ dimension of the core tensor $\u{C}$. (See \cite{de2000multilinear} for a review of tensor decomposition.)

%For simplicity, in this paper, we will stick to the case of $n=3$, though are methods are general to any $n\geq3$.
As in this paper we are interested in separable dictionary learning and its application to spatial-angular dMRI signals, we will simplify the rest of the presentation by restricting our discussion to decomposition of 3-tensors (although the results of this section can be readily extended to general $n$). Using notations consistent with \eqref{eq:SeparableDictionaryLearning}, we will consider the tensor signal $\u{S} \in \mathbb{R}^{G\times V \times T}$ where each slice $S_t\in \mathbb{R}^{G\times V}$ corresponds to a training example. Our goal will be to decompose this tensor (at least approximately) as $\u{S} = \u{C} \times_1 \Gamma \times_2 \Psi$, where $\u{C} \in \mathbb{R}^{r_1 \times r_2 \times T}$ is a core tensor and $\Gamma \in \R^{G \times r_1}$ and $\Psi \in \R^{V \times r_2}$ are two matrix factors. Note that this is equivalent to writing $S_t = \Gamma C_t \Psi^\top$ for each $t=1,\dots,T$ and that it can be interpreted as a Tucker decomposition of $\u{S}$ where the last factor $F_3$ is set to the identity unlike the setting analyzed \eg in \cite{Haeffele:arXiv2015-PosFactor}, which does not impose this constraint. To index the tensor $\u{C}$, all 2D slices will be written with an upper case letter and a single index, \eg $C_t \in \mathbb{R}^{r_1\times r_2}$ or $C_i \in \mathbb{R}^{r_2\times T}$ or $C_j \in \mathbb{R}^{r_1 \times T}$. Furthermore, 1D slice vectors of $\u{C}$ will be written with an upper case letter and two indices, \eg $C_{i,j}\in\mathbb{R}^T$ or $C_{i,t} \in \mathbb{R}^{r_2}$ or $C_{j,t} \in \mathbb{R}^{r_1}$. Finally, single elements of $\u{C}$ will be simply written as $c_{i,j,t}$.

Similar to the matrix factorization problem in \eqref{eq:fUV}, we will consider general tensor factorization problems formulated as:
%\begin{equation}
%\min_{\Gamma,\Psi,\u{C}} \sum_{t=1}^T \ell(S_t,\Gamma C_t \Psi^\top) + \lambda \sum_{t=1}^T ||C_t||_1 \ \ \textnormal{s.t.}\ \ ||\Gamma_i||_2\leq 1, ||\Psi_j||_2 \leq 1 \ \ \forall \ (i,j).
%\end{equation}
%We then incorporate the constraints on $\Gamma$ and $\Psi$ into the objective function as:
\begin{equation}
\label{eq:fT}
\min_{\Gamma,\Psi,\u{C}} \{f(\Gamma,\Psi,\u{C}) \equiv \ell(\u{S}\ ,\ \u{C}\!\times_1\!\Gamma \times_2\!\Psi) + \lambda \Theta(\Gamma,\Psi,\u{C})\},
\end{equation}
where the first term $\ell$ is a measure of similarity to the data and $\Theta$ is a certain regularizer that enforces some constraints on the factorization. We will make the assumption that $\ell$ is separable in the different $t$ slices of the tensor, namely, with a slight abuse of notation, we will write $\ell(\u{S},\u{Y}) = \sum_{t=1}^{T} \ell(S_t,Y_t)$. For instance, the separable dictionary learning problem of \eqref{eq:SeparableDictionaryLearning} that we shall consider more specifically in the next section corresponds to the choice $\ell(\u{S},\u{C}\times_1 \Gamma \times_2 \Psi) = \frac{1}{2} ||\u{C}\times_1 \Gamma \times_2 \Psi - \u{S} ||_F^2 = \frac{1}{2}\sum_{t=1}^T||\Gamma C_t \Psi^\top - S_t||_F^2$ while the constraints $||\Gamma_i||_2\leq1$, $||\Psi_i||_2\leq1$ and sparse $\u{C}$ can be shown to be achieved by introducing a regularizer of the form:
\begin{equation}
%\label{eq:ThetaT}
    \Theta(\Gamma,\Psi,\u{C}) = \sum_{i=1}^{r_1} \sum_{j=1}^{r_2} ||\Gamma_i||_2||\Psi_j||_2||C_{i,j}||_1.
\end{equation}

Now, as in the previous case of matrix factorization, we wish to link stationary points of the non-convex $f(\Gamma,\Psi,\u{C})$ with a global minimum of a convex function with respect to $\u{X}$.

\subsection{A global optimality criterion}
To develop the theories of global optimality for separable dictionary learning we begin by extending Definitions~\ref{def:theta} and \ref{def:Omega} from Section~\ref{sec:globalMatrixFactorization}.  First, we will consider a regularizer in \eqref{eq:fT} of the form $\Theta(\Gamma,\Psi,\u{C}) = \sum_{i=1}^{r_1}\sum_{j=1}^{r_2}\theta(\Gamma_i,\Psi_j,C_{i,j})$, where $\theta$ satisfies the following conditions:
\begin{definition}
\label{def:theta_T}
A function $\theta: \mathbb{R}^G \times \mathbb{R}^V \times \mathbb{R}^T \rightarrow \mathbb{R}_+ \cup \infty$ is said to be a \textbf{rank-1 regularizer} if
\begin{enumerate}
\item $\theta$ is positively homogeneous of degree 3, \ie $\theta(\alpha \gamma,\alpha \psi,\alpha c) = \alpha^3\theta(\gamma,\psi,c) \ \forall \alpha \geq 0, \ \forall (\gamma,\psi,c)$. 
\item $\theta$ is positive semi-definite and $\theta(\gamma,\psi,c)>0$ if and only if $\gamma \otimes \psi \otimes c \neq 0$.
\item For any sequence $(\gamma_n,\psi_n,c_n)$ such that $||\gamma_n \otimes \psi_n \otimes c_n|| \rightarrow \infty$, we have $\theta(\gamma_n, \psi_n, c_n) \rightarrow \infty$.
\end{enumerate}
\end{definition}
%where $\otimes$ is the tensor outer product, \ie for $\u{X} = a\otimes b \otimes c$, $\u{X}_{i,j,k} = a_ib_jc_k$.

Then, similarly to Definition~\ref{def:Omega}, we define the related regularizer for tensor $\u{X}$:
\begin{definition}
\label{def:Omega_T}
Given a rank-1 regularizer $\theta$ that satisfies the conditions of  Definition~\ref{def:theta_T}, the \textbf{tensor factorization regularizer} $\Omega_{\theta}: \mathbb{R}^{G\times V \times T} \rightarrow \mathbb{R}_+ \cup \infty$ is defined as:
\begin{equation}
\label{eq:OmegaT}
\Omega_{\theta}(\underline{X}) := \inf_{r_1,r_2 \in \mathbb{N}_+} \inf_{\substack{\Gamma \in \mathbb{R}^{G \times r_1}\\\Psi \in \mathbb{R}^{V \times r_2}\\\underline{C}\in \mathbb{R}^{r_1\times r_2\times T}}} \sum_{i=1}^{r_1}\sum_{j=1}^{r_2} \theta(\Gamma_i,\Psi_j,C_{i,j}) \ \ \textnormal{s.t.} \ \ \  \underline{C} \times_1 \Gamma \times_2 \Psi = \underline{X}.
\end{equation}
If the infimum is achieved for some $(\Gamma, \Psi, \u{C})$ and $r_1, r_2 \in \mathbb{N}_{+}$ then we say that $\u{C} \times_1 \Gamma \times_2 \Psi$ is an optimal factorization of $\u{X}$.
\end{definition}

\begin{proposition}
\label{appendix:propOmega}
Given regularizer $\theta$ that satisfies the properties of Definition~\ref{def:theta_T}, the tensor factorization regularizer $\Omega_{\theta}(\u{X})$ satisfies the following properties:
\begin{enumerate}
\item $\Omega_{\theta}(0) = 0$ and $\Omega_{\theta}(\u{X}) > 0 \ \forall \u{X}\neq0$.
\item $\Omega_{\theta}(\alpha \u{X}) = \alpha \Omega_{\theta}(\u{X}) \ \forall \alpha \geq 0 \ \forall \u{X}$.
\item $\Omega_{\theta}(\u{X}+\u{Y}) \leq \Omega_{\theta}(\u{X}) + \Omega_{\theta}(\u{Y}) \ \forall (\u{X},\u{Y})$.
\item If $\theta$ is symmetric about the origin in $\gamma$,$\psi$ or $c$, then $\Omega_{\theta}(-\u{X}) = \Omega_{\theta}(\u{X}) \ \forall \u{X}$ and $\Omega_\theta$ is a norm.
\item The infimum of $\Omega_{\theta}(\u{X})$ in \eqref{eq:OmegaT} can be achieved with finite $r_1$ and $r_2$, and $r_1,r_2 \leq G\times V \times T$.
\end{enumerate}
\end{proposition}
The proof of Proposition~\ref{appendix:propOmega} is provided in Appendix A.  By definition, satisfying the first three properties show that $\Omega_\theta$ is gauge function, and properties 2 and 3 show that $\Omega_\theta$ is convex. 
%Furthermore, for our choice of $\theta(\gamma,\psi,c)=||\gamma||_2||\psi||_2||c||_1$), $\Omega_{\theta}(\underline{X})$ is a norm. 
Then, with respect to $\u{X}$, we have the convex problem:
\begin{equation}
\label{eq:FT}
\min_{\u{X}} \{F(\underline{X}) \equiv \ell(\u{S}, \u{X})+ \lambda \Omega_{\theta}(\underline{X})\},
\end{equation}
where $F$ is a global lower bound for $f$.  The next theorem is an extension of Theorem~\ref{thm:ben} that relates global minimizers of the non-convex $f$ in \eqref{eq:fT} to the convex $F$ in \eqref{eq:FT}.
%\subsection{Global Optimality Results of Matrix Factorization}
%\label{sec:MatrixFactorization}
%To motivate the proposed separable dictionary learning framework, we first review results of \cite{Haeffelle:15,16} on general matrix factorizations and then extend their theorems to our specific case of sparse Kronecker (and tensor) decomposition in Section~\ref{sec:proposed}.
%\subsection{Proposed Separable Dictionary Learning with Global Optimality}
%\label{sec:proposed}
%In particular, \cite{Haeffelle:15,16} considers the problem 
%where $\mathcal{C}\equiv \{C_t\}_{t=1}^T$  and $\mathcal{X} \equiv \{X_t\}_{t=1}^T$ are the collections of representations. Because our theorems have been for a single training example (\ie $T=1$), the same results of Thoerem~\ref{thm:main} and Corollary~\ref{thm:corollary2} follow when we sum over $T>1$ training examples. Formally,
\begin{theorem}
\label{thm:mainT}
Given a function $\ell(\u{S},\u{X})$ that is convex and once differentiable w.r.t. $\u{X}$, a rank-1 regularizer $\theta$ that satisfies the conditions in Definition~\ref{def:theta_T}, and constants $r_1, r_2 \in \mathbb{N}_+$, and $\lambda >0$, any local minima $(\tilde{\Gamma},\tilde{\Psi},\tilde{\underline{C}})$ of $f(\Gamma,\Psi,\underline{C})$ in \eqref{eq:fT} is globally optimal if there exists $(i,j)$ such that $(\tilde{\Gamma}_{i},\tilde{\Psi}_{j}) = (0,0)$ and for all $t$, $(\tilde{C}_{i,t}, \tilde{C}_{j,t}) = (0,0)$. Moreover, $\hat{\underline{X}}=\tilde{\underline{C}} \times_1 \tilde{\Gamma} \times_2 \tilde{\Psi}$ is a global minimum of $F(\underline{X})$ in \eqref{eq:FT} and $\tilde{\underline{C}} \times_1 \tilde{\Gamma} \times_2 \tilde{\Psi}$ is an optimal factorization of $\hat{\underline{X}}$ in \eqref{eq:OmegaT}.
\end{theorem}
In order to prove Theorem~\ref{thm:mainT}, we first note that since $F(\u{X})$ is convex, $\hat{\u{X}}$ is a global minimum of $F(\u{X})$ if and only if $ 0 \in \partial F(\u{X})$, which is equivalent to $-\frac{1}{\lambda} \nabla_{\u{X}} \ell(\u{S},\hat{\u{X}}) \in \partial\Omega_{\theta}(\hat{\u{X}})$. Therefore, we must first characterize the subgradient $\partial\Omega_{\theta}(\u{X})$, which is the subject of the following lemma.
\begin{lemma}
\label{thm:lemma1T}
The subgradient $\partial \Omega_{\theta}(\underline{X})$ is given by:
\begin{equation}
\left \{\underline{W}: \langle \underline{W},\underline{X} \rangle = \Omega_{\theta}(\underline{X}) \ \ \text{and} \ \ \sum_{t=1}^T c_t\gamma^\top W_t \psi \leq \theta(\gamma,\psi,c) \ \forall (\gamma,\psi,c) \right \},
\end{equation}
for $\gamma \in \mathbb{R}^{r_1}, \psi \in \mathbb{R}^{r_2}, c \in \mathbb{R}^T, where \  \langle \underline{W},\underline{X} \rangle := \sum_t \langle W_t,X_t\rangle.$
\end{lemma}
%\marginnote{The new section start point will be pointed out by an arrow}[.2cm]
%\myparagraph{Proof of Lemma~\ref{thm:lemma1T}}
\begin{proof}
Since $\Omega_{\theta}$ is convex, by Fenchel duality, $\u{W}\in \partial\Omega_{\theta}$ if and only if $\langle \u{W},\u{X}\rangle = \Omega_{\theta}(\u{X}) + \Omega^*_{\theta}(\u{W})$, where $\Omega^*_{\theta}$ is the Fenchel dual of $\Omega_{\theta}$ given by $\Omega^*_{\theta}(\u{W}) \equiv \sup_{\u{Z}} \langle \u{W},\u{Z}\rangle - \Omega_{\theta}(\u{Z})$.  From the definition of $\Omega_{\theta}(\u{Z})$ we can expand the dual as
\begin{subequations}
\begin{align}
\Omega^*_{\theta}(\u{W}) 
%&= \sup_{r_1,r_2} \sup_{\Gamma,\Psi,\u{C}} \langle \u{W},\u{Z} \rangle - \sum_{i=1}^{r_1}\sum_{j=1}^{r_2} {\theta}(\Gamma_i,\Psi_j,C_{i,j}) \ \textnormal{s.t.} \ \u{C} \times_1 \Gamma \times_2 \Psi = \u{Z}\\
&= \sup_{Z} ~ \langle \u{W},\u{Z} \rangle  - \inf_{r_1,r_2} \inf_{\Gamma,\Psi,\u{C}}  ~ \sum_{i=1}^{r_1}\sum_{j=1}^{r_2} {\theta}(\Gamma_i,\Psi_j,C_{i,j}) \ ~ \textnormal{s.t.}~ \ \u{C} \times_1 \Gamma \times_2 \Psi = \u{Z}\\
&= \sup_{r_1,r_2} ~ \sup_{\Gamma,\Psi,\u{C}} ~ \langle \u{W},\u{C} \times_1 \Gamma \times_2 \Psi \rangle - \sum_{i=1}^{r_1}\sum_{j=1}^{r_2} {\theta}(\Gamma_i,\Psi_j,C_{i,j})\\
&= \sup_{r_1,r_2} ~ \sup_{\Gamma,\Psi,\u{C}} ~ \sum_{t=1}^{T} \langle \Gamma^\top W_t \Psi,C_t \rangle - \sum_{i=1}^{r_1}\sum_{j=1}^{r_2} {\theta}(\Gamma_i,\Psi_j,C_{i,j})\\ 
&= \sup_{r_1,r_2} ~ \sup_{\Gamma,\Psi,\u{C}} ~ \sum_{i=1}^{r_1} \sum_{j=1}^{r_2} \left(\sum_{t=1}^T c_{i,j,t}\Gamma_i^\top W_t \Psi_j -  {\theta}(\Gamma_i,\Psi_j,C_{i,j}) \right).
\label{eq:FenchelDual}
\end{align}
\end{subequations}
If there exists $(\gamma,\psi,c)$ such that $\sum_{t=1}^T c_t\gamma^\top W_t \psi > {\theta}(\gamma,\psi,c)$, we can see that $\Omega^*_{\theta}(\u{W}) = \infty$ by considering $(\alpha \gamma,\alpha \psi,\alpha c)$ as $\alpha \rightarrow \infty$ and using the positive homogeneity of $\theta$. 

Now let $\u{W}\in \partial\Omega_{\theta}(\underline{X})$. Then $\Omega^*_{\theta}(\u{W}) < +\infty$ and thus, from the previous argument, we have that $\sum_{t=1}^T c_t \gamma^\top W_t \psi \leq {\theta}(\gamma,\psi,c) \ \forall (\gamma,\psi,c)$. This also implies that all the terms 
%in the summation in 
inside the parenthesis of \eqref{eq:FenchelDual} will be non-positive, leaving the supremum to be $0$, which is achieved when $(\Gamma,\Psi,\u{C}) = (0,0,\u{0})$.  It follows that $\Omega^*_{\theta}(\u{W}) = 0$ and consequently $\langle \u{W},\u{X} \rangle = \Omega_{\theta}(\u{X})$.  

Conversely, if $\langle \u{W},\u{X} \rangle = \Omega_{\theta}(\u{X})$ and $\sum_{t=1}^T c_t\gamma^\top W_t \psi \leq {\theta}(\gamma,\psi,c) \ \forall (\gamma,\psi,c)$ then, reasoning as previously, we see that $\Omega^*_{\theta}(\u{W}) = 0$ which implies $\langle \u{W},\u{X} \rangle = \Omega_{\theta}(\u{X}) + \Omega^*_{\theta}(\u{W})$ and thus $\u{W} \in \partial \Omega_{\theta}(\u{X})$.
\end{proof}

Next, using the characterization of $\partial \Omega_\theta(\u{X})$ in Lemma~\ref{thm:lemma1T}, we identify when a factorization $\u{X}= \u{C}\times_1 \Gamma \times_2 \Psi$ is optimal, \ie when a point $(\Gamma,\Psi,\u{C})$ achieves the infimum of $\Omega_\theta(\u{X})$ in \eqref{eq:OmegaT}, in the following corollary.  

%These two conditions of the subgradient lead to the following Corollary which discusses when the infimum of $\Omega_{\theta}(\u{X})$ is achieved. In particular, we now know that :
\begin{corollary}
\label{thm:cor1T}
For factorization $\u{X} = \u{C} \times_1 \Gamma \times_2 \Psi$, if there exists $\u{W}$ such that $\langle \u{W},\u{X}\rangle = \Theta(\Gamma,\Psi,\u{C})$ and $\sum_{t=1}^T c_t\gamma^\top W_t \psi \leq \theta(\gamma,\psi,c) \ \forall (\gamma,\psi,c)$, then $\u{W} \in \partial \Omega_{\theta}(\u{X})$ and $\u{C} \times_1 \Gamma \times_2\Psi$ is an optimal factorization of $\u{X}$, \ie it achieves the infimum of $\Omega_{\theta}(\u{X})$.
\end{corollary}

%\myparagraph{Proof of Corollary~\ref{thm:cor1T}}
\begin{proof}
By contradiction, assume $\u{W} \notin \partial\Omega_{\theta}(\u{X})$. Then $\langle \u{W},\u{X} \rangle < \Omega_{\theta}(\u{X}) + \Omega^*_{\theta}(\u{W}) = \Omega_{\theta}(\u{X})$ because $\sum_{t=1}^T c_t \gamma^\top W_t \psi \leq \theta(\gamma,\psi,c) \ \forall (\gamma,\psi,c)$, implies $\Omega^*_{\theta}(\u{W}) = 0$ as in the proof of Lemma \ref{thm:lemma1T}.  Then, from our assumption, $\langle \u{W},\u{X} \rangle = \Theta(\Gamma,\Psi,\u{C}) = \sum_{i=1}^{r_1}\sum_{j=1}^{r_2} \theta(\Gamma_i,\Psi_j,C_{i,j}) < \Omega_{\theta}(\u{X})$ which violates the definition of $\Omega_{\theta}(\u{X})$ being the infimum, producing a contradiction. Therefore, $\u{W} \in \partial \Omega_{\theta}(\u{X})$. Now, since $\u{W} \in \partial \Omega_{\theta}(\u{X})$, by Lemma~\ref{thm:lemma1T}, $\langle \u{W},\u{X} \rangle = \Omega_{\theta}(\u{X})$, which implies  %Now, with the assumption made, we have $\sum_{i=1}^{r_1}\sum_{j=1}^{r_2} \theta(\Gamma_i,\Psi_j,C_{i,j}) = \sum_{i=1}^{r_1}\sum_{j=1}^{r_2} c_{i,j,t}\Gamma_i^\top W_t \Psi_j = \langle \u{W},\u{X} \rangle$ and since $\u{W}\in \partial \Omega_{\theta}(\u{X})$, $\langle \u{W},\u{X} \rangle = \Omega_{\theta}(\u{X})$ 
$\Theta(\Gamma,\Psi,\u{C}) = \sum_{i=1}^{r_1}\sum_{j=1}^{r_2} \theta(\Gamma_i,\Psi_j,C_{i,j})= \Omega_{\theta}(\u{X})$ thus showing that $\u{C} \times_1 \Gamma \times_2\Psi$ achieves the infimum of $\Omega_\theta(\u{X})$ and is an optimal factorization of $\u{X}$.
\end{proof}
Finally, with Lemma~\ref{thm:lemma1T} and Corollary~\ref{thm:cor1T} we can now prove Theorem~\ref{thm:mainT}:

\begin{proof}[Proof of Theorem~\ref{thm:mainT}]
%\begin{proof}
%\subsubsection{Theorem 1:} First need to show (22) (from Ben's paper)
%[-\frac{1}{\lambda}  \nabla_{\u{X}} \ell(S_t,\tilde{C}_t \times_1\tilde{\Gamma}\times_2 \tilde{\Psi})
%Because we define $\ell$ through the Frobenius norm on tensors $\ell(\u{S},\u{X}) = ||\u{S} - \u{X}||_F^2 = \sum_{t=1}^T||S_t - X_t||_F^2 = \sum_{t=1}^T \ell(S_t,X_t)$. Therefore $\nabla_{\u{X}} \ell(\u{S},\u{X}) = \sum_{t=1}^T \nabla_{X_t} \ell(S_t,X_t)$.

From \eqref{eq:FT}, we know $\hat{\u{X}}=\tilde{\u{C}} \times_1\tilde{\Gamma}\times_2 \tilde{\Psi}$ is a global minimum of $F(\u{X})$ if and only if $-\frac{1}{\lambda} \nabla_{\u{X}} \ell(\u{S},\hat{\u{X}}) \in \partial \Omega_{\theta}(\hat{\u{X}})$. Notice $-\frac{1}{\lambda} \nabla_{\u{X}} \ell(\u{S},\hat{\u{X}})$ can be written in terms of its slices, $\sum_{t=1}^T-\frac{1}{\lambda} \nabla_{X_t} \ell(S_t,\hat{X}_t)=\sum_{t=1}^T-\frac{1}{\lambda} \nabla_{X_t} \ell(S_t,\tilde{\Gamma} \tilde{C}_t \tilde{\Psi}^\top)$. To prove that $\hat{\underline{X}}=\tilde{\underline{C}} \times_1 \tilde{\Gamma} \times_2 \tilde{\Psi}$ is a global minimum and an optimal factorization of $\hat{\underline{X}}$, from Corollary~\ref{thm:cor1T}, it suffices to show two conditions:
\begin{enumerate}
\item 
$\sum_{i=1}^{r_1} \sum_{j=1}^{r_2}\sum_{t=1}^T \tilde{c}_{i,j,t}\tilde{\Gamma}^\top_i (-\frac{1}{\lambda} \nabla_{X_t} \ell(S_t,\tilde{\Gamma} \tilde{C}_t \tilde{\Psi}^\top))\tilde{\Psi}_j = \sum_{i=1}^{r_1} \sum_{j=1}^{r_2}  \theta(\tilde{\Gamma}_i,\tilde{\Psi}_j,\tilde{C}_{i,j})$, and
\label{cond1}
\item $\sum_{t=1}^T c_t\gamma^\top(-\frac{1}{\lambda} \nabla_{X_t} \ell(S_t,\tilde{\Gamma} \tilde{C}_t
\tilde{\Psi}^\top))\psi \leq \theta(\gamma,\psi,c) \ \forall (\gamma,\psi,c)$. \label{cond2}
\end{enumerate}
To show condition  \ref{cond1}, let $\Gamma_{1\pm \epsilon} = (1\pm\epsilon)^{1/3} \tilde{\Gamma}$ and $\Psi_{1\pm \epsilon} = (1\pm\epsilon)^{1/3} \tilde{\Psi}$ and $\u{C}_{1\pm \epsilon} = (1\pm\epsilon)^{1/3} \tilde{\u{C}}$. Since $(\tilde{\Gamma},\tilde{\Psi},\tilde{\u{C}})$ is a local minimum, there exists $\delta>0$ such that for all $\epsilon \in (0,\delta)$ we have
\begin{align}
&\sum_{t=1}^T\ell(S_t,\Gamma_{1\pm \epsilon}C_{t_{1\pm \epsilon}}\Psi^\top_{1\pm \epsilon}) + \lambda \sum_{i=1}^{r_1} \sum_{j=1}^{r_2} \theta((1\pm\epsilon)^{1/3} \tilde{\Gamma}_i,(1\pm\epsilon)^{1/3} \tilde{\Psi}_j,(1\pm\epsilon)^{1/3} \tilde{C}_{i,j}) \label{eq:balance1}\\
&=\sum_{t=1}^T\ell(S_t,(1\pm \epsilon)\tilde{\Gamma} \tilde{C}_t\tilde{\Psi}^\top) + \lambda(1\pm\epsilon)\sum_{i=1}^{r_1} \sum_{j=1}^{r_2}  \theta(\tilde{\Gamma}_i,\tilde{\Psi}_j,\tilde{C}_{i,j})\\
&\geq \sum_{t=1}^T\ell(S_t,\tilde{\Gamma}\tilde{C}_t\tilde{\Psi}^\top) + \lambda \sum_{i=1}^{r_1} \sum_{j=1}^{r_2}  \theta(\tilde{\Gamma}_i,\tilde{\Psi}_j,\tilde{C}_{i,j}).
\end{align}
Rearranging the last inequality gives
\begin{equation}
\frac{-1}{\lambda\epsilon} [\sum_{t=1}^T \ell(S_t,(1\pm \epsilon)\tilde{\Gamma}\tilde{C_t}\tilde{\Psi}^\top) -  \ell(S_t,\tilde{\Gamma}\tilde{C}_t\tilde{\Psi}^\top)] \leq \pm  \sum_{i=1}^{r_1} \sum_{j=1}^{r_2}  \theta(\tilde{\Gamma}_i,\tilde{\Psi}_j,\tilde{C}_{i,j}).
\end{equation}
Taking the limit as $\epsilon \searrow 0$ gives the directional derivative:
\begin{equation}
\label{eq:balance2}
\sum_{i=1}^{r_1} \sum_{j=1}^{r_2}  \theta(\tilde{\Gamma}_i,\tilde{\Psi}_j,\tilde{C}_{i,j}) \leq \sum_{t=1}^T \langle \frac{-1}{\lambda} \nabla_{X_t} \ell(S_t,\tilde{\Gamma}\tilde{C}_t\tilde{\Psi}^\top), \tilde{\Gamma}\tilde{C}_t\tilde{\Psi}^\top \rangle \leq \sum_{i=1}^{r_1} \sum_{j=1}^{r_2} \theta(\tilde{\Gamma}_i,\tilde{\Psi}_j,\tilde{C}_{i,j})
\end{equation}
which implies equality. Rearranging the inner product gives Condition~\ref{cond1}.

Next, to show condition 2, we use the assumption that there exists $(i,j)$ such that $(\tilde{\Gamma}_{i},\tilde{\Psi}_{j}) = (0,0)$ and for all $t$, $(\tilde{C}_{i,t}, \tilde{C}_{j,t}) = (0,0)$. Without loss of generality let the last column pair of $(\tilde{\Gamma},\tilde{\Psi})$ be zero and the last columns and rows of $\tilde{\u{C}}$ be zero for all $t$.  Then, given $(\gamma,\psi,c)$, let 
\begin{subequations}
\label{eq:descent_updates}
\begin{align}
\Gamma_\epsilon &= [\tilde{\Gamma}_1,\dots,\tilde{\Gamma}_{r_1-1}, \epsilon^{1/3} \gamma],\\ 
\Psi_\epsilon &= [\tilde{\Psi}_1,\dots,\tilde{\Psi}_{r_2-1}, \epsilon^{1/3} \psi], \text{and}\\
C_{t_\epsilon} &=
\begin{bmatrix}
 \tilde{\u{c}}_{1,1,t} & \dots & \tilde{\u{c}}_{1,r_2-1,t} & 0 \\
 \vdots & \ddots & \vdots  & \vdots \\
\tilde{\u{c}}_{r_1-1,1,t} & \dots &  \tilde{\u{c}}_{r_1-1,r_2-1,t} & 0 \\
0 & \dots & 0 & \epsilon^{1/3} c_t
\end{bmatrix} \ \ \forall t.
\end{align}
\end{subequations}
Now, since $\tilde{\u{C}}\times_1 \tilde{\Gamma} \times_2 \tilde{\Psi}$ is a local minimum of $F(\Gamma,\Psi,\u{C})$, there exists $\delta>0$ such that for all $\epsilon \in (0,\delta)$ we have
\begin{subequations}
\begin{align}
&\sum_{t=1}^T\ell(S_t,\Gamma_\epsilon C_{t_{\epsilon}} \Psi^\top_\epsilon) + \lambda \sum_{i=1}^{r_1} \sum_{j=1}^{r_2}\theta(\tilde{\Gamma}_i,\tilde{\Psi}_j,\tilde{C}_{i,j}) + \lambda \theta(\epsilon^{1/3}\gamma, \epsilon^{1/3}\psi,\epsilon^{1/3}c) = \\
&\sum_{t=1}^T\ell(S_t,\tilde{\Gamma}\tilde{C}_t\tilde{\Psi}^\top + \epsilon c_t \gamma \psi^\top) + \lambda \sum_{i=1}^{r_1} \sum_{j=1}^{r_2}\theta(\tilde{\Gamma}_i,\tilde{\Psi}_j,\tilde{C}_{i,j}) + \epsilon \lambda\theta(\gamma, \psi,c) \geq \\
&\sum_{t=1}^T\ell(S_t,\tilde{\Gamma}\tilde{C}_t\tilde{\Psi}^\top) + \lambda \sum_{i=1}^{r_1} \sum_{j=1}^{r_2}\theta(\tilde{\Gamma}_i,\tilde{\Psi}_j,\tilde{C}_{i,j}),
\end{align}
\end{subequations}
where the first equality follows from the positive homogeneity of $\theta$. Therefore, by rearranging the inequality we arrive at:
\begin{equation}
\frac{-1}{\lambda \epsilon} [ \sum_{t=1}^T \ell(S_t,\tilde{\Gamma}\tilde{C}_t\tilde{\Psi}^\top + \epsilon c_t \gamma \psi^\top) -  \ell(S_t,\tilde{\Gamma}\tilde{C}_t\tilde{\Psi}^\top) ] \leq \theta(\gamma,\psi,c)
\end{equation}
Since $\ell(S_t,X_t)$ is differentiable with respect to $X_t$, taking the limit as $\epsilon\searrow 0$, the directional derivative in the direction of $c_t\gamma\psi^\top$ gives us 
\begin{align}
\sum_{t=1}^T\langle \frac{-1}{\lambda} \nabla_{X_t} \ell(S_t,\tilde{\Gamma}\tilde{C}_t\tilde{\Psi}^\top ),c_t\gamma\psi^\top \rangle \leq \theta(\gamma,\psi,c)
\implies  &\sum_{t=1}^Tc_t\gamma^\top(\frac{-1}{\lambda}\nabla_{X_t} \ell(S_t,\tilde{\Gamma}\tilde{C}_t\tilde{\Psi}^\top )\psi \leq \theta(\gamma,\psi,c),
\end{align}
which proves Condition 2. This together with Condition 1 proves Theorem~\ref{thm:mainT}.
\end{proof}
In addition, Theorem \ref{thm:mainT} also shows the following immediate Corollary, which provides sufficient and necessary conditions for global optimality.
\begin{corollary}
\label{thm:corrNS}
Given a function $\ell(\u{S},\u{X})$ that is convex and once differentiable w.r.t. $\u{X}$, a rank-1 regularizer $\theta$ that satisfies the conditions in Definition~\ref{def:theta_T}, and constants $r_1, r_2 \in \mathbb{N}_+$, and $\lambda >0$, any point $(\tilde{\Gamma},\tilde{\Psi},\tilde{\underline{C}})$ is a global minimum of $f(\Gamma,\Psi,\underline{C})$ in \eqref{eq:fT} if it satisfies the following conditions:
\begin{enumerate}
    \item $\sum_{i=1}^{r_1} \sum_{j=1}^{r_2} \sum_{t=1}^T\tilde{c}_{i,j,t}\tilde{\Gamma}^\top_i (-\frac{1}{\lambda} \nabla_{X_t} \ell(S_t,\tilde{\Gamma} \tilde{C}_t \tilde{\Psi}^\top))\tilde{\Psi}_j =  \sum_{i=1}^{r_1} \sum_{j=1}^{r_2} \theta(\tilde{\Gamma}_i, \tilde{\Psi}_j, \tilde{C}_{i,j})$, and
    \item $\sum_{t=1}^T c_t \gamma^\top (-\frac{1}{\lambda} \nabla_{X_t} \ell(S_t,\tilde{\Gamma} \tilde{C}_t \tilde{\Psi}^\top)) \psi \leq \theta(\gamma,\psi,c) \ \forall(\gamma,\psi,c)$.
\end{enumerate}
Further, the first condition is always a necessary condition for global optimality, and if one additionally optimizes \eqref{eq:fT} over the number of dictionary atoms $(r_1,r_2)$, then both conditions are necessary conditions for global optimality.
\end{corollary}
\begin{proof}
The two conditions being sufficient is easily seen from Corollary \ref{thm:cor1T} and using identical arguments as in the beginning of the proof of Theorem \ref{thm:mainT}.  Note that by reversing the arguments from \eqref{eq:balance2} to \eqref{eq:balance1} one sees that if the first condition is not satisfied, then $(\tilde{\Gamma},\tilde{\Psi},\tilde{\underline{C}})$ is not a local minimum, so the first condition is always necessary for global optimality.  Finally, if $r_1$ and $r_2$ are also optimized over, then from the final property of $\Omega_\theta(\underline{X})$ in Proposition  \ref{appendix:propOmega} we have that the infimum can always be achieved with a finite value for $r_1$ and $r_2$.  As a result, this implies that we can always achieve the global minimum of the convex lower bound in \eqref{eq:FT}.  From this and Lemma \ref{thm:lemma1T} it is easily seen that the second condition is a necessary condition for global optimality of the convex function in \eqref{eq:FT}, which further implies that it is a necessary condition for global optimality of the non-convex function \eqref{eq:fT} if one additionally optimizes over $(r_1,r_2)$.
\end{proof}

The result of Theorem~\ref{thm:mainT} holds for any local minimum of $f$. Yet, in general, descent methods (\eg gradient descent) can only be guaranteed to converge to a stationary point at best (which may only be a saddle point) and therefore even arriving at a local minimum of $f$ may be challenging in practice. In the next section, we examine a choice of regularizer more specific to the dictionary learning problem for which we can eventually derive a more useful condition of global optimality for any $(\Gamma, \Psi, \u{C})$. %   that satisfies Definition~\ref{def:theta_T}, verifying the existence of $(i,j)$ such that $(\tilde{\Gamma}_{i},\tilde{\Psi}_{j}) = (0,0)$, and $(\tilde{C}_{i,t}, \tilde{C}_{j,t}) = (0,0)$
%for all $t$, may be difficult in practice. However, for our particular choice of $\theta$, we can derive a sufficient condition that is easier to verify.

\section{Global optimality for separable dictionary learning}
\label{sec:glob_opt_dict_learning}
We will now apply the previous analysis to the case of the rank-1 regularizer given by $\theta(\gamma,\psi,c) = ||\gamma||_2 ||\psi||_2 ||c||_1$, for which one easily verifies that the three conditions of Definition~\ref{def:theta_T} are satisfied.  Then, we have
\begin{equation}
\label{eq:ThetaT}
    \Theta(\Gamma,\Psi,\u{C}) = \sum_{i=1}^{r_1} \sum_{j=1}^{r_2} ||\Gamma_i||_2||\Psi_j||_2||C_{i,j}||_1.
\end{equation}
In that case, the tensor factorization problem of the previous section becomes:
\begin{equation}
\label{eq:SeparableDictionaryLearning_tensor_fact}
\min_{r_1,r_2,\Gamma,\Psi,\u{C}} \ell \left( \u{S}, \u{C}\times_1 \Gamma \times_2\Psi \right) + \lambda \sum_{i=1}^{r_1} \sum_{j=1}^{r_2} \|\Gamma_i\|_2 \|\Psi_j\|_2 \|C_{i,j}\|_1 .
\end{equation}
When $\ell\left(\u{S}, \u{C}\times_1 \Gamma \times_2\Psi \right) = \frac{1}{2} \|\u{C}\times_1 \Gamma \times_2\Psi - \u{S}\|_{F}^2$, this problem is simply an unconstrained reformulation of the separable dictionary learning problem of \eqref{eq:SeparableDictionaryLearning}. Yet, in contrast with state-of-the-art dictionary learning approaches, the results from the previous section will allow us to specify an explicit global optimality check for that problem. 

\subsection{Necessary and sufficient conditions for global optimality}
As a consequence of Theorem~\ref{thm:mainT} and Corollary~\ref{thm:corrNS}, for our particular choice of regularizer, the following characterization holds:
\begin{corollary}
\label{thm:corollary2T}
Let $\theta(\gamma,\psi,c) = ||\gamma||_2 ||\psi||_2 ||c||_1$. A point $(\tilde{\Gamma},\tilde{\Psi},\tilde{\u{C}})$ is a global minimum of %$f(\Gamma,\Psi,\u{C})$ in 
\eqref{eq:SeparableDictionaryLearning_tensor_fact} if and only if it satisfies the following conditions:
\begin{enumerate}
    \item $\sum_{t=1}^T\tilde{c}_{i,j,t}\tilde{\Gamma}^\top_i (-\frac{1}{\lambda} \nabla_{X_t} \ell(S_t,\tilde{\Gamma} \tilde{C}_t \tilde{\Psi}^\top))\tilde{\Psi}_j =  ||\tilde{\Gamma}_i||_2 ||\tilde{\Psi}_j||_2 ||\tilde{C}_{i,j}||_1 \ \forall \ (i,j)$, and
    \item $\max_{1\leq t \leq T }\sigma_{max}(-\frac{1}{\lambda} \nabla_{X_t} \ell(S_t,\tilde{\Gamma} \tilde{C}_t \tilde{\Psi}^\top)) \leq 1$,
    \label{eq:cor2}
\end{enumerate}
%\begin{equation}
%\label{eq:cor2}
%\max_{1\leq t \leq T }\sigma_{max}(-\frac{1}{\lambda} \nabla_{X_t} \ell(S_t,\tilde{\Gamma} \tilde{C}_t \tilde{\Psi}^\top)) \leq 1,
%\end{equation}
where $\sigma_{max}(Q)$ denotes the maximum singular value of matrix $Q$. Further, the first condition is always satisfied for any point which satisfies first-order optimality w.r.t. $\underline{C}$ in \eqref{eq:SeparableDictionaryLearning_tensor_fact} -- which by trivial extension includes all first-order optimal points of \eqref{eq:SeparableDictionaryLearning_tensor_fact}.
\end{corollary}

\begin{proof}
Because \eqref{eq:SeparableDictionaryLearning_tensor_fact} optimizes over the choice of $r_1$ and $r_2$, we know that the conditions in Corollary \ref{thm:corrNS} are both necessary and sufficient for global optimality, so we need to show that the two conditions given in the current statement are equivalent to those in Corollary \ref{thm:corrNS} for the particular choice of $\theta$ we make here. First, we know that to be a global minimum, a point must first satisfy first-order optimality for $f$.
%We begin by proving condition 1. is a result of first-order optimality given our choice of $\theta$. We extend the proof of Proposition 3 within \cite{Haeffele:PAMI2019} to the case of $\theta(\gamma,\psi,c) = ||\gamma||_2||\psi||_2||c||_1$ in the following Lemma:
%\begin{lemma}
%\label{appendix:propFirstOrder}
%Given a function $\ell(\u{S},\u{X})$ that is convex and once differentiable w.r.t. $\u{X}$, constant $\lambda>0$ and $\theta(\gamma,\psi,c) = ||\gamma||_2||\psi||_2||c||_1$, any first-order optimal point $(\tilde{\Gamma},\tilde{\Psi},\tilde{\u{C}})$ of $f(\Gamma,\Psi,\u{C})$ in \eqref{eq:fT} satisfies condition 1 of Corollary~\ref{thm:corollary2T}.
Noting that $\theta(\tilde{\Gamma}_i,\tilde{\Psi}_i,\tilde{C}_{i,j}) = \|\tilde{\Gamma}_i\|_2 \|\tilde{\Psi}_j\|_2 \sum_{t=1}^{T} |\tilde{c}_{i,j,t}|$ and writing the first-order optimality conditions on the coefficients $\tilde{c}_{i,j,t}$, we obtain that:
\begin{equation}
   % 0 &\in \sum_{j=1}^{r_2}\sum_{t=1}^T \tilde{c}_{i,j,t} \nabla_{X_t} \ell(S_t,\tilde{\Gamma} \tilde{C}_t \tilde{\Psi}^\top)\tilde{\Psi}_j + \lambda  \partial_\gamma \theta(\tilde{\Gamma}_i,\tilde{\Psi}_j,\tilde{C}_{i,j})\\
   % 0 &\in \sum_{i=1}^{r_1} \sum_{t=1}^T\tilde{c}_{i,j,t} \nabla_{X_t} \ell(S_t,\tilde{\Gamma} \tilde{C}_t \tilde{\Psi}^\top)^\top \tilde{\Gamma}_i + \lambda  \partial_\psi \theta(\tilde{\Gamma}_i,\tilde{\Psi}_j,\tilde{C}_{i,j})\\
    0 = \tilde{\Gamma}_i^\top \nabla_{X_t} \ell(S_t,\tilde{\Gamma} \tilde{C}_t \tilde{\Psi}^\top)\tilde{\Psi}_j + \lambda  \|\tilde{\Gamma}_i\|_2 \|\tilde{\Psi}_j\|_2 \ \text{sign}(\tilde{c}_{i,j,t})
\end{equation}
for all $i=1,\ldots,r_1$, $j=1,\ldots,r_2$ and $t=1,\ldots,T$ if $\tilde{c}_{i,j,t} \neq 0$. Multiplying by $\tilde{c}_{i,j,t}$ then leads, in all cases (including $\tilde{c}_{i,j,t}=0$), to:
\begin{equation}
    %0 &\in \sum_{t=1}^T \tilde{c}_{i,j,t} \tilde{\Gamma}_i^\top \nabla_{X_t} \ell(S_t,\tilde{\Gamma} \tilde{C}_t \tilde{\Psi}^\top)\tilde{\Psi}_j + \lambda  \langle \tilde{\Gamma}_i,\partial_\gamma \theta(\tilde{\Gamma}_i,\tilde{\Psi}_j,\tilde{C}_{i,j})\rangle \\
   % 0 &\in \sum_{t=1}^T\tilde{c}_{i,j,t} \tilde{\Psi}_j^\top \nabla_{X_t} \ell(S_t,\tilde{\Gamma} \tilde{C}_t \tilde{\Psi}^\top)^\top\tilde{\Gamma}_i + \lambda  \langle \tilde{\Psi}_j, \partial_\psi \theta(\tilde{\Gamma}_i,\tilde{\Psi}_j,\tilde{C}_{i,j}) \rangle\\
    0 = \tilde{c}_{i,j,t} \tilde{\Gamma}_i^\top \nabla_{X_t} \ell(S_t,\tilde{\Gamma} \tilde{C}_t \tilde{\Psi}^\top)\tilde{\Psi}_j + \lambda \|\tilde{\Gamma}_i\|_2 \|\tilde{\Psi}_j\|_2 |\tilde{c}_{i,j,t}| \ \ \forall(i,j,t).
\end{equation}
Now, summing over $t$ gives for all $i,j$:
\begin{equation}
\sum_{t=1}^T \tilde{c}_{i,j,t} \tilde{\Gamma}_i^\top (-\frac{1}{\lambda} \nabla_{X_t} \ell(S_t,\tilde{\Gamma} \tilde{C}_t \tilde{\Psi}^\top))\tilde{\Psi}_j = \|\tilde{\Gamma}_i\|_2 \|\tilde{\Psi}_j\|_2 \sum_{t=1}^{T} |\tilde{c}_{i,j,t}| = \|\tilde{\Gamma}_i\|_2 \|\tilde{\Psi}_j\|_2 \|\tilde{C}_{i,j}\|_1.
\end{equation}
Therefore, all stationary points w.r.t. $\underline{C}$ must satisfy the first condition of the current statement.  Additionally, summing over all $(i,j)$ shows that all stationary points must satisfy the first condition of Corollary \ref{thm:corrNS}.
%if a point satisfies condition 1 then it is a stationary point.

%\end{lemma}
%The proof of Lemma~\ref{appendix:propFirstOrder} is left to Appendix A. To be a global minimum, a point must be stationary and this result shows that any stationary point must satisfy condition 1.

%Next, from Theorem~\ref{thm:mainT}, we know that for a stationary point to be a global minimum we need to check that the following condition is satisfied:
Turning to the second condition of Corollary \ref{thm:corrNS}, we need to consider the following:
\begin{equation}
\label{eq:sup1}
\sum_{t=1}^T c_t\gamma^\top(-\frac{1}{\lambda} \nabla_{X_t} \ell(S_t,\tilde{\Gamma} \tilde{C}_t
\tilde{\Psi}^\top))\psi \leq \theta(\gamma,\psi,c) \ \forall (\gamma,\psi,c).
\end{equation}
For simplicity let $W_t := -\frac{1}{\lambda} \nabla_{X_t} \ell(S_t,\tilde{\Gamma} \tilde{C}_t \tilde{\Psi}^\top)$. With our choice of $\theta$, this condition becomes:
\begin{equation}
\label{eq:globalcriteria}
\sum_{t=1}^T c_t \gamma^\top W_t\psi \leq ||\gamma||_2 ||\psi||_2 ||c||_1 \ \ \forall (\gamma,\psi,c).
\end{equation}

Note that if either $\gamma=0$, $\psi=0$, or $c=0$ then the condition is trivially satisfied, so w.l.o.g. assume none of the vectors are all zero and normalize each variable by its respective norm, such that $\hat{\gamma} = \gamma/||\gamma||_2, \hat{\psi} = \psi/||\psi||_2, \hat{c}_t = c_t/||c||_1$. Then, the previous condition becomes
%\begin{equation}
%\sum_{t=1}^T \hat{c}_t\hat{\gamma}^\top W_t\hat{\psi} \leq 1 \ \ \forall %\hat{\gamma} \neq 0, \hat{\psi} \neq 0,\hat{c} \neq 0,
%\end{equation}
%which we can equivalently state as:
\begin{equation}
\label{eq:sup3}
\sup_{\substack{\hat \gamma,\hat \psi, \hat c \\ ||\hat{\gamma}||_2=||\hat{\psi}||_2=||\hat{c}||_1=1}}\sum_{t=1}^T \hat{c}_t \hat{\gamma}^\top W_t\hat{\psi} \leq 1.
\end{equation}
Now, maximizing with respect to $\hat{c}$, note that since $||\hat{c}||_1=1$, the supremum of a linear function can be attained by choosing $\hat{c}_{t_{\ast}}=1$ and $\hat{c}_t =0$ for $t \neq t_{\ast}$ where $t_\ast \in \argmax_{t}\{ \sup_{||\hat{\gamma}||_2=||\hat{\psi}||_2=1} \hat{\gamma}^\top W_t\hat{\psi}\}$. Therefore, \eqref{eq:sup3} is equivalent to:
\begin{equation}
\label{eq:descent_directions}
    \max_{1\leq t \leq T} \{\sup_{\substack{\hat \gamma, \hat \psi \\ ||\hat{\gamma}||_2=||\hat{\psi}||_2=1}} \hat{\gamma}^\top W_t\hat{\psi}\} \leq 1,
\end{equation}
%Because the supremum of a sum is less than or equal to the sum of suprema, we can instead check the stronger condition:
%\begin{equation}
%& \sup_{||\hat{c}||_1=1}\sum_{t=1}^T \hat{c}_t \sup_{||\hat{\gamma}||_2=||\hat{\psi}||_2=1}\hat{\gamma}^\top W_t \hat{\psi} \leq 1
%\end{equation}
Note that the inner supremum is equivalent to finding the maximum singular value of $W_t$, so with $\sigma_{max}$ denoting the largest singular value of the corresponding matrix, this is the same as:
%\begin{equation}
%\label{eq:max_SV}
%\sup_{||\hat{c}||_1=1}\sum_{t=1}^T \hat{c}_t \sigma_{max}(W_t) \leq 1
%\end{equation}
%Now the supremum in \eqref{eq:max_SV} is obviously attained by taking $\hat{c}_{t_{\ast}}=1$ and $c_t =0$ for $t \neq t_{\ast}$ with $t_\ast$ such that $\sigma_{max}(W_{t_\ast}) = \max_{t} \sigma_{max}(W_{t})$. In other words, the previous condition is satisfied if  
\begin{equation}
\label{eq:sup4}
\max_{1\leq t \leq T }\sigma_{\max}(W_t) \leq 1,
\end{equation}
which shows that condition \ref{eq:cor2} of the current statement is equivalent to the second condition in Corollary \ref{thm:corrNS}, completing the result.
%Thus, if a point satisfies conditions 1 and 2 then it is a global minimum of $f$. Conversely, a global minimum of $f$ satisfies first-order optimality which is equivalent to condition 1. In addition, the resulting tensor $\u{X}$ is then a minimum of $F$. It results that $\sum_{t=1}^T c_t\gamma^\top(-\frac{1}{\lambda} \nabla_{X_t} \ell(S_t,\tilde{\Gamma} \tilde{C}_t \tilde{\Psi}^\top))\psi \leq \theta(\gamma,\psi,c)$ for all $(\gamma,\psi,c)$ which, by the previous reasoning, implies condition 2.
\end{proof}

Using the results of Corollary~\ref{thm:corollary2T}, we can devise an algorithm to find a global minimum of the separable dictionary learning problem by first finding a stationary point of \eqref{eq:SeparableDictionaryLearning} and then checking if it satisfies condition \ref{eq:cor2} in Corollary~\ref{thm:corollary2T}. A logical next question of this routine is what happens if the stationary point does not satisfy condition \ref{eq:cor2}.  In \cite{Haeffele:PAMI2019}, the authors demonstrate that for the classical dictionary learning problem, any non-optimal stationary point can be escaped by adding a column to the dictionary $D$ and a row to the coefficient matrix $W$ where the column/row appended to $D$/$W$ is chosen to be an example that violates the form of condition \ref{eq:cor2} that arises in the matrix factorization setting. %{\color{blue}by appending a column of zeros to the dictionary $\tilde D$ and a row of zeros to the coefficients $\tilde W$, they are guaranteed to continue in a descent direction.}
This is also shown here for the separable dictionary learning case.  The part of the proof that gives \eqref{eq:sup3} is essentially this argument. In other words, if \eqref{eq:sup3} is not satisfied then the $\gamma,\psi,c$ achieving the maximum in \eqref{eq:sup3} will provide a descent direction by appending them to $\tilde\Gamma$, $\tilde\Psi$ and $C$, as in \eqref{eq:descent_updates}. Therefore, the algorithm will consist of iterating between local descent and global optimality check, appending new dictionary atoms if necessary.

This approach incidentally allows to learn the dictionary size throughout the process. For separable dictionaries, we have in fact two size parameters $r_1$ and $r_2$. Therefore, in this case, one has the additional option to augment one or both of the dictionaries at the end of a local descent. Based on the application or preference of the relative dictionary sizes, we have the opportunity to schedule the increments of $r_1$ and $r_2$. %{\color{red} This is somewhat just semantics, but I'm not sure we should call this the 'a priori' size that we know.  Really, it's just that the optimal size is not unique, so we'd prefer to find solutions with small $r_1 / r_2$.} 
We will formalize and study more closely such an algorithm in Section \ref{sec:algorithm}.

\subsection{Connection with low-rank tensor decomposition}
\label{ssec:connection_low_rank}
Before we delve into the algorithmic side of the proposed dictionary learning approach, there are a few more important remarks to be made on the optimization problem \eqref{eq:SeparableDictionaryLearning_tensor_fact}. In particular, we give here an alternative interpretation of this problem in terms of low-rank tensor decomposition and incidentally show a better bound for the number of dictionary elements of some global optima than the general one of Proposition \ref{appendix:propOmega}.

First, we have the following statement showing that $\Omega_\theta$ corresponds to the summation of the nuclear norms of each t-slice: 
\begin{proposition}
\label{prop:polar_fun_nucl_norm}
With $\theta(\gamma,\psi,c) = ||\gamma||_2 ||\psi||_2 ||c||_1$, we have:
\begin{equation}
    \Omega_{\theta}(\u{X}) = \sum_{t=1}^{T} \|X_t\|_{*}.
\end{equation}
\end{proposition}
\begin{proof}
First, with this choice of $\theta$, Proposition \ref{appendix:propOmega} shows that $\Omega_{\theta}$ is a norm on $\R^{G \times V \times T}$. We can thus consider the dual norm, $\mathring{\Omega}_{\theta}$, defined as $\mathring{\Omega}_{\theta}(\u{X}) = \sup_{\Omega_{\theta}(\u{W})\leq 1} \langle \u{W}, \u{X} \rangle_F$. Now, for any $\u{W} \in \R^{G\times V \times T}$ such that $\Omega_{\theta}(\u{W})\leq 1$, using again Proposition \ref{appendix:propOmega}, we can write $\u{W} = \u{C} \times_1 \Gamma \times_2\Psi$ with $\Theta(\Gamma,\Psi,\u{C}) = \Omega_{\theta}(\u{W})\leq 1$ and:
\begin{equation}
    \langle \u{W} , \u{X} \rangle_F = \sum_{t=1}^{T} \sum_{i=1}^{r_1} \sum_{j=1}^{r_2} c_{i,j,t} \Gamma_i^T X_t \Psi_j,
\end{equation}
which shows that: 
\begin{equation}
   \mathring{\Omega}_{\theta}(\u{X}) = \sup_{\Gamma,\Psi,\u{C}}  \sum_{t=1}^{T} \sum_{i=1}^{r_1} \sum_{j=1}^{r_2} c_{i,j,t} \Gamma_i^T X_t \Psi_j \ \ s.t \ \ \Theta(\Gamma,\Psi,\u{C}) \leq 1.
\end{equation}
Defining $\Upsilon = \sup_{\gamma,\psi,c} \sum_{t=1}^{T} c_t \gamma^T X_t \psi$ s.t $\theta(\gamma,\psi,c) \leq 1$, it is then clear that $\Upsilon \leq \mathring{\Omega}_{\theta}(\u{X})$. Conversely, for any $\Gamma,\Psi,\u{C}$ s.t. $\Theta(\Gamma,\Psi,\u{C}) \leq 1$:
\begin{align}
    \sum_{t=1}^{T} \sum_{i=1}^{r_1} \sum_{j=1}^{r_2} c_{i,j,t} \Gamma_i^T X_t \Psi_j &= \sum_{i=1}^{r_1} \sum_{j=1}^{r_2} \sum_{t=1}^{T} c_{i,j,t} \Gamma_i^T X_t \Psi_j 
    \leq \sum_{i=1}^{r_1} \sum_{j=1}^{r_2} \theta(\Gamma_i,\Psi_j,C_{i,j}) \Upsilon 
    \leq \Theta(\Gamma,\Psi,\u{C}) \Upsilon \leq \Upsilon,
\end{align}
which leads to
\begin{equation}
    \mathring{\Omega}_{\theta}(\u{X}) = \sup_{\gamma,\psi,c} \sum_{t=1}^{T} c_t \gamma^T W_t \psi \ \ \text{s.t.} \ \ \theta(\gamma,\psi,c)\leq 1 .
\end{equation}
Following the same reasoning as in the proof of Corollary \ref{thm:corollary2T}, we obtain that:
\begin{equation}
    \mathring{\Omega}_{\theta}(\u{X}) = \max_{t=1,\ldots,T} \sigma_{max}(W_t) .
\end{equation}
We have now by biduality $\Omega_{\theta}(\u{X}) = \sup_{\mathring{\Omega}_{\theta}(\u{W}) \leq 1} \langle \u{X} , \u{W} \rangle_F$. For any $\u{W}$ such that $\mathring{\Omega}_{\theta}(\u{W})\leq 1$, i.e. $\sigma_{max}(W_t) \leq 1$ for all $t$, since $\langle \u{X} , \u{W} \rangle_F = \sum_{t=1}^{T} \langle X_t,W_t \rangle_F$, it follows from the standard expression of the nuclear norm of matrices and its dual that:
\begin{equation}
    \Omega_{\theta}(\u{X}) = \sum_{t=1}^{T} \|X_t\|_{*} = \sum_{t=1}^{T} \sum_{i=1}^{r} \sigma_i(X_t),
\end{equation}
where $r=\min(V,G)$ and $\sigma_i(X_t)$ are all the singular values of the slice $X_t$.
\end{proof}

Based on this alternative expression of the regularizer, we can rewrite the convex problem over $\u{X}$ of \eqref{eq:FT} as a slice by slice low rank regularization of the signal: 
\begin{equation}
    \min_{\u{X}} \sum_{t=1}^{T} \ell(S_t,X_t) + \lambda \|X_t\|_{*}.
\end{equation}
Thus, we have just shown that solutions of the dictionary learning problem with respect to the tensor $\u{X}$ in \eqref{eq:SeparableDictionaryLearning} are essentially obtained by solving a set of low rank matrix approximation problems. This type of problem has been studied in many past works for instance on matrix completion \cite{Candes-Recht:FCM09,Cai:SJO10}. %{\color{red} Here it might also be good to reference (1.5) since that is the problem people will likely be most familiar with instead of the somewhat different notation that we use.} 

In the rest of this section, let's consider the simple particular fidelity term given by $\ell(S_t,X_t) = \frac{1}{2} \|S_t-X_t\|_F^2$. It is well-known that each optimal $X_t$ is then given by the singular value shrinkage operator applied to $S_t$, i.e. (with the notations of \cite{Cai:SJO10}) for all $t=1,\ldots,T,$ we have $ X_t = \mathcal{D}_{\lambda}(S_t)$, where $\mathcal{D}_\lambda(Y) = U \mathcal{D}_\lambda(\Sigma) V^\top$ with $U\Sigma V^\top$ being the SVD of $Y$ and $\mathcal{D}_\lambda(\Sigma) = \text{Diag}([\sigma_i-\lambda]_+)$. This does not yet characterize solutions for the dictionary learning problem itself as one still needs to obtain an optimal factorization for $\u{X}$ for the regularizer $\Omega_{\theta}$. However, in that particular case, ``explicit" global minima of \eqref{eq:SeparableDictionaryLearning_tensor_fact} can be in fact constructed as follows. Writing the SVD of $S_t$ as $S_t = U_t \Sigma_t V_t^\top$, where $U_t \in O(G)$, $V_t \in O(V)$ and $\Sigma_t = \text{Diag}(\sigma_{t,1},\ldots,\sigma_{t,m})$ with $m=\min\{G,V\}$, one can set $r_1=GT, r_2 = VT$, $\Gamma = [U_1,\ldots,U_T]$, $\Psi = [V_1,\ldots V_T]$, and $\u{C} \in \R^{GT \times VT \times T}$ such that 
\begin{equation}
    C_t = \begin{bmatrix}
    \boxed{0} &        &                               &        &   \\
      & \ddots &                               &        &   \\
      &        & \boxed{\mathcal{D}_\lambda(\Sigma_t)} &        &   \\
      &        &                               & \ddots &   \\
      &        &                               &        & \boxed{0}  
    \end{bmatrix}, \qquad t=1,\dots, T.
\end{equation}
Then, by construction, $\u{X} = \u{C} \times_1 \Gamma \times_2 \Psi$ and it is a simple verification that $(\Gamma,\Psi,\u{C})$ satisfy the two optimality conditions of Corollary \ref{thm:corollary2T}. Consequently, for that particular choice of $\theta$ and $\ell$, we see that global minima in \eqref{eq:SeparableDictionaryLearning_tensor_fact} exist with $r_1 \leq GT$ and $r_2 \leq VT$. In addition, solutions can be computed based on the SVDs of the slices $S_t$ as we just described. Note also that, while $r_1=GT$ and $r_2=VT$ are upper bounds on the dictionary sizes universal to all the signals $S_t$, one can improve the compactness of the dictionaries obtained by this SVD approach in each specific case by discarding atoms which are associated to the zero diagonal elements of the shrinked matrices $\mathcal{D}_\lambda(\Sigma_t)$ in the above SVD decomposition of $S_t$. Namely, we can simply restrict to the columns $i$ of $U_t$ and $V_t$ such that $\mathcal{D}_\lambda(\Sigma_t)_{i,i} \neq 0$. Since the number of those columns is exactly the rank of $\mathcal{D}_\lambda(S_t)$, this would eventually lead to dictionaries $\tilde{\Gamma}$ and $\tilde{\Psi}$ both with $\tilde{r} = \sum_{t=1}^{T} \text{rank}(\mathcal{D}_\lambda(S_t))$ atoms.     

There are however several remaining limitations to this approach for solving the dictionary learning problem. From a numerical point of view, computing that many complete SVDs of such potentially large matrices can prove very intensive for practical applications. More importantly, although the resulting sizes of the dictionaries $\Gamma$ and $\Psi$ are smaller than the upper bound of Proposition \ref{appendix:propOmega}, those are still constructed by direct concatenation of one dictionary for each slice while enforcing a diagonality constraint on the $C_t$'s. From the perspective of dictionary learning, one is typically interested in more compact representations with a total number of atoms in $\Gamma$ and $\Psi$ that is on the order of the size of the data (i.e. with $r_1 \sim G, \ r_2 \sim V$) and independent of the number $T$ of training samples. In the following sections, we will show empirically that much more compact solutions can be found by instead introducing a more efficient algorithm that iteratively increases $r_1$ and $r_2$ until global optimality conditions are satisfied.            

\section{Algorithm for finding global minimum}
\label{sec:algorithm}
Now that Corollary \ref{thm:corollary2T} provides practical conditions to guarantee global minimality of the separable dictionary learning problem, we will outline an algorithm to reach a globally optimal solution. This involves alternating between two main sub-routines: 1) local descent to reach a stationary point with fixed number of atoms $r_1$ and $r_2$ in the dictionaries, and 2) a check for global optimality via Corollary \ref{thm:corollary2T}. Note that since we consider the particular choice of regularizer $\theta(\gamma,\psi,c) = \|\gamma\|_2 \|\psi\|_2 \|c\|_1$, the global optimality check only amounts to verifying that a stationary point satisfies condition    \ref{eq:cor2} in Corollary \ref{thm:corollary2T}. If by the end of the local descent we have not reached a globally optimal solution, then we can find a global descent direction by adding additional atoms to the dictionaries. Algorithm~\ref{alg:meta} describes this general meta-algorithm in more detail and refers to each sub-routine discussed in the following sections.
 \begin{algorithm}
     \caption{Meta-Algorithm: Local Descent and Global Optimality Check}
     \begin{algorithmic}
 \label{alg:meta}
 \STATE Initialize dictionaries with set number of atoms.
 \WHILE{not globally optimal}
 \WHILE{objective residual $> \epsilon$}
 \STATE descent to local minimum via Algorithm~\ref{alg:prox}
 \ENDWHILE
 \IF{Condition \ref{eq:cor2} is satisfied}
 \STATE solution is globally optimal 
 \ELSE
\STATE update dictionaries via Algorithm~\ref{alg:globalalgorithm} 
 \ENDIF
 \ENDWHILE
     \end{algorithmic}
   \end{algorithm}

\subsection{Proximal gradient descent to stationary point}
In this section, we provide an algorithm to find a stationary point of the separable dictionary learning problem with fixed sizes for the dictionaries. We again state the problem:
%We again state the separable dictionary learning problem we wish to solve:
\begin{equation}
\label{eq:SepDictionaryLearningGlobal}
\min_{\Gamma,\Psi,\u{C}} \frac{1}{2}\sum_{t=1}^T||\Gamma C_t\Psi^\top - S_t||_F^2 + \lambda \sum_{i=1}^{r_1} \sum_{j=1}^{r_2} ||\Gamma_i||_2||\Psi_j||_2||C_{i,j}||_1.
\end{equation}
For an optimization problem of the form $\min_x \{\ell(x) + \lambda\Theta(x)\}$, where $\ell$ is differentiable and $\Theta$ is non-differentiable, proximal gradient descent \cite{Parikh2014} is a common algorithm to arrive at a stationary points, \ie local minima or saddle points. The general updates for proximal gradient descent follow:
\begin{equation}
    x^{k+1} = \textnormal{prox}_{\tau \lambda\Theta(\cdot)}(x^k - \tau \nabla \ell),
\end{equation}
where $\textnormal{prox}_{\tau \lambda\Theta(\cdot)}(y) = \argmin_x \{ \frac{1}{2\tau\lambda}||x-y||_2^2 + \Theta(x) \}$.  
%We have developed this notion for Kronecker structured problems in our two previous sections with the Kronecker Fast Iterative Shrinkage Algorithm (Kron-FISTA) used for solving $\ell^1$ minimization with $\theta(x) = ||x||_1$.  
%Here we must derive the proximal operator for our chosen regularizer $\Theta$. In addition, we now have three variables to optimize over. 
To solve \eqref{eq:SepDictionaryLearningGlobal}, we apply a  proximal gradient descent step to each variable while holding the remaining ones constant. This local descent procedure is outlined in Algorithm~\ref{alg:prox}. Recall that $\ell(\Gamma,\Psi,\u{C}) = \frac{1}{2}\sum_{t=1}^T||\Gamma C_t\Psi^\top -S_t||_F^2$. We derive the update for each variable as:
\begin{align}
    \Gamma^{k+1}_i &= \textnormal{prox}_{\xi_i^k||\cdot||_2}( \Gamma^{k}_i - \xi_i^k [\nabla_{\Gamma^k} \ell]_i)\label{eq:updateGamma}\\
    c_{i,j,t}^{k+1} &= \textnormal{prox}_{\kappa_{i,j}^k|\cdot|}( c_{i,j,t}^{k} - \kappa_{i,j}^k [\nabla_{C_t^k} \ell]_{i,j}) \label{eq:updateCt}\\
       \Psi^{k+1}_j &= \textnormal{prox}_{\pi_j^k||\cdot||_2}( \Psi^{k}_j - \pi_j^k [\nabla_{\Psi^k} \ell]_j).\label{eq:updatePsi}
\end{align}
where the proximal operators for $||\cdot||_2$ and $|\cdot|$ can be written in closed form as:
\begin{align}
%\label{eq:prox}
\textnormal{prox}_{\tau||\cdot||_2}(x) &= \left\{\begin{array}{lr}
        (1 - \frac{\tau}{||x||_2})x & \text{for } ||x||_2\geq \tau\\
        0 & \text{otherwise}
        \end{array},\right.\\
  \textnormal{prox}_{\tau|\cdot|}(\alpha) &= \textnormal{max}(0, \alpha -\tau) - \textnormal{max}(0,-\alpha - \tau),
\end{align}
for $x \in \mathbb{R}^N$, $\alpha \in \mathbb{R}$ and $\tau \geq 0$, and
\begin{align}
\nabla_\Gamma \ell  &= \sum_{t=1}^T(\Gamma C_t \Psi^\top - S_t) \Psi C_t^\top\\
\nabla_{C_t} \ell &= \Gamma^\top(\Gamma C_t \Psi^\top - S_t) \Psi\\
\nabla_\Psi \ell &= \sum_{t=1}^T(\Psi C_t^\top \Gamma^\top - S_t^\top) \Gamma C_t.
\end{align}
Finally, $\xi_i, \kappa_{i,j}$, and $\pi_j$ are constants composed of the other fixed variables in $\theta$, specifically
\begin{align}
\xi_i &:= \lambda \sum_{t=1}^T\sum_{j=1}^{r_2} |c_{i,j,t} | ||\Psi_j||_2/L_\Gamma\\
\kappa_{i,j} &:= \lambda||\Gamma_i||_2 ||\Psi_j||_2/L_{C_t}\\
\pi_j &:= \lambda\sum_{t=1}^T\sum_{i=1}^{r_1} |c_{i,j,t} | ||\Gamma_i||_2/L_\Psi,
\end{align}
where the parameters $1/L_\Gamma, 1/L_{C_t}$, and $1/L_\Psi$ correspond to the step sizes in the proximal gradient descent.
In general, to determine an appropriate step-size $\tau$, it has been shown that convergence is guaranteed if $\tau \leq \frac{1}{L}$, where $L$ is the Lipschitz constant of $\nabla l$:
\begin{equation}
    ||\nabla \ell(x^{(1)}) - \nabla \ell(x^{(2)})||_2 \leq L ||x^{(1)}-x^{(2)}||_2.
\end{equation}
In our setting, we can calculate (or at least bound) the Lipschitz constants with respect to, $L_\Gamma, L_{\u{C}_{t}}$, and $L_\Psi$. For example, for $L_\Gamma$ we have:
\begin{align}
|| \nabla_\Gamma \ell(\Gamma^{(1)}) - \nabla_\Gamma \ell(\Gamma^{(2)})||_F &=  ||\sum_{t=1}^T(\Gamma^{(1)} C_t \Psi^\top - S_t) \Psi C_t^\top -(\Gamma^{(2)} C_t \Psi^\top - S_t) \Psi C_t^\top||_F\\
&=  ||\sum_{t=1}^T\Gamma^{(1)} C_t \Psi^\top\Psi C_t^\top - \Gamma^{(2)} C_t \Psi^\top\Psi C_t^\top||_F\\
&= ||\sum_{t=1}^T C_t \Psi^\top\Psi C_t^\top (\Gamma^{(1)} - \Gamma^{(2)})||_F\\
 &\leq||\sum_{t=1}^T C_t \Psi^\top\Psi C_t^\top ||_F||(\Gamma^{(1)} - \Gamma^{(2)})||_F\\ 
 &= L_\Gamma||(\Gamma^{(1)} - \Gamma^{(2)})||_F 
\end{align}
where $L_\Gamma = ||\sum_{t=1}^T C_t \Psi^\top\Psi C_t^\top) ||_F$ is thus an upper bound for the Lipschitz constant of $\nabla_\Gamma \ell$.
Similarly for $\nabla_\Psi \ell$, we can take as the Lipschitz constant $L_\Psi = ||\sum_{t=1}^T C_t \Gamma^\top\Gamma C_t^\top||_F$. Then for $\nabla_{\underline{C}_t} \ell$, $L_{\underline{C}_t} = ||\Gamma^\top \Gamma||_F ||\Psi^\top \Psi||_F.$ Lastly, the convergence of the descent can be accelerated through the standard Nesterov scheme described as an extension of the Proximal Gradient Descent in Algorithm~\ref{appendix:nesterovDL} in Appendix B.
\begin{algorithm}
  \caption{Proximal Gradient Descent}
  \begin{algorithmic}
  \label{alg:prox}
  \STATE Initialize: $k=0,\Gamma^0, \Psi^0, \u{C}^0, \lambda, r_1, r_2.$
  \WHILE{error $> \epsilon$}
  \STATE Update $\Gamma^k$ via \eqref{eq:updateGamma}
  \STATE Update $\u{C}^k$ via \eqref{eq:updateCt}
  \STATE Update $\Psi^k$ via \eqref{eq:updatePsi}
%  \STATE Nesterov Acceleration via Algorithm~\ref{alg:nesterovDL}
  \STATE $k \rightarrow k+1$
  \ENDWHILE
  \RETURN stationary point $(\tilde{\Gamma},\tilde{\Psi},\tilde{\u{C}})$
  \end{algorithmic}
  \end{algorithm}

\subsection{Global optimality check}
Once proximal gradient descent reaches a stationary point $(\tilde{\Gamma},\tilde{\Psi},\tilde{\u{C}})$ via Algorithm~\ref{alg:prox}, we need to check if the solution is a global minimum. By the result of Corollary~\ref{thm:corollary2T}, since the first condition is always satisfied by a stationary point, we just needs to check if the second condition holds. If so, we have reached a global minimum and the algorithm stops. If not, the optimal solution $(\gamma,\psi,c)$ to the optimization problem in \eqref{eq:sup3} provides us a descent direction. Specifically, if we augment the current dictionary and coefficients $(\tilde{\Gamma},\tilde{\Psi},\tilde{\u{C}})$ by the new atoms and coefficients $(\gamma,\psi,c)$ by following the update in \eqref{eq:descent_updates} with $\epsilon$ sufficiently small, then the violation of the second condition of the corollary guarantees that the objective value at the augmented variables $(\Gamma,\Psi,\u{C})$ will be lower. The precise update of the variables is obtained as follows.
%by adding additional atoms to the dictionaries, we can escape from the local minimum or saddle point we have reached and search for a global descent direction. Following the discussions in \cite{Haeffele:PAMI2019} for matrix factorization involving the nuclear norm, by augmenting the locally optimal dictionaries with the singular vectors of the maximum singular value of condition \ref{eq:cor2}, we are guaranteed to move in a global descent direction. 
First, let $t_\ast = \argmax_t \sigma_{\textnormal{max}}(W_t)$ where $W_t := -\frac{1}{\lambda} \nabla_{X_t} \ell(S_t,\tilde{\Gamma} \tilde{C}_t \tilde{\Psi}^\top)$. Then with $(\gamma_{t_\ast},\psi_{t_\ast})$ the left and right singular vector pair corresponding to the maximum singular value of $W_t$ over all $t$, we can update the locally optimal dictionaries $\tilde{\Gamma}$ and $\tilde{\Psi}$ by appending the last column $\Gamma = [\tilde{\Gamma}, \gamma_{t_\ast}]$ and $\Psi = [\tilde{\Psi}, \psi_{t_\ast}]$.  Finally, $\u{C}$ can be updated by appending the slice corresponding to the maximum singular value by 
\begin{equation}
C_{t_\ast} = 
\begin{bmatrix}\tilde{C}_{t_\ast} & 0\\
 0 & \tau 
 \end{bmatrix}
\end{equation}
and appending a matrix of zeros for all other slices $C_t$ for $t\neq t_*$. By comparing these updates with \eqref{eq:descent_directions}, we can see that the main difference is that here we have chosen $\gamma_{t_*}$ and $\psi_{t_*}$ to be unit norm, we have chosen $c_{t_*} = 1$, and we have neglected the factor $\epsilon^{1/3}$. Instead, due to the homogeneity of the product $\Gamma_t C_t \Psi_t^\top$, we have absorbed all the scaling factors into a single variable $\tau$, which is proportional to $\epsilon$ and thus must be chosen small enough so that the objective function decreases. As a consequence, $\tau$ can be thought as a step-size for the coefficient in slice $t_*$ associated to the new atom $\gamma_{t_*} \otimes \psi_{t_*}$. The joint update gives rise to
\begin{equation}
\Gamma_{t^*} C_t \Psi_{t^*}^\top = \tilde{\Gamma}_{t^*} \tilde{C}_{t^*} \tilde{\Psi}_{t^*}^\top + \tau E_{t^*},
\end{equation}
where $\tau E_{t^*} = \tau \gamma_{t_*} \psi_{t_*}^\top$ is a descent direction for $\tau$ small enough. In our implementation, we select the optimal $\tau^*$ that leads to the largest decrease of the energy when all dictionary atoms and all other coefficients are fixed, which can be found by solving:
\begin{equation}
\label{eq:globalstep}
%\tau^* = \argmin_{\tau} \frac{1}{2}\sum_t^T||S_t - \hat{X}_t - \tau E_t||_F^2 + \lambda|\tau|,
\tau^* = \argmin_{\tau} \frac{1}{2}||S_{t_*} - \hat{X}_{t_*} - \tau E_{t_*}||_F^2 + \lambda|\tau|,
\end{equation}
where $E_{t_\ast} = \gamma_{t_\ast}\psi_{t_\ast}^\top$. By vectorizing all tensors, \eqref{eq:globalstep} reduces to the simple proximal operator of the absolute value function given in closed-form by the soft-thresholding operator.

Now, because of the separable form of this problem, we actually have the option to update just one of the two dictionaries, and not both simultaneously during each global check. In particular, if $\gamma_{t_\ast} \in \textnormal{Span}(\tilde{\Gamma})$ then it is unnecessary to add this atom to the dictionary. Likewise for $\psi_{t_\ast}\in \textnormal{Span}(\tilde{\Psi})$. Instead of checking these conditions \textit{a posteriori}, we can check criteria akin to \eqref{eq:globalcriteria} with the added constraint that $\gamma \in \textnormal{Span}(\tilde{\Gamma})$. By definition, $\gamma \in \textnormal{Span}(\tilde{\Gamma})$ means that there exists an $\alpha$ such that $\gamma = \tilde{\Gamma}\alpha$. Using this we can make a change of variable in \eqref{eq:globalcriteria} as:
\begin{equation}
\label{eq:globalCheckGamma}
\sum_{t=1}^T c_t \alpha^\top \tilde{\Gamma}^\top W_t\psi \leq ||\tilde{\Gamma}\alpha||_2 ||\psi||_2 ||c||_1 \ \ \forall (\alpha,\psi,c).
\end{equation}
By noting that $||\tilde{\Gamma}\alpha||_2 \leq ||\tilde{\Gamma}||_2 ||\alpha||_2 = \sigma_{\max}(\tilde{\Gamma}) ||\alpha||_2$, we can check the looser criterion:
\begin{equation}
\label{eq:globalCheckGammaF}
\sum_{t=1}^T c_t \alpha^\top \tilde{\Gamma}^\top W_t\psi \leq \sigma_{\max}(\tilde{\Gamma})||\alpha||_2 ||\psi||_2 ||c||_1 \ \ \forall (\alpha,\psi,c).
\end{equation}
Therefore, if \eqref{eq:globalCheckGammaF} is violated then so is \eqref{eq:globalCheckGamma}. We prefer to check \eqref{eq:globalCheckGammaF} because of its simplicity to compute, which can be done as follows.
As before, normalizing 
$\hat{\psi} = \psi/||\psi||_2$, $\hat{c}_t = c_t/||c||_1$, and $\hat{\alpha}/||\alpha||_2$ give
\begin{align}
&\frac{1}{\sigma_{\max}(\tilde{\Gamma})}\sum_{t=1}^T \hat{c}_t\hat{\alpha}^\top \tilde{\Gamma}^\top W_t\hat{\psi} \leq 1 \ \ \forall (\hat{\alpha},\hat{\psi},\hat{c}) \nonumber\\
\iff \sup_{\hat{\alpha},\hat{\psi},\hat{c}}&\frac{1}{\sigma_{\max}(\tilde{\Gamma})}\sum_{t=1}^T \hat{c}_t \hat{\alpha}^\top\tilde{\Gamma}^\top W_t\hat{\psi} \leq 1 \ \ \textnormal{s.t.} \ \ ||\hat{\alpha}||_2=||\hat{\psi}||_2=||\hat{c}||_1=1.
\end{align}
This is equivalent to checking that 
\begin{equation}
\label{eq:globalCheckGamma2}
\max_{1\leq t \leq T }\frac{1}{\sigma_{\max}(\tilde{\Gamma})}\sigma_{\max}(\tilde{\Gamma}^\top W_t) \leq 1.
\end{equation}

  \begin{algorithm}[t]
  \caption{Global Optimality Check and Update}
  \begin{algorithmic}
  \label{alg:globalalgorithm}
  \FOR{$t=1 \dots T$}
  \STATE $\hat{X}_t = \tilde{\Gamma} \tilde{C}_t \tilde{\Psi}^\top;$
  \STATE $g_t = \sigma_{\max}(-\tilde{\Gamma}^\top(\hat{X}_t - S_t)/\lambda \sigma_{\max}(\tilde{\Gamma}));$
   \STATE $p_t = \sigma_{\max}(-(\hat{X}_t - S_t)\tilde{\Psi}/\lambda \sigma_{\max}(\tilde{\Psi}));$
      \STATE $c_t = \sigma_{\max}(-(\hat{X}_t - S_t)/\lambda);$
  \ENDFOR
  \STATE $g = \max_{t} g_t$;
    \STATE $p = \max_{t} p_t$;
  \STATE $c = \max_{t} c_t$;
  \IF{$g > 1$ and $g > p$}
    \STATE Compute global step-size $\tau$ via \eqref{eq:globalstep}
    \STATE Update $\u{C}$ and $\Psi$
  \ELSIF{$p > 1$ and $p > g$}
    \STATE Compute global step-size $\tau$ via \eqref{eq:globalstep}
    \STATE Update $\Gamma$ and $\u{C}$
  \ELSIF{$c >1$}
    \STATE Compute global step-size $\tau$ via \eqref{eq:globalstep}
    \STATE Update $\Gamma$, $\u{C}$ and $\Psi$
  \ELSE
    \STATE $\Gamma^* = \tilde{\Gamma};\u{C}^* = \tilde{\u{C}}; \Psi^* = \tilde{\Psi};$
    \STATE $\hat{S} = \hat{X};$
  \ENDIF
  \end{algorithmic}
  \end{algorithm}

If this inequality is violated, this implies that $\gamma_{t_\ast}$, the right singular vector of $\tilde{\Gamma}^\top W_t$ corresponding to the maximum singular value $\sigma_{\max}(\tilde{\Gamma}^\top W_t)$, could be appended to $\tilde{\Gamma}$ to give a global descent direction. But because $\gamma_{t_\ast} \in \textnormal{Span}(\tilde{\Gamma})$, it is not necessary to add it to find the descent direction. Therefore, we can just update $\Psi$ and $\u{C}$ as $\Psi = [\tilde{\Psi}, \psi_{t_\ast}]$ and $C_{t_\ast} = [\tilde{C}_{t_\ast}, \tau\alpha_{t_\ast}]$ and replace $\alpha_{t_\ast}$ by 0 for all other slices. The optimal step-size $\tau$ can again be found by solving \eqref{eq:globalstep} with $E_{t_\ast} = \tilde{\Gamma}\alpha_{t_\ast} \psi_{t_\ast}^\top$.

On the other hand, if \eqref{eq:globalCheckGamma2} is satisfied, we must then check the analogous condition for $\Psi$ with $\psi = \tilde{\Psi}\beta$:
\begin{equation}
\label{eq:globalCheckPsi2}
    \max_{1\leq t\leq T}\frac{1}{\sigma_{\max}(\tilde{\Psi})} \sigma_{\max}(W_t\tilde{\Psi}) \leq 1
\end{equation}
Now, if \eqref{eq:globalCheckPsi2} is violated this means we do not need to update $\Psi$ and just update $\Gamma = [\tilde{\Gamma}, \gamma_{t_\ast}]$ and $C_{t_\ast} = [\tilde{C}_{t_\ast}; \tau\beta_{t_\ast}]$ with $\beta_{t_\ast}$ replaced by 0 for all other slices. The optimal step size $\tau$ is found by \eqref{eq:globalstep} with $E_{t_\ast} = \gamma_{t_\ast}\beta_{t_\ast}^\top \tilde{\Psi}^\top$. If this too is satisfied, then we must check the original criteria \eqref{eq:cor2} to potentially update both dictionaries if violated. The order of these global checks can depend on knowledge of the intended sizes of each dictionary.  For our purposes we propose to check which of the two violates the corresponding constraint the most, \ie which one leads to the larger global step. Because \eqref{eq:globalCheckGamma2} and \eqref{eq:globalCheckPsi2} are lower bounds of \eqref{eq:cor2}, satisfying them will not be sufficient to guarantee that we have reached a global minimum and so \eqref{eq:cor2} is still necessary to check in this case. 
%Figure~\ref{fig:cmaxobjval} (left) illustrates the convergences of the relationship between the maximum $\sigma_{max}$ of \eqref{eq:globalCheckGamma2}, \eqref{eq:globalCheckPsi2} and \eqref{eq:globalCheck2} In practice, as a computational approximation to avoid adding a large number of atoms, we add a small error threshold $\epsilon$ to the global criteria \eqref{eq:globalCheck2}.
%At the same time, Figure~\ref{fig:cmaxobjval} (right) shows the decrease in objective value during the local descent phase and each global step. 
The complete procedure for the global optimality check and dictionary size update is outlined in Algorithm~\ref{alg:globalalgorithm}. %{\color{red} Overall I think this discussion could be cleaned up a little.}

The numerical cost of the overall dictionary learning algorithm highly depends on the number of iterations needed until convergence of each local descent as well as the total number of runs necessary to reach global optimality. One easily finds that each local descent iteration has a complexity of the order of $O(T(GV \tilde{N}_{\Gamma} + \tilde{N}_{\Gamma} \tilde{N}_{\Psi} V))$, where $\tilde{N}_{\Gamma}, \tilde{N}_{\Psi}$ are the number of atoms in the dictionaries at the current run. In addition, the global optimality check at the end of each run essentially amounts to the computation of $T$ maximum singular values of matrices of size $G\times V$, which can be computed more efficiently using iterative algorithms such as the power method as opposed to having to compute full SVD decompositions.

 % \begin{figure}
 %     \centering
 %     \includegraphics[width=.5\linewidth]{}
 %     \includegraphics[width=.5\linewidth]{}
 %     \caption{Left: Right:}
 %     \label{fig:cmaxobjval}
 % \end{figure}

 \section{Numerical validation on synthetic data}
 \label{sec:numerics} 
%Before evaluating the above approach on the case of real diffusion MRI images in the next section, we first 
In this section, we provide a few experiments on simple, small synthetic data to validate numerically the global convergence properties of the proposed method. Specifically, we consider a synthetic dataset of $T=1,200$ signals each having $V=100$ spatial pixel samples and $G=10$ angular samples. These signals are generated as follows. We start with the predefined six spatial and three angular atoms shown in Figure \ref{fig:original_atoms}, where the spatial atoms $\{\bar{\Psi}_j\}_{j=1}^6$ as displayed as as $10\times 10$ images and the angular atoms $\{\bar{\Gamma}_i\}_{i=1}^3$ are displayed as 10x1 signals,
%
%Denoting those $\bar{\Psi}_j$ for $j=1,\ldots,6$ and $\bar{\Gamma}_i$ for $i=1,\ldots,3$ respectively, 
%
and generate dMRI signals as $S_t = \sum_{p=1}^{m_t} \sum_{q=1}^{n_t} s_{p,q,t} \bar{\Gamma}_{i_p}^T \bar{\Psi}_{j_q} + \epsilon_t$, where $m_t$ and $n_t$ are random integers in $\{1,2\}$ and $\{1,2,3\}$ respectively, $i_p$, $i_q$ are random indices drawn uniformly in $\{1,2,3\}$ and $\{1,\ldots,6\}$ respectively, $s_{p,q,t}$ are random coefficients rescaled to verify $\sum_{p,q} s_{p,q,t} =1$ and $\epsilon_t$ is a random $10 \times 100$ Gaussian noise matrix of variance $0.003$. 
 
 \begin{figure}
    \centering
    \begin{tabular}{ccc}
    \includegraphics[width=.48\linewidth]{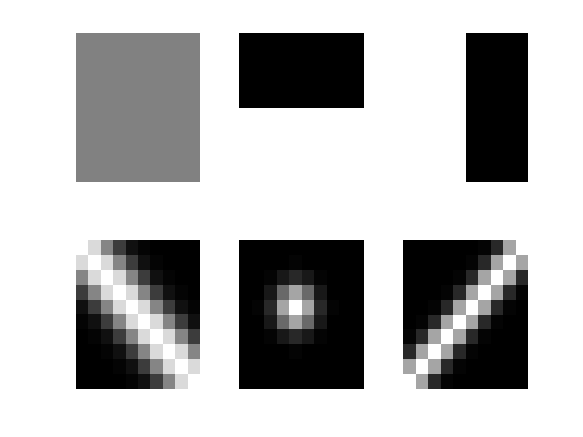} & &
    \includegraphics[width=.35\linewidth,height=0.35\linewidth]{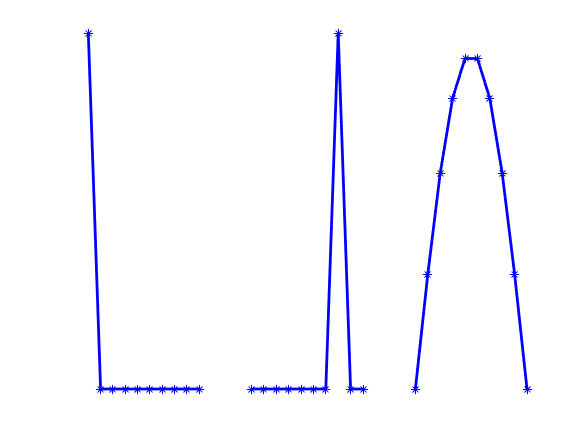}\\
    $\Psi_j$ & & $\Gamma_i$
    \end{tabular}
    \caption{Original set of six spatial dictionary atoms ($\bar{\Psi}_j$, left) and three angular dictionary atoms ($\bar{\Gamma}_i$, right) used in our synthetic experiment. A random combination of a subset of these atoms are used to generate synthetic $10 \times 100$ signals.}
    \label{fig:original_atoms}
\end{figure}

 \begin{figure}
    \centering
    \begin{tabular}{ccc}
    \includegraphics[width=.46\linewidth]{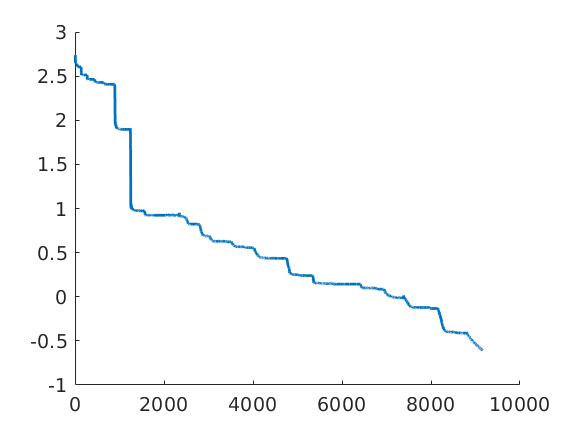} & & \includegraphics[width=.46\linewidth]{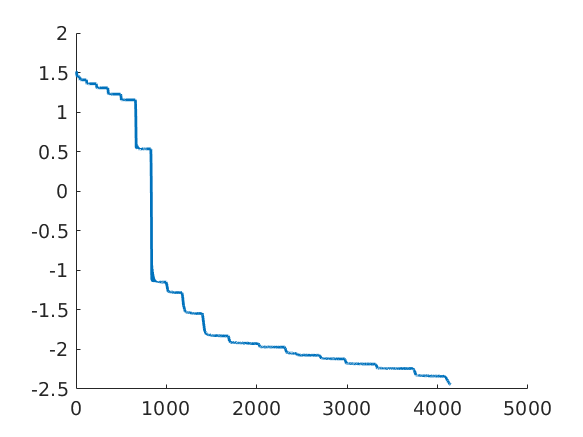} \\
    $\lambda=0.75$: $r_1=75, r_2 = 98, \text{sp}=0.5\%, \tilde{r}=1002$ & & $\lambda=0.85$: $r_1=32, r_2 = 44, \text{sp}=0.34\%,\tilde{r}=702$ \\
    \includegraphics[width=.46\linewidth]{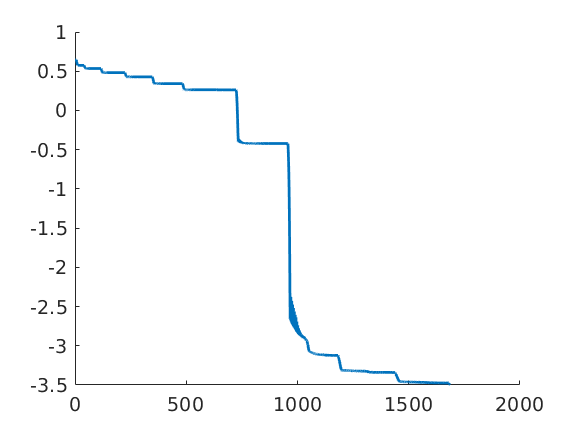} & & \includegraphics[width=.46\linewidth]{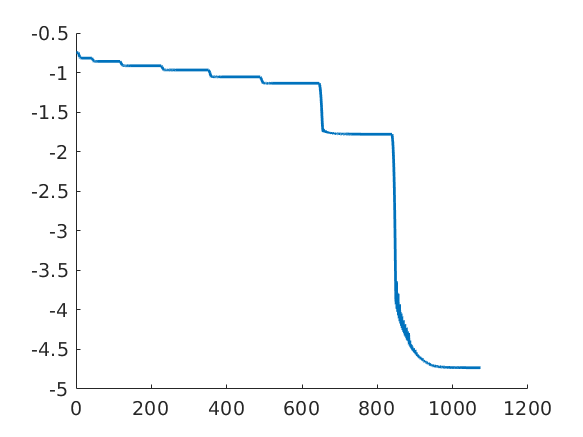} \\
    $\lambda=0.9$: $r_1=7, r_2 = 12, \text{sp}=0.59\%,\tilde{r}=460$ & & $\lambda=0.95$: $r_1=3, r_2 = 8, \text{sp}=1.12\%,\tilde{r}=369$
    \end{tabular}
    \caption{Evolution of the difference (in logarithmic scale) between the objective value computed by our algorithm and the globally optimal value computed by the slice by slice SVD method as a function of the number of iterations of our algorithm for four different values of $\lambda$. Also provided are the number of atoms $r_1$, $r_2$ of the estimated dictionaries with the overall sparsity of the coefficient tensor $\u{C}$, as well as the size $\tilde{r}=\tilde{r}_1=\tilde{r}_2$ of the dictionaries obtained by the direct slice SVD approach.}
    \label{fig:synth_energy_curves}
\end{figure}

By following the direct slice by slice SVD approach described in Section \ref{ssec:connection_low_rank}, we can compute the groundtruth global minimum value for problem \eqref{eq:SepDictionaryLearningGlobal}. The corresponding optimal dictionaries $\Gamma$ and $\Psi$ are, however, very redundant since they both contain $\tilde{r} = \sum_{t=1}^{T} \text{rank}(\mathcal{D}_\lambda(S_t))$ atoms, where the values of $\tilde r$ for the different choices of $\lambda$ are given in Figure \ref{fig:synth_energy_curves}. We expect that more compact near optimal solutions can be recovered by our proposed separable dictionary learning algorithm. We thus solve \eqref{eq:SepDictionaryLearningGlobal} for different choices of the regularization parameter $\lambda$ using Algorithm \ref{alg:meta}, where $\Gamma$ and $\Psi$ are both initialized with a single random atom. We stop our algorithm when the global optimality certificate is satisfied up to a small error.

The gap between the objective value computed by our algorithm and the globally optimal value at each iteration, as well as the final number of dictionary atoms, are shown in Figure \ref{fig:synth_energy_curves}. There are several remarks to be made on those results. First, we see that in all cases the gap reduces to a value that is close to zero.
%energy decreases to a value close to the groundtruth minimum energy which, we recall, is not known a priori by the algorithm.
%\footnote{\color{blue}Not quite. In log scale, curve should have negative slope, but some are converging.}
%Each jump in the energy curve
Second, we verified that each sudden reduction of the gap corresponds to the end of a local descent and the addition of a new atom to either $\Gamma$, $\Psi$ or both. Third, we observed that this iterative increase of dictionary sizes provides in general better initializations for each new local descent and thus leads to critical points much closer to the global minimum than a single local descent with the last values of $r_1$ and $r_2$. In particular, we noted that running local descent initialized with random dictionaries of the `optimal' sizes $r_1$, $r_2$ estimated by our algorithm leads to a critical point that is usually very far from a global minimum. Fourth, we see that convergence of our method may take in general many more iterations for lower values of $\lambda$, which also leads to a higher number of atoms. On the other hand, for values of $\lambda$ larger than $1$ we find that the optimal solutions always reduce to $\Gamma=0$, $\Psi=0$ and $\u{C}=0$.
Fifth, note that our method leads to much smaller values for $r_1$ and $r_2$ than the direct slice by slice SVD approach and also consistently estimates more atoms for $\Psi$ than for $\Gamma$, which is expected given the higher dimensionality of the spatial component for these signals and the higher number of original spatial atoms used to generate the dataset.

 \begin{figure}
    \centering
    \begin{tabular}{ccc}
    \includegraphics[width=.52\linewidth]{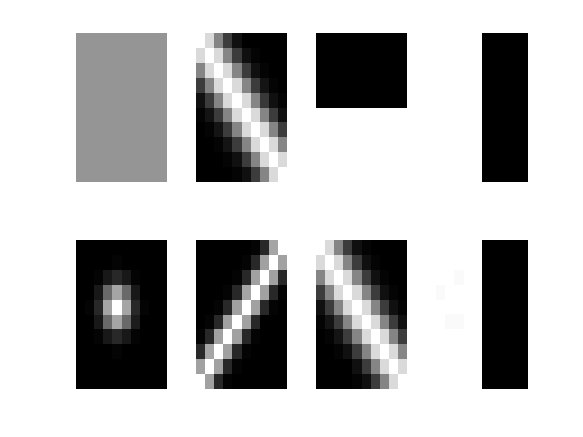} & & \includegraphics[width=.40\linewidth,height=0.35\linewidth]{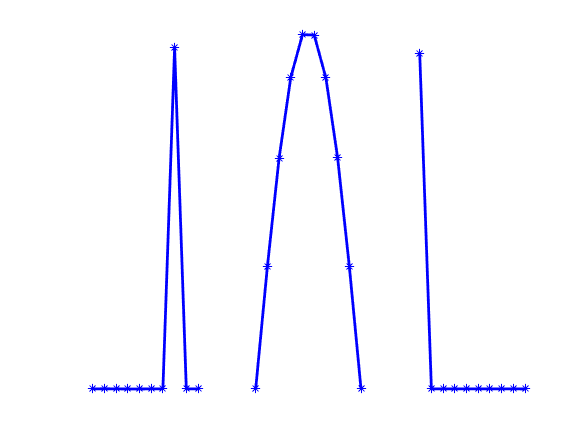} \\
    $\Psi_j$ & & $\Gamma_i$
    \end{tabular}
    \caption{With $\lambda=0.95$, the original dictionary atoms are recovered (with some redundancy).}
    \label{fig:synth_estimated_atoms_lam095}
\end{figure}

It is also very informative to visualize the different dictionary atoms 
%of $\Gamma$ and $\Psi$ 
recovered by our algorithm. Figure \ref{fig:synth_estimated_atoms_lam095} shows the atoms found for $\lambda =0.95$. We can see that the three atoms in $\Gamma$ are very close to the original ones used to create the data, while $\Psi$ contains the original spatial atoms of Figure \ref{fig:original_atoms}, plus an additional near replica of the third atom. For lower values of $\lambda$, we observe that dictionaries typically contain additional spurious duplicates as well as combinations of those original atoms, some of which are shown in Figure \ref{fig:synth_estimated_atoms_lam075}.

 \begin{figure}
    \centering
    \begin{tabular}{ccc}
    \includegraphics[width=.48\linewidth]{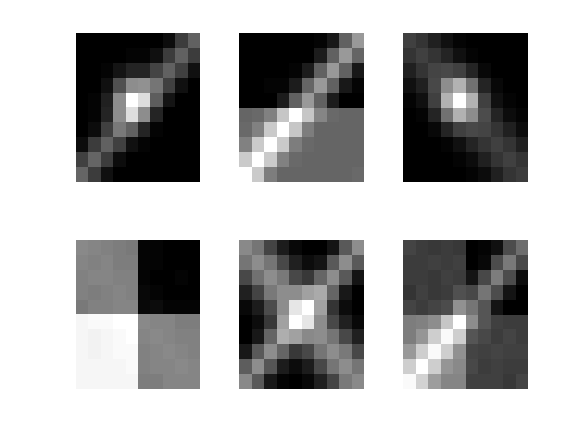} & & \includegraphics[width=.37\linewidth,height=0.33\linewidth]{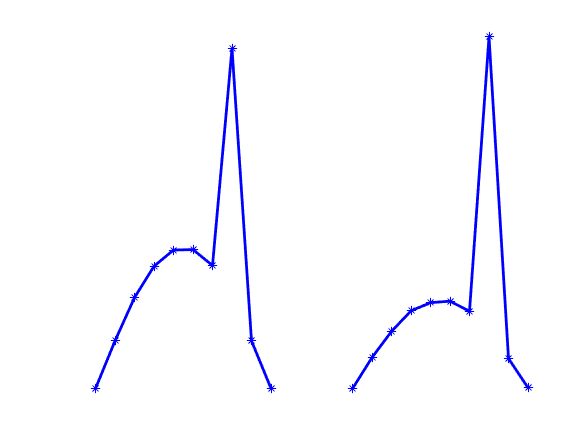}\\
    $\Psi_j$ & & $\Gamma_i$
    \end{tabular}
    \caption{With $\lambda=0.75$, some more complex spatial and angular atoms are obtained.}
    \label{fig:synth_estimated_atoms_lam075}
\end{figure}

\section{Application to diffusion magnetic resonance imaging}
 \label{sec:dMRIbackground}

In this section, we demonstrate the applicability of the proposed separable dictionary learning algorithm to the analysis of diffusion magnetic resonance imaging (dMRI) data. We first summarize the basics of dMRI reconstruction and review the literature of dictionary learning applied in this field. We then show how our method can be used to learn dictionaries for dMRI, and demonstrate the performance of such dictionaries in the task of signal denoising.

\subsection{Basics of dMRI}
\label{sec:introdMRI}

Diffusion magnetic resonance imaging (dMRI) is a medical imaging modality that can be used to study the structure of the complex network of neurons in the brain $\textit{in vivo}$  \cite{Tuch:Neuron03}. At each voxel, dMRI measures the diffusion of water molecules along multiple orientations in 3D space. Since the direction of maximum diffusivity is correlated with the direction of axons in the brain, dMRI can be used to estimate the local orientation of neuronal fiber bundles and reconstruct a network of fiber tract connections. Such networks allow researchers to study, \eg anatomical brain variations that are essential for understanding and predicting neurological disorders such as Alzheimer's disease or traumatic brain injury. 

The basics of dMRI can be summarized as follows. First, a diffusion signal $s_v\in\mathbb{R}^G$ is measured at each voxel $v=1,\dots,V$ in a 3D brain volume, resulting in a spatial-angular signal $S$ of total dimension $G\times V$ \cite{Tuch:Neuron03}. Diffusion is commonly measured angularly on a unit sphere of the diffusion domain known as $q$-space and therefore diffusion signals are commonly represented by an angular or spherical dictionary, $\Gamma \in \mathbb{R}^{G\times r_1}$, \eg spherical harmonics, spherical wavelets, as $S \approx \Gamma W$. Then, from these diffusion signals, one can estimate the orientation distribution function (ODF), which measures the probability of having a fiber bundle oriented in a particular direction at each voxel \cite{Aganj:MRM10}. The ODF is often the starting point of tractography algorithms that exploit this directional information to extract fiber bundles \cite{Fillard:Neuroimage11}. Specifically, the direction of of maximum diffusion is used as an estimate of local orientation, and fiber tracts are obtained by following these estimates of local orientation.  In this paper, we are focused primarily in the first step of dMRI analysis, which is the reconstruction of the image volumes from the dMRI signals. That said, 
%As a note, even though we will reconstruct the diffusion signals as part of our optimization, 
for convenience in the visualization of the results, we shall display the reconstructed ODFs in our figures, which are computed using a standard method involving spherical harmonics and the Funk-Radon transform \cite{Goh:Neuroimage11}. Each ODF is a spherical probability distribution where yellow are high values and blue are low values in the visualization.

%segmentation \cite{Cetingul:ISBI11}

 \subsection{Dictionary learning for dMRI}
 As discussed above, dMRI reconstruction methods often use a fixed dictionary $\Gamma$. In practice, in applications such as denoising and compressed sensing, better representations can be obtained by learning a dictionary directly from dMRI data.
% In this work, we are interested in learning dictionaries for highly structured medical imaging data known as diffusion magnetic resonance imaging (dMRI) \cite{}. 
%\subsection{Dictionary learning for dMRI}
%\label{sec:introdMRI}
%\label{sec:AngularDictionaryLeanring}
%Learning dictionaries directly from dMRI data has been proposed for many applications such as denoising and compressed sensing. 
Most existing dictionary learning methods for dMRI are based on solving the classical dictionary learning problem \eqref{eq:DictionaryLearning}:
\begin{equation}
\label{eq:AngularDictionaryLearning}
\min_{\Gamma,W} \  \frac{1}{2} ||\Gamma W - Y||_F^2 + \lambda ||W||_1 \ \ \textnormal{s.t.} \ \ ||\Gamma_i||_2 \leq 1 \ \ \textnormal{for} \ \ i=1\dots r_1,
\end{equation}
where $Y = [y_1,\dots,y_T] \in \mathbb{R}^{G \times T}$ are $T$ training examples of angular signals taken, for example, from a subset of representative voxels in a brain image, and $W\in \mathbb{R}^{r_1 \times T}$ are the associated angular coefficients.
%Building the matrix of signal and coefficient training examples, $S = [s_1,\dots, s_T] \in \mathbb{R}^{G\times T}$ and $A = [a_1,\dots,a_T] \in \mathbb{R}^{N_\Gamma \times T}$, \eqref{AngularDictionaryLearning} can be restated as:
%\begin{equation}
%\label{eq:AngularDictionaryLearning_mat}
%\min_{A,\Gamma} g(A) \ \ \textnormal{s.t.} \ \ \frac{1}{2} ||\Gamma A - S|_F^2 \leq \epsilon, \ \ ||\Gamma_i||_2^2 \leq 1 \ \ \forall  \ i=1\dots N_\Gamma.
%\end{equation}
There have been a multitude of works that aim to solve \eqref{eq:AngularDictionaryLearning} or some alternative versions by proposing different models like parametric dictionary learning \cite{Aranda:MIA15,Cheng:MICCAI13,Cheng:MICCAI15,Merlet:MICCAI12,Merlet:MIA13,Ye:ISBI12}, which learn parameters of predefined diffusion models, Bayesian learning \cite{Gupta:IPMI17,Pisharady:Neuroimage17}, manifold learning \cite{Sun:IPMI13} and dictionary learning directly from undersampled data for compressed sensing \cite{Bilgic:MRM12,Gramfort:MIA14,Gupta:CDMRI16,Gupta:IPMI17,Mcclymont:MRM15}.
%Each of these various methods learns a purely angular $q$-space dictionary used to reconstruct signals in each voxel. 
% \begin{figure}
% \centering
% \includegraphics[width=.7\linewidth,trim=0 0 0 0,clip]{AngularDictionaries2.pdf}
% \caption{Angular Dictionaries for HARDI: \textit{A}. ODFs produced from analytic spherical ridgelet dictionary atoms. \textit{B}. ODFs produced from angular dictionary atoms learned from HARDI data. \textit{C}. ODFs of real HARDI brain data. Learning angular dictionary atoms from data may be more exemplary of real dMRI signals, producing a sparser representation than fixed analytic dictionaries.}
% \label{fig:AngularDictionaries}
% \end{figure}
While some of these methods impose additional spatial coherence between neighboring voxels \cite{Ye:ISBI12}, training examples $y_t$ are usually taken voxel-wise without considering their spatial correlations. In other words, these works have not considered learning joint spatial-angular dictionaries in the context of dMRI. The work of \cite{st2016non} considers both the spatial and angular components in dictionary learning applied to dMRI denoising, but restrict their method to both spatial and angular patches and solve the classical dictionary learning problem \eqref{eq:DictionaryLearning} after vectorizing the spatial-angular diffusion signal.

To incorporate the spatial domain, we can compile the signals $s_v$ at each voxel of the brain volume into the matrix $S=[s_1,\dots,s_V]\in\mathbb{R}^{G\times V}$.  Then, given an angular dictionary $\Gamma \in \mathbb{R}^{G\times r_1}$ and a spatial dictionary $\Psi \in \mathbb{R}^{V\times r_1}$, the entire dMRI image may be represented (or approximated) as
\begin{equation}
\label{eq:SpatialAngularMatrixForm}
S = \Gamma C \Psi^\top,
\end{equation}
where $C\in \mathbb{R}^{r_1 \times r_2}$ stores the coefficients in the joint spatial-angular dictionary.
This separable spatial-angular representation of dMRI data fits directly into our separable dictionary learning framework \eqref{eq:SeparableDictionaryLearning} with $T$ training example diffusion volumes $\{S_t\}_{t=1}^T \in \mathbb{R}^{G\times V}$.

In prior work \cite{Schwab:MICCAI16,Schwab:MIA18} we demonstrated that sparse coding with respect to fixed separable dictionaries over the spatial and angular domain provides sparser reconstructions than the traditional angular sparse coding. Therefore, by learning both spatial and angular dictionaries directly from dMRI data, we should be able to provide even sparser reconstructions than the state of the art. Our related work also support the use of analytic spatial-angular dictionaries in order to lower subsampling rates for dMRI compressed sensing \cite{Schwab:CDMRI17}.  In the future work, we also hope to improve such results through spatial-angular dictionaries learned from the data.

\subsection{Patch-based training for dMRI}
\label{DL:training}

%We tested our methods on the phantom ($G=64$) and real HARDI ($G=127$) datasets. 
In theory, our spatial-angular dictionary learning method is capable of learning global spatial and angular dictionaries, $\Psi \in \mathbb{R}^{G\times r_1}$ and $\Gamma \in \mathbb{R}^{V\times r_2}$, over an entire dMRI dataset of size $G\times V$.  However, the typical size of a HARDI brain volume is on the order of $V=100^3$ voxels, and $G=100$ diffusion measurements, \ie of size $G\times V = 10^8$. Furthermore, the number of training examples $T$ depends on the size of the training sample. This would require a very large number of training examples of entire dMRI datasets, which is largely infeasible for our algorithm. Because the spatial domain is orders of magnitude larger than the angular domain, one way to curb the computational burden 
%may be to treat the spatial volume as a tensor $\u{S}\in\mathbb{R}^{G\times V_x \times V_y \times V_z}$ and learn separate dictionaries for each spatial dimension in addition to the diffusion dimension. However training examples would still consist of the large tensor data and would still be computationally demanding. Alternatively, we can
is to reduce our dictionary learning to local spatial patches for all diffusion measurements (\ie 3D patches of size $P\times P \times P$).
%In this paper we investigate learning spatial patches to be consistent with the algorithm formulation of \eqref{eq:SeparableDictionaryLearning}. 
Patch-based methods are indeed very popular in image processing for tasks such as denoising, filtering, inpainting, and object detection \cite{Yellin:ISBI17}. In addition, local dictionaries are beneficial for capturing local features that are often repeated in an image, such as edges, textures or objects. Note that another possible approach could consist in learning separable dictionaries in the different axes $x,y,z$ of the spatial domain as well and thus involve separable dictionaries with a larger number of factors.   
\begin{figure}
    \centering
    \includegraphics[width=.66\linewidth]{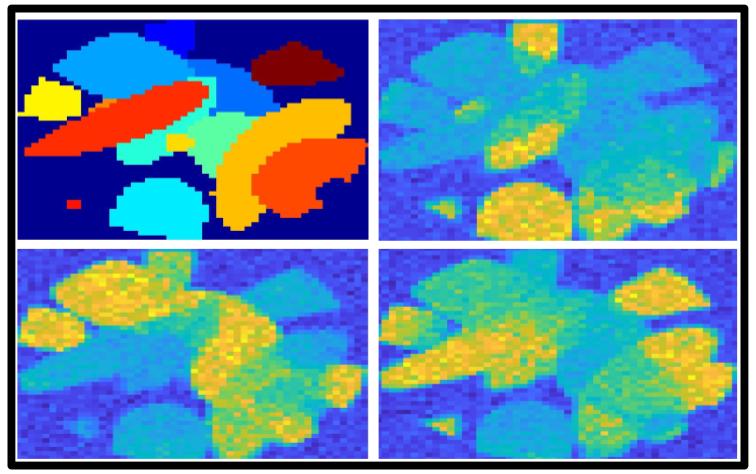}
    \includegraphics[width=.69\linewidth]{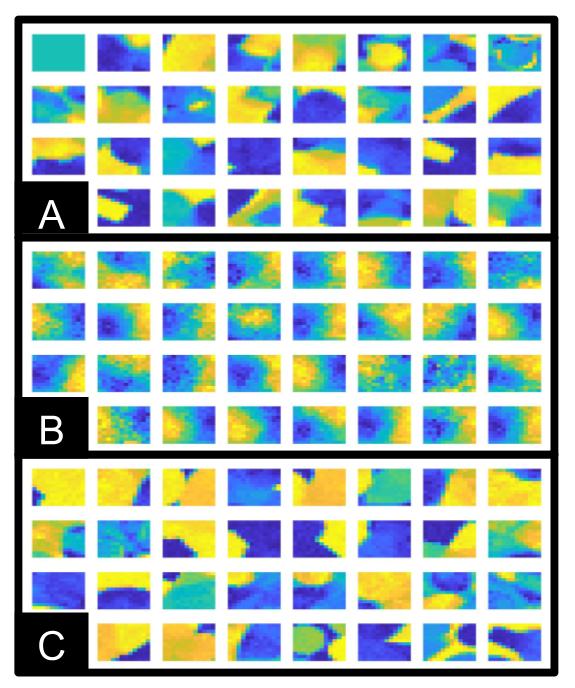}
    \caption{Top: Phantom HARDI ground truth fiber segmentations (top left, each color represents a different segmentation) and three diffusion weighted images (signal color coding: yellow positive, green negative, blue zero) used for training on patches of size $12\times 12$. Bottom: Subset of spatial patch dictionaries learned via A. KSVD independently from angular dictionary, B. KDRSDL jointly with angular dictionary, C. the proposed method jointly with angular dictionary. B. appears to have reached a spurious local minimum while A. and C. closely resemble each other and pick up sharp edges and shapes present in the training phantom. (Spatial dictionary color coding: yellow positive, blue negative, green zero.)}
    \label{fig:SpatialDictionaries}
\end{figure}

 %In fact, many analytic spatial dictionaries are naturally patch-based, such as wavelets, or curvelets, which contain a filter bank of multiple scales and orientations repeated at each location in an image. 
%Another popular filter for sparse coding is the discrete gradient operator for total variation.

\begin{figure}
    \centering
    \includegraphics[width=.7\linewidth]{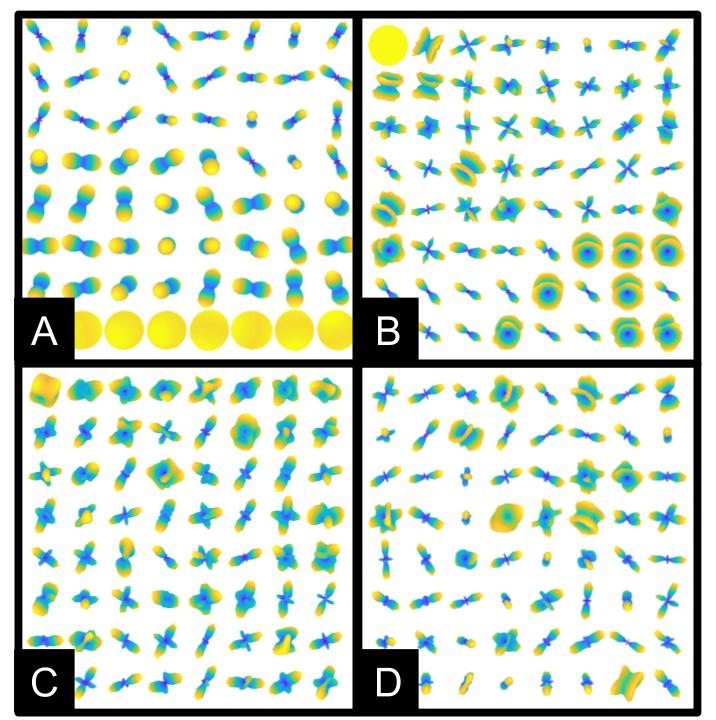}
    \caption{Comparison of angular dictionaries. $A.$ Fixed spherical ridglets. $B.-D.$ Subset of angular dictionaries trained on the phantom HARDI data learned via B. KSVD  independently from spatial dictionary, $C.$ KDRSDL jointly with spatial dictionary, and $D.$ the proposed method jointly with spatial dictionary. The learned dictionaries provide more complex fiber crossing configurations than the fixed spherical ridglets. (Each ODF is a spherical probability distribution where yellow are high values and blue are low values).}%KSVD and the proposed method produce clean single fiber ODFs while KDRSDL ODFs are noisier.}
    \label{fig:AngularDictionaries}
\end{figure}

\begin{figure}
    \centering
    \includegraphics[width=1\linewidth]{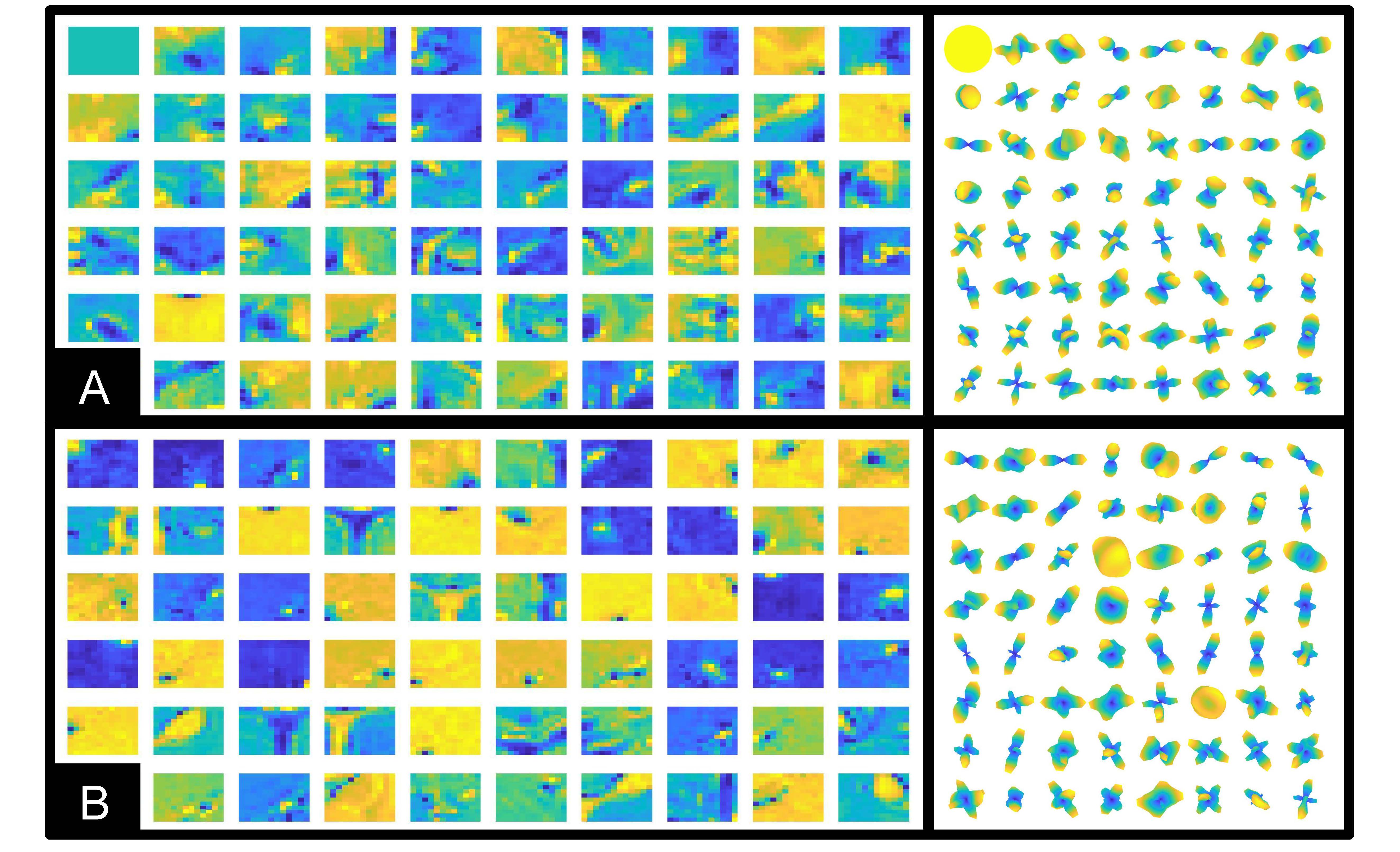}
    \caption{$A.$ Spatial and angular dictionaries learned independently via KSVD. Each are sorted (left to right, top to bottom) by their individual frequencies of use in modeling the training data. $B.$ Spatial and angular dictionaries learned jointly by the proposed method. Each are sorted (left to right, top to bottom), by their joint frequencies. For example, the top left spatial and angular atoms are together the most frequently used joint spatial-angular atom.}
    \label{fig:spatialangulardictionariesReal}
\end{figure}
For training we thus choose a random selection of spatial patches that is consistent along the diffusion domain. For computational simplicity and purposes of visualization, we limit our experiments in this paper to 2D spatial patches of size $P\times P$ instead of 3D,  \ie $S_t \in \mathbb{R}^{G\times P^2}$. We acknowledge that extending to 3D would provide a further reduction in spatial redundancies during sparse coding and the ideal case would be to learn a dictionary for the entire 3D volume.
%Thus training examples consist of randomly located $P\times P\times G$ subsets of the dMRI data over a set of slices.  
Depending on the detail and size of an image, popular patch sizes range from $P=5$ to $15$. For our data, $P=12$ gives a good amount of detail and is not too large to process.

For choosing the number of training examples, $T$, we can consider the number of training examples typical of angular dictionary learning.  For instance, the work of \cite{Merlet:MICCAI12} use $5,000$ angular signals to train their angular dictionary. To reach this number, we only need a relatively small number of $G\times P\times P$ patches to provide an adequate number of spatial and angular training examples, respectively.
%For instance, the phantom dataset has large sections of homogeneity while the real dataset has more minute variability. 
In total, the number of angular training examples will be $P^2T$ and the number of spatial training examples will be $GT$. For a typical dMRI dataset with $G=100$ and $P=12$, we will need on the order of $T=40$ training patches, to have around $5,000$ angular training examples and around $4,000$ spatial training examples. In this work, we use $T=100$ training patches over multiple image slices, resulting in about $14,400$ angular training examples and $10,000$ spatial training examples.
%In order to properly compare with KSVD which learns the spatial and angular dictionaries independently, we use the same set of $GT$ spatial and $P^2T$ angular training examples.

\subsection{Denoising experiment}
%\label{sec:experiments}

%\subsection{Setup}
%\label{sec:data}

For our application we learn our dictionaries from high angular resolution diffusion imaging (HARDI) data \cite{Descoteaux:TMI09}. Specifically, we experimented on a phantom and a real HARDI brain dataset. The phantom is taken from the ISBI 2013 HARDI Reconstruction Challenge\footnote{http://hardi.epfl.ch/static/events/2013\_ISBI/},
a $V\!=\!50\!\times\!50\!\times\!50$ volume consisting of 20 phantom fibers crossing intricately within an inscribed sphere, measured with $G\!=\!64$ diffusion measurements. Our initial experiments test on a 2D $50\!\times\!50$ slice of this data for simplification. We use patches of size $P=12$, \ie $12 \times 12$.

The phantom dataset includes two noise levels: a low noise level of SNR=30 dB and a high noise level of SNR=10 dB. The denoising task will be to denoise the SNR=10 dB data using dictionaries learned from the SNR=30 dB data and record the error with respect to the ``ground truth'' SNR=30 dB data by calculating Peak SNR (PSNR):
\begin{equation}
    PSNR = 10\log_{10}\frac{MAX_I^2}{MSE},
\end{equation}
where $MAX_I$ indicates the maximum value in the original SNR=30 dB signal, and MSE is the mean squared error between the original SNR=30 dB signal and the reconstruction. The higher the PSNR, the more accurate the reconstruction will be. We chose a subset of slices of the SNR=30 dB to learn our 2D spatial-angular dictionaries and used a selection of the remaining slices as test data for denoising.  

After validation on phantom data, we show qualitative denoising results on a real HARDI volume with $G=127$ diffusion measurements using our proposed spatial-angular dictionaries learned on a subset of 2D slices with patches of size $12 \times 12$.

 %\subsection{Data}
 %\label{sec:data}

%We preprocessed this data by skull stripping using FSL \cite{} and narrowed the field of view to the mask of nonzero diffusion signal resulting in a compact size of $V\!=\!76\!\times\!89\!\times\!72$.  
%We conducted experiments on the core white matter brain region of size $V\!=\!60\!\times\!60\!\times\!30$.
 
 \subsection{Methods}
 \label{sec:methods}
 
We validate the proposed separable dictionary learning method by showing its performance on denoising HARDI data. While there are numerous denoising methodologies in the literature, we will focus on utilizing learned dictionaries in a sparse denoising method, which has been used frequently in the dMRI literature \cite{Gramfort:MIA14}. Note that our aim here is primarily to evaluate and compare different dictionary learning strategies through sparse denoising experiments but not to compare those results with the most advanced denoising algorithms in the field which typically involve additional processing steps \cite{st2016non}.       
\begin{table}[t]
\centering
\footnotesize{
\begin{tabular}{|c|c|>{\columncolor[gray]{0.9}}c|c|>{\columncolor[gray]{0.9}}c|c|>{\columncolor[gray]{0.9}}c|c|>{\columncolor[gray]{0.9}}c|}
\hline
Comparison & \multicolumn{2}{  c| }{1} & \multicolumn{2}{  c| }{2} & \multicolumn{2}{  c| }{3} & \multicolumn{2}{  c| }{4}\\ \hline
Method & Angular & Spatial-Angular & Fixed & Learned & Separate & Joint & Local & Global\\ \hline
\cellcolor{blue!15} I-SR & \checkmark & & \checkmark & & &  & & \\ \hline
\cellcolor{blue!15} I-SR + TV & \checkmark & & \checkmark & & &  & & \\ \hline
 \cellcolor{blue!15} Curve-SR & & \checkmark & \checkmark & & & &  & \\ \hline
    \cellcolor{red!15} I-KSVD & \checkmark & & & \checkmark & \checkmark &  & \checkmark &  \\ \hline
    \cellcolor{red!15} KSVD-KSVD & & \checkmark & & \checkmark & \checkmark & & \checkmark & \\ \hline
    \cellcolor{green!15} KDRSDL & & \checkmark & & \checkmark & & \checkmark & \checkmark & \\ \hline
    \cellcolor{green!15} Proposed & & \checkmark & & \checkmark & & \checkmark & & \checkmark \\ \hline
    \end{tabular}}
     \caption{Checklist of properties for each dictionary type to compare each method. Purple indicates fixed dictionaries, pink indicates spatial and/or angular dictionaries learned independently, and green indicates a joint spatial-angular dictionary.}
     \label{table:checklist}
\end{table}

\begin{table}[t]
\center
\begin{tabular}{cc|c|c|c|c|}
\cline{3-6}
& & \multicolumn{4}{ c| }{Angular} \\ \cline{3-6}
& & \multicolumn{1}{ c| }{\cellcolor{yellow!15} SR} & \multicolumn{1}{  c| }{\cellcolor{yellow!15} KSVD} & \multicolumn{1}{ c| }{\cellcolor{yellow!15} KDRSDL} & \multicolumn{1}{ c| }{\cellcolor{yellow!15} Proposed} \\ \cline{1-6}
\multicolumn{1}{ |c  }{\multirow{5}{*}{Spatial}}&
\multicolumn{1}{ |c| }{\cellcolor{yellow!15} I} &\cellcolor{blue!15} I-SR (+ TV) &\cellcolor{red!15} I-KSVD & & \\ 
\cline{2-6}
\multicolumn{1}{ | }{}  &
\multicolumn{1}{ |c| }{\cellcolor{yellow!15} Curve} & \cellcolor{blue!15}Curve-SR & \cellcolor{red!15}Curve-KSVD & &   \\ 
\cline{2-6}
\multicolumn{1}{ |  }{}  & \multicolumn{1}{ |c| }{\cellcolor{yellow!15} KSVD} & \cellcolor{red!15}KSVD-SR & \cellcolor{red!15}KSVD-KSVD & &   \\
\cline{2-6}
\multicolumn{1}{ |  }{}  & \multicolumn{1}{ |c| }{\cellcolor{yellow!15} KDRSDL} & & & \cellcolor{green!15} KDRSDL &   \\
\cline{2-6}
\multicolumn{1}{ |  }{}  & \multicolumn{1}{ |c| }{\cellcolor{yellow!15} Proposed} & & & & \cellcolor{green!15} Proposed  \\
\cline{1-6}
\end{tabular}
\caption{Organization of spatial and angular dictionaries. Purple indicates fixed dictionaries, pink indicates spatial and/or angular dictionaries learned independently, and green indicates a joint spatial-angular dictionary.}
%$^*$Though our proposed framework is intended for $(k,q)$-CS, the work in this paper performs only joint spatial-angular sparse coding.)}
\label{table:spatialangular}
\end{table}
For our learned spatial and angular dictionaries we use the spatial-angular sparse coding approach proposed in \cite{Schwab:MICCAI16,Schwab:MIA18} which amounts in solving 
\eqref{eq:SeparableDictionaryLearning} only for $C$ with $T=1$. For spatial patch-based dictionaries, we will apply sparse coding for each patch and average the results across overlapping patches. For denoising, we choose a value of $\lambda$, consistent for all patches, that gives the highest PSNR.

To validate the results of our proposed separable dictionary learning method we consider four dictionary comparisons based on the denoising performance:\\

\begin{enumerate}

    \item \textit{\textbf{Angular vs. Spatial-Angular:} will the proposed spatial-angular framework for dictionary learning and sparse coding outperform state-of-the-art framework for angular dictionary learning and sparse coding for denoising?}\label{comp:spatialangular_Vs_angular}
    \item \textit{\textbf{Fixed vs. Learned:} will dictionaries learned from dMRI data outperform fixed analytic dictionaries for denoising?}\label{comp:learned_Vs_fixed}
    \item \textit{\textbf{Separate vs. Joint:} will learning spatial and angular dictionaries $\textit{jointly}$ via separable dictionary learning better represent dMRI data than learning spatial and angular dictionaries independently each by classical methods like KSVD?}
    \label{comp:joint_Vs_separate}
    \item \textit{\textbf{Local vs. Global:} will our globally optimal separable dictionary learning outperform other locally optimal separable dictionary learning methods?}\label{comp:global_Vs_local}\\
\end{enumerate}

%1) learning spatial and angular dictionaries $\textit{jointly}$ via separable dictionary learning better represents dMRI data than learning spatial and angular dictionaries independently each by classical dictionary learning, 2) if our globally optimal separable dictionary learning outperforms other locally optimal separable dictionary learning methods, and 3) the learned dictionaries outperforms fixed analytic dictionaries.

For comparison~\ref{comp:spatialangular_Vs_angular}, we will compare against state-of-the-art angular dictionary learning and sparse coding frameworks. In particular, we will solve the angular dictionary learning problem \eqref{eq:AngularDictionaryLearning} with the commonly used KSVD algorithm \cite{Elad:TSP06}. For the angular sparse denoising step, we also add a spatial regularization term based on the total-variation (TV) in the spatial domain, as is commonly done in state-of-the-art dMRI denoising \cite{bao2008sparse}.
%In fact, angular sparse coding is a special case of spatial-angular sparse coding with the spatial dictionary equal to the identity matrix ($\Psi = I$). 

For comparison~\ref{comp:learned_Vs_fixed}, we will compare against two fixed angular and spatial dictionaries used in the dMRI literature:  the spherical ridgelet (SR) dictionary popularly used in angular sparse coding and compressed sensing for dMRI \cite{TristanVega:MICCAI11,Michailovich:TMI11,Michailovich:TIP10,Michailovich:MICCAI10} (see Figure~\ref{fig:AngularDictionaries} $A$ for visualization) and, for the spatial domain, the curvelet dictionary which has proved to be very efficient in sparsely representing classical images and was also shown to be a good choice for representing dMRI images in our recent works \cite{Schwab:MICCAI16,Schwab:MIA18}.  

For comparison~\ref{comp:joint_Vs_separate}, we will use the KSVD algorithm \cite{Elad:TSP06} to learn spatial and angular dictionaries independently.%\footnote{\textcolor{red}{As a note, we cannot apply our proposed method to solve spatial and angular dictionary learning separately, for example by setting $\Psi$ or $\Gamma$ to 1, because this reduces to the classical dictionary learning problem for which there is not a clear algorithm to solve while guaranteeing global optimality \cite{Haeffele:ICML14}.}} 
Identifying whether the proposed joint learning method is advantageous over the faster and easier approach of applying KSVD to each domain separately is indeed an important point to examine.

Finally, for comparison~\ref{comp:global_Vs_local}, we evaluate our approach against the Kronecker-Decomposable Robust Sparse Dictionary Learning (KDRSDL) algorithm of \cite{Bahri:ICCV17}, which is also a separable dictionary learning method that does not, however, provide guarantees of global optimality. KDRSDL solves a low-rank variation of \eqref{eq:SeparableDictionaryLearning} which the authors show is useful for background subtracting and image denoising.
%Within each, we compare the use of fixed dictionaries (curvelets (spatial) and spherical ridgelets (SR) (angular)) and learned dictionaries. The learned dictionaries for spatial-angular sparse coding come from the methods described above: using the angular and spatial dictionaries learned from KSVD independently and the joint spatial-angular dictionaries learned jointly with KDRSDL and our proposed method. For the angular sparse coding we also included the use of the TV norm which is a common spatial filter for denoising.  The organization of each method is outlined in Table~\ref{table:denoising} along with their PSNR values for three different HARDI phantom test slices not used in training. 

We use a ``Spatial-Angular" notation to keep track of the different dictionary choices, where, for example, I-SR uses the identity for the spatial dictionary and spherical ridgelets for the angular dictionary, I-KSVD learns the angular dictionary only using KSVD, and KSVD-KVSD uses the spatial and angular dictionaries learned by KSVD independently. See Table~\ref{table:checklist} for a checklist of the different dictionary properties for each of the 4 comparisons and Table~\ref{table:spatialangular} for a summary of the spatial and angular domains for each method.

%This method, however solves a variation of the separable dictionary learning problem which emphasizes low rank solutions instead of over-complete dictionaries.  In effect, our method lies somewhere between the overcomplete favoring KSVD and the low rank producing KDRSDL because we provide a flexibility to increase the ranks of our dictionaries at different rates, depending on the number of atoms we initialize with and how we choose to alternate between updates for each dictionary.  Therefore, based on the application we can produce global dictionaries that are low-rank or overcomplete.

\subsection{Visualization}
\label{sec:visualization}
In Figures~\ref{fig:SpatialDictionaries} and \ref{fig:AngularDictionaries} we visualize the spatial and angular dictionaries learned from each method on phantom HARDI data as well as the spherical ridgelet dictionary atoms in Figure~\ref{fig:AngularDictionaries} $A$. The learned dictionary atoms are organized left to right from top to bottom by the number of training examples that used each atom, \ie the number of nonzero coefficients associated to each atom in training. For KSVD, this ordering is independent for the spatial and angular dictionaries, while the atoms resulting from KDRSDL and the proposed method are ordered jointly (without repeats), \ie the top left spatial and angular atoms combine to create the most utilized spatial-angular atom.

\begin{figure}
    \centering
    \includegraphics[width=1\linewidth]{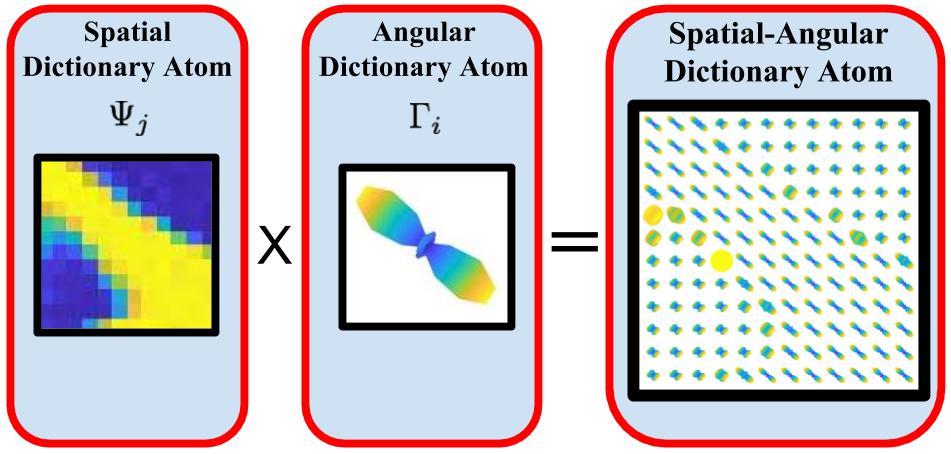}
    \caption{Spatial-Angular dictionary atom example learned jointly from phantom HARDI data with the proposed method. We can see that we have the ability to model fiber tracts with very few atoms. Because the spatial dictionary atom contains positive (yellow), negative (blue), and zero (green) values, when multiplied with the angular dictionary, there are non-informative ODFs in the blue regions of the spatial domain.}
    \label{fig:SpatialAngularDictionariesPhantom}
\end{figure}

For the spatial dictionaries in Figure~\ref{fig:SpatialDictionaries}, we notice clear similarities between our method and the atoms produced by KSVD.  In contrast, the spatial atoms produced by KDRSDL are fuzzier, lacking the clearly defined edges and geometric shapes that are evident in the phantom dataset. These shapes resemble atoms that have landed in a local minimum or saddle point, farther from the global minimum reached by our method. This trend is similar for the angular atoms in Figure~\ref{fig:AngularDictionaries}. We can see that the results of the proposed method has greater variation in the orientations of single fiber ODFs.  The most utilized atoms in KSVD are the purely isotropic atom and the noisy isotropic atoms, whereas the atoms most frequently used with the other methods are the single fiber atoms. In Figure~\ref{fig:SpatialAngularDictionariesPhantom} we show an example of a single spatial-angular atom learned jointly by our proposed separable dictionary learning method the resembles a fiber tract structure.

Finally, in Figure~\ref{fig:spatialangulardictionariesReal} we show spatial and angular dictionaries (bottom) learned from real HARDI brain data (top) for KSVD ($A$.) and our proposed method ($B$.). We notice large structures in the spatial atoms like the CSF region as well as atoms with specific spatial patterns resembling fiber structure. Each spatial-angular atom is sorted (left to right, top to bottom) by their frequency of use in the representation of the training data. For example, the top left spatial and angular atoms are together the most frequently used joint spatial-angular atom in training. (We show only a subset of unique spatial and angular atoms for visualization.)  

%EXPERIMENTS
%While there are numerous denoising methodologies in the literature, we will focus on utilizing our learned dictionaries within sparse coding which has been used frequently in the dMRI literature\cite{Gramfort:MIA14}. %With the dictionaries we have learned in the previous section, we reconstruct each $12\times 12$ patch in the image and average the overlapping reconstructions to produce a global denoised reconstruction.  

\subsection{Denoising results}
\label{sec:results}

%We use two types of sparse coding for denoising, angular sparse coding and our proposed spatial-angular sparse coding. Within each, we compare the use of fixed dictionaries (curvelets (spatial) and spherical ridgelets (SR) (angular)) and learned dictionaries. The learned dictionaries for spatial-angular sparse coding come from the methods described above: using the angular and spatial dictionaries learned from KSVD independently and the joint spatial-angular dictionaries learned jointly with KDRSDL and our proposed method. For the purely angular sparse coding method we also included an additional regularization using the TV norm which is a usual spatial denoising approach in image processing.

The results of the denoising experiment on the phantom HARDI data are recorded in Table~\ref{table:denoising}. We repeated the experiment on three slices of the phantom HARDI data that were not used for training. For each experiment, our reconstruction using the dictionaries learned jointly from our method achieved the highest PSNR values (right-most column of Table~\ref{table:denoising}). These preliminary results give an indication as to answers to the four types of dictionary comparison questions we outlined in Section~\ref{sec:methods}, showing that 1. Spatial-angular dictionaries outperform purely angular dictionaries (see Curve-SR vs I-SR, and KSVD-KSVD vs I-KSVD), 2. Learned dictionaries outperform fixed dictionaries (see I-KSVD vs I-SR, and Proposed vs. Curve-SR), 3. Joint dictionary learning sometimes outperforms separate dictionary learning (Proposed outperforms KSVD-KSVD but KSVD-KSVD outperforms KDRSDL), and 4. Globally optimal solutions outperform locally optimal solutions (see Proposed vs KDRSDL and KSVD-KSVD).

%Our method outperforms both KDRSDL which learns dictionaries jointly but with potential suboptimal local minimizers, and KSVD-KSVD which learns spatial and angular dictionaries separately, as well as fixed spatial-angular dictionaries. 
These results provide a preliminary validation of the importance of separable dictionary learning with global optimality guarantees. Note that the reported results were obtained using a mere sparse LASSO reconstruction algorithm once the dictionaries were estimated, which could likely be improved with more advanced approaches tailored to the case of dMRI data \cite{Gramfort:MIA14}.
%n independently. The comparisons between each of the other methodologies are somewhat less consistent in these findings.
%The organization of each method is outlined in Table~\ref{table:denoising} along with their PSNR values for three different HARDI phantom test slices not used in training. 
%Figure~\ref{fig:denoising} shows the qualitative results of our denoising experiment in comparison to the denoising results of the SR fixed dictionary with a close-up in Figure~\ref{fig:denoisingcloseup}. 
Figure~\ref{fig:denoisingcloseup} shows the qualitative results of our denoising experiment in comparison to the denoising results of the SR fixed dictionary (I-SR). We do not show the qualitative visualization of all methods here for simplicity. KSVD-KSVD and KDRSDL have very similar qualitative results to the proposed method. Then, in Figure~\ref{fig:denoisingcloseupreal} we show denoising results on real HARDI data using our proposed dictionaries with noticeable regions of improvement highlighted in red.
%One recent work \cite{st2016non} learns non-separable, local dictionaries on spatial and angular patches. This differs from our method which learns angular dictionaries for the entire q-space. (Theoretically we can learn global spatial dictionaries but choose local patches for computational purposes).

%The main application for this thesis is for sparse coding and compressed sensing. But, as an important note, because we have been restricted to learning patch-based dictionaries for computational ease, reconstructing a full image by sparsely modeling patches independently or with spatial regularization between patches will be only marginally better than reconstructing each voxel independently with spatial regularization between voxels. Therefore it is currently difficult to validate the performance of our proposed dictionaries in terms of global sparse coding and compressed sensing. As discussed in the background Section~\ref{Background:DictionaryLearning}, in order to relate our patch-based dictionaries to the global signal, we consider using convolutional methods. In the next chapter we will formulate our preliminary work on convolutional spatial-angular sparse coding and convolutional ($k,q$)-CS.

\begin{table}
\centering
\begin{tabular}{|c|c|c|c|c|c|c|c|}
\hline
\cellcolor{yellow!15}Domain & \multicolumn{3}{c|}{\cellcolor{gray!15}Angular} & \multicolumn{4}{c|}{\cellcolor{gray!35}Spatial-Angular}\\ \hline
\cellcolor{yellow!15}Type & \cellcolor{blue!15}Fixed & \cellcolor{blue!15}Fixed & \cellcolor{red!15}Separate & \cellcolor{blue!15}Fixed & \cellcolor{red!15}Separate & \cellcolor{green!15}Joint & \cellcolor{green!15}Joint\\
 \hline
    \cellcolor{yellow!15} Method &\cellcolor{blue!15} I-SR & \cellcolor{blue!15}I-SR+TV &\cellcolor{red!15} I-KSVD &\cellcolor{blue!15} Curve-SR &\cellcolor{red!15} KSVD-KSVD & \cellcolor{green!15}KDRSDL & \cellcolor{green!15}\textbf{Proposed} \\ \hline
    \cellcolor{yellow!15}  Slice 25 & 16.631 & 16.634 & 18.011 & 17.000 & 19.182 & 18.793 & \textbf{19.501} \\ \hline
     \cellcolor{yellow!15} Slice 30 & 16.715 & 16.720 & 16.090 & 17.087 & 17.001 & 16.725 & \textbf{17.221} \\ \hline
     \cellcolor{yellow!15} Slice 35 & 17.311 & 17.323 & 16.679 & 17.793 & 17.675 & 17.418 & \textbf{17.868} \\ \hline
     \cellcolor{yellow!15} Average & 16.886 & 16.892 & 16.927 & 17.293 & 17.953 & 17.645 & \textbf{18.197} \\ \hline
\end{tabular}
     \caption{Peak Signal-to-Noise Ratio (PSNR) denoising results on three different 2D HARDI phantom image slices. We compared the domains of angular vs spatial-angular sparse coding with dictionaries that are either of type fixed (purple), learned in the spatial and angular domains separately (pink), or learned in the spatial-angular domain jointly (green). Denoising using our proposed joint spatial-angular dictionary learning method with global optimality outperforms denoising with both fixed and learned dictionaries from other methods. }
     \label{table:denoising}
\end{table}

%\begin{figure}
%    \centering
%    \includegraphics[width=.45\linewidth,trim= 1500 70 1500 70, clip]{PhantomSNR30.jpg}
%    \includegraphics[width=.45\linewidth,trim= 1500 70 1500 70, clip]{PhantomSNR10.jpg}\\
%    \includegraphics[width=.45\linewidth,trim= 1500 70 1500 70, clip]{PhantomSNR10denoisedOurs.jpg}
%    \includegraphics[width=.45\linewidth,trim= 150 210 190 100, clip]{PhantomSNR10denoisedSRI.jpg}
%    \caption{Results of HARDI phantom denoising experiment. Top left: Original Phantom data with SNR=30 dB. Top right: Noisy version with SNR=10 dB. Bottom left: Denoised reconstruction of noisy phantom using our learned spatial-angular dictionaries with spatial-angular sparse coding. Bottom right: Denoised reconstruction of noisy phantom using a fixed spherical ridgelet dictionary with angular sparse coding (I-SR).  We notice our proposed method produces a more accurate reconstruction in comparison to the original SNR=30 dB. For more detailed visualization see the close-ups in Figure~\ref{fig:denoisingcloseup}.}
%    \label{fig:denoising}
%\end{figure}
\begin{figure}
    \centering
    \includegraphics[width=.45\linewidth,trim= 1500 70 1500 70, clip]{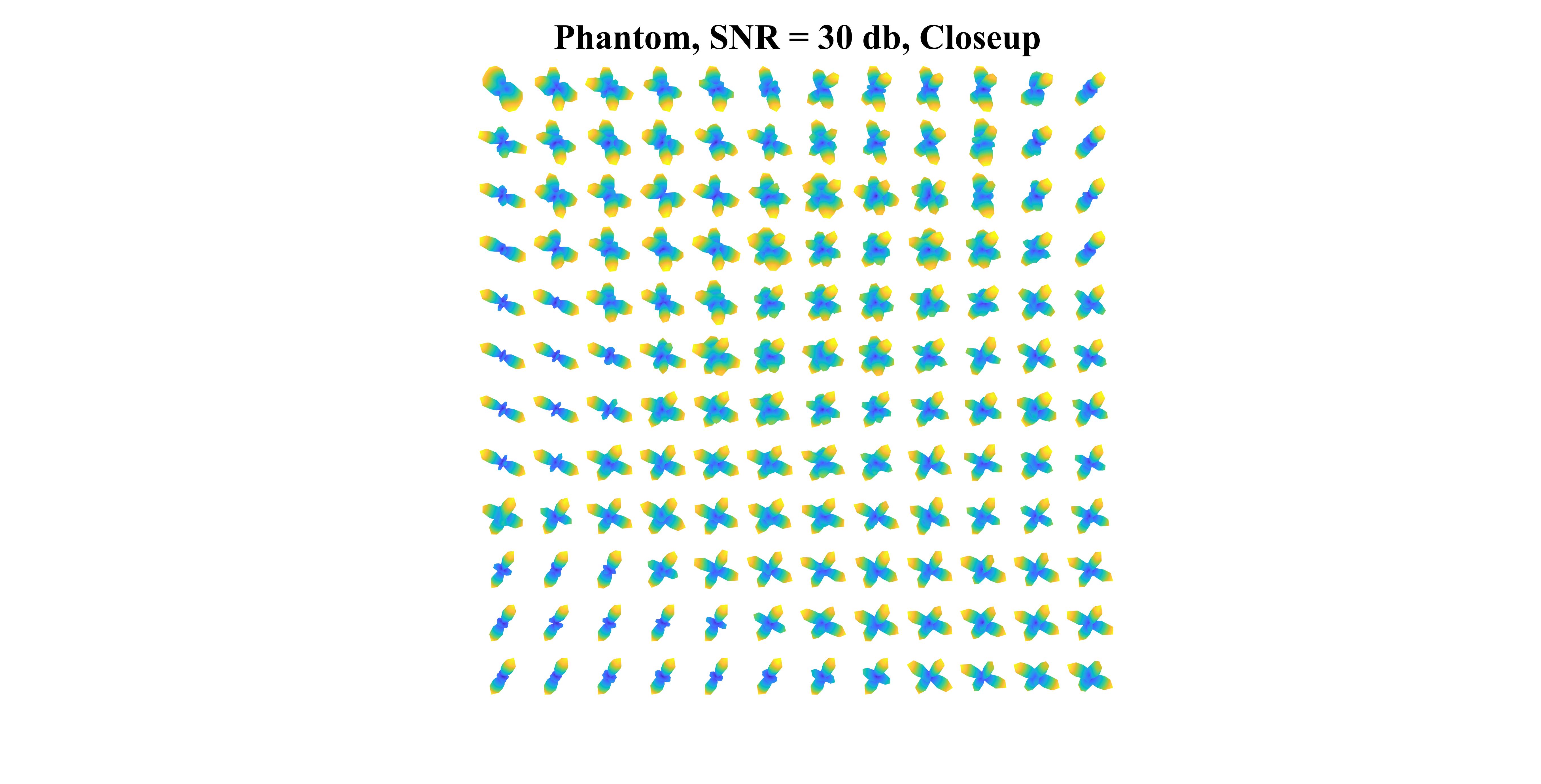}
    \includegraphics[width=.45\linewidth,trim= 1500 70 1500 70, clip]{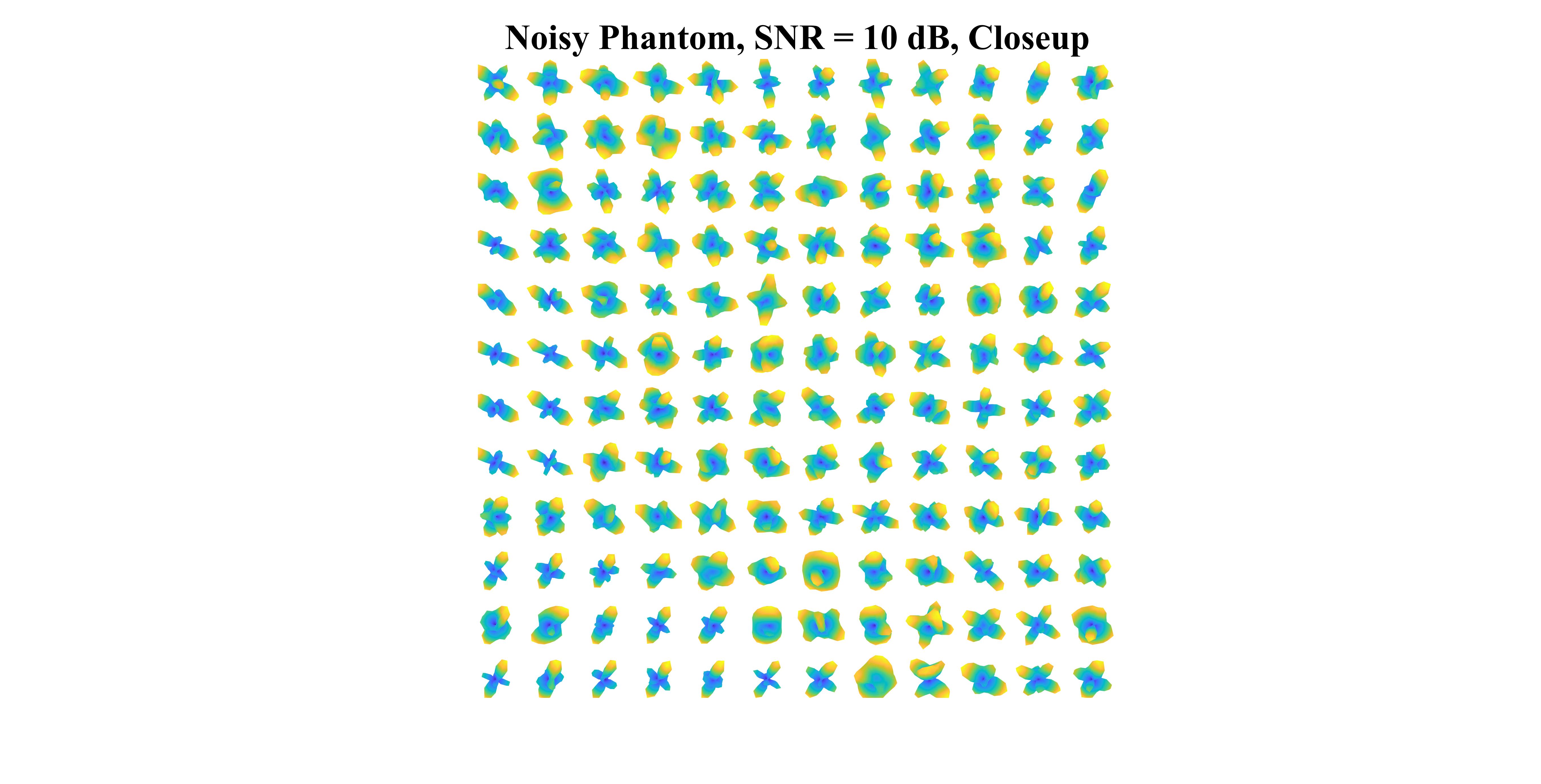}\\
    \includegraphics[width=.45\linewidth,trim= 1500 70 1500 70, clip]{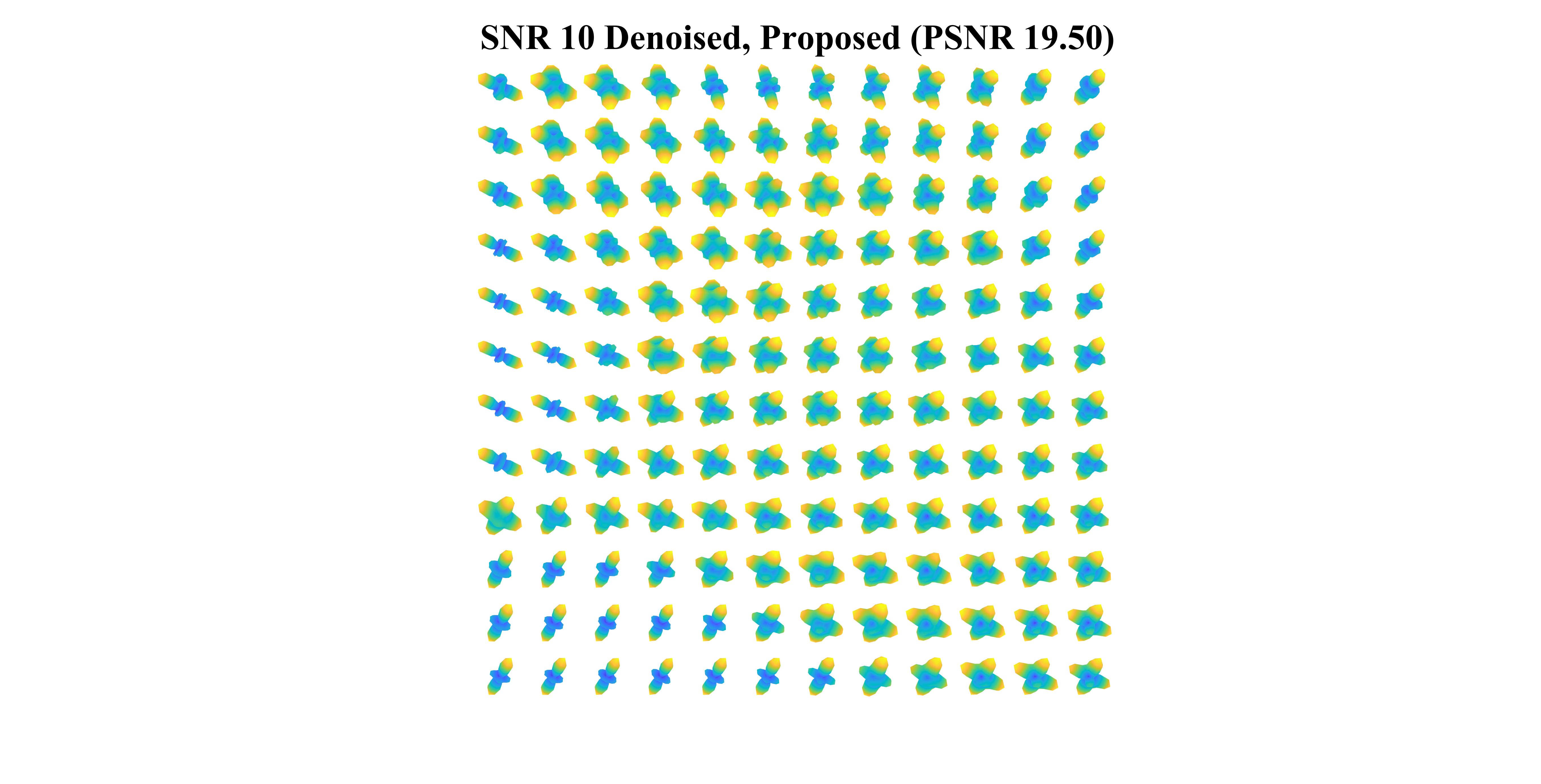}
    \includegraphics[width=.45\linewidth,trim= 400 70 450 70, clip]{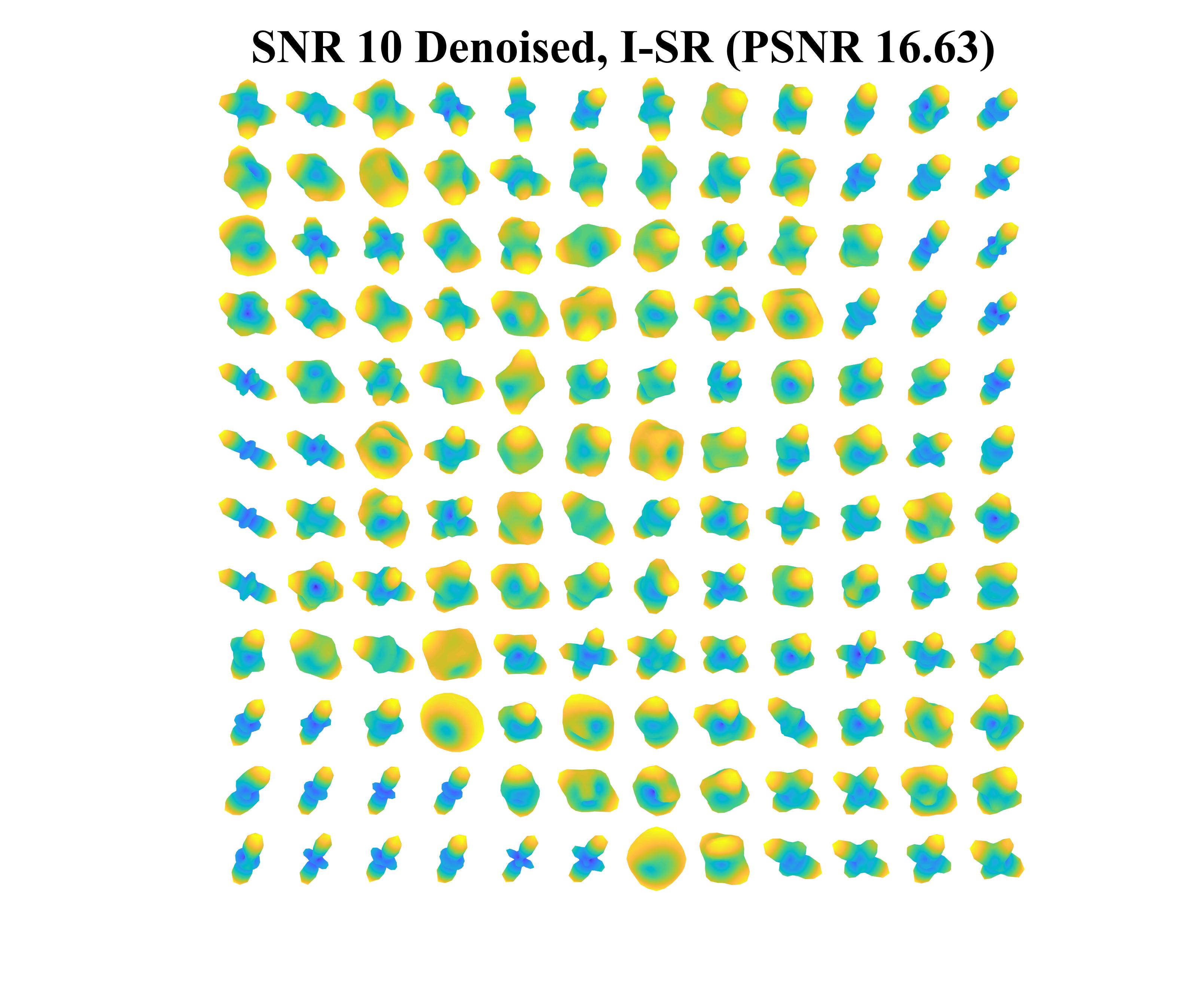}
    \caption{Qualitative results of HARDI phantom denoising experiment. The reconstruction of the noisy SNR=10 dB HARDI phantom (top right) using our proposed spatial-angular dictionary (bottom left) produces a more accurate denoised reconstruction in comparison to the original phantom with SNR=30 dB (top left), than for the fixed spherical ridgelet (SR) dictionary (bottom right).}
    \label{fig:denoisingcloseup}
\end{figure}

\begin{figure}
    \centering
    \includegraphics[width=.48\linewidth,trim= 1950 250 1800 70, clip]{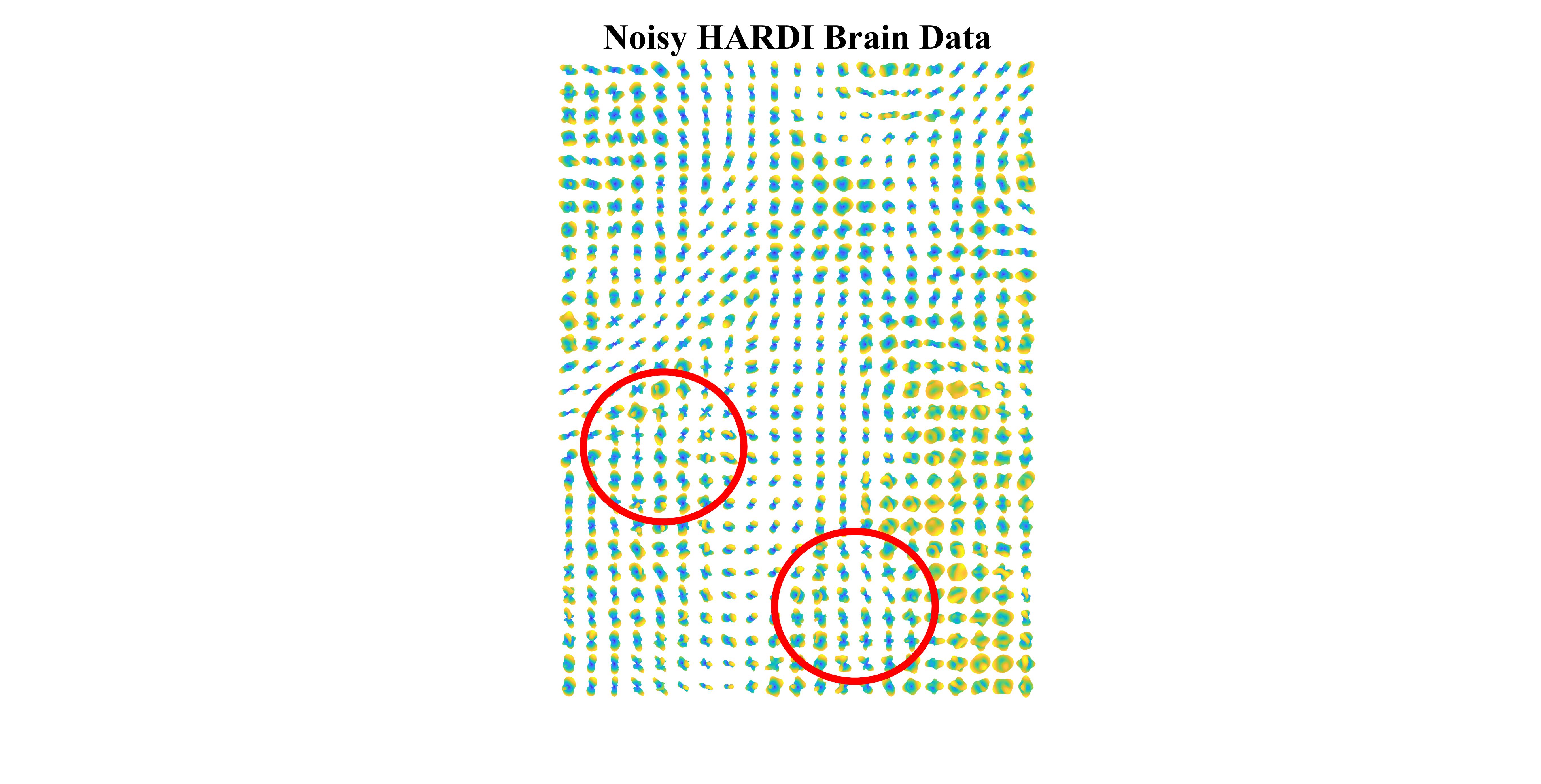}
    \includegraphics[width=.48\linewidth,trim= 1950 250 1800 70, clip]{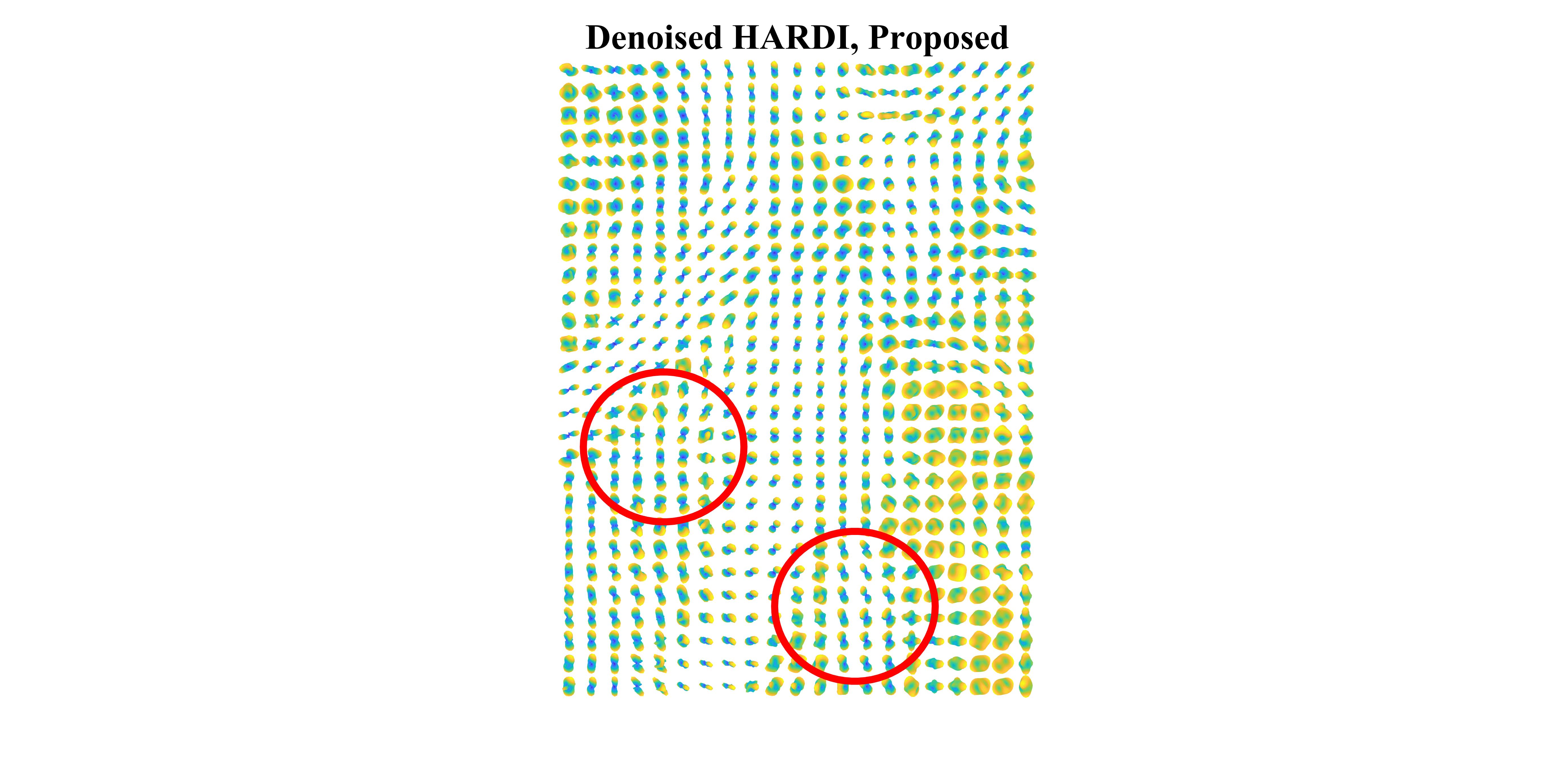}
    \caption{Denoising Real HARDI brain data. Left: Original noisy HARDI brain region. Right: Denoised reconstruction using our learned spatial-angular dictionaries within our spatial-angular sparse coding. Denoising of diffusion signals can allow for more robust fiber tractography with clearer peak directions.}
    \label{fig:denoisingcloseupreal}
\end{figure}

\section{Conclusion}
\label{sec:conclusion}

In this work, we proposed a mathematical formulation of the separable dictionary learning problem for which we are able to derive, to the best of our knowledge, the first conditions of global optimality. To this end, we have framed this problem as a tensor factorization, extending theoretical results from two-factor matrix factorization to the more complex case of three-factor tensor factorization appearing in separable dictionary learning.

With this theoretical base, we have proposed a novel algorithm to find global minima of the separable dictionary learning problem with unspecified dictionary sizes by alternating between a local descent step to a stationary point and a check for global optimality. If the global criteria is not satisfied, the algorithm will append an additional dictionary atom and continue the descent to another stationary point. In this way, our algorithm provides a ``rank-aware" methodology that could provide low-rank or overcomplete solutions, a reasonable midpoint between the low-rank solutions of KDRSDL and the overcomplete solutions of KSVD. This too depends on the initial dictionary size which may be application specific.  Furthermore, the alternation of updates between each separate dictionary is flexible in our algorithm, and can be tailored to specific $\textit{a priori}$ knowledge of the relative dictionary sizes based on the data.

As a proof of concept, we applied the proposed algorithm to the domain of dMRI which is well suited for our framework due to the spatial-angular structure of the data. While most dictionary learning methods for dMRI restrict to learning dictionaries for the angular domain or learn separately spatial and angular dictionaries, we learn both of those jointly in this work. We showed, for a simple denoising procedure, that using spatial and angular dictionaries learned jointly, outperforms dMRI denoising algorithms relying on angular dictionaries alone. Furthermore, we validated that joint learning provides better reconstructions than the alternative of learning spatial and angular dictionaries independently by simpler methods such as KSVD. Finally, our results indicate that having a globally optimal solution also outperforms methods like KDRSDL that may be subject to convergence toward local minima.

In future work we will aim to extend the theory and algorithms presented in this paper to incorporate convolutional methods that will relate local patch dictionaries to the global image for the task of global sparse coding and compressed sensing in diffusion MRI and other applications.

%In this work, we have developed a novel separable dictionary learning algorithm to learn spatial and angular dictionaries for dMRI, jointly.  We show the visual results of our learned dictionaries which may highly resemble complex fiber tract structures, suggesting that we can may be able to very sparsely reconstruct dMRI signals using spatial-angular sparse coding. We validated our method on a denoising experiment, resulting in more accurate reconstruction than with other dictionary learning methods. Furthermore, using recent results from matrix factorization, we have developed an algorithm with proven guarantees of global optimality. The best of our knowledge this is the first separable dictionary learning algorithm to provide such a result of global convergence. We also provide flexibility in terms of the rank or number of dictionary atoms chosen in each of the dictionaries which can provide low-rank or overcomplete solutions. General to any separable dictionary structure, we hope our proposed method may also be adopted to other applications outside of dMRI.

%\linenumbers

%\section*{References}
%\bibliographystyle{plain}
%{\footnotesize
%\bibliography{vidal,dti,math,biomedical,sparse}}
%\pagebreak

\section*{Appendix A}
\label{sec:appendixA}
%\begin{proof}[Proof of Proposition~\ref{appendix:propTheta}]
%For $\alpha \geq 0$, $\theta(\alpha \gamma, \alpha \psi, \alpha c) = ||\alpha\gamma||_2||\alpha\psi||_2||\alpha c||_1 = \alpha^3||\gamma||_2||\psi||_2|| c||_1 = \alpha^3 \theta( \gamma, \psi, c)$, positively homogeneous of degree 3. Because $\theta$ is a multiplication of norms, $\theta$ is positive semi-definite and positive for any triple $(\gamma,\psi,c)$ with $\gamma \neq 0, \psi \neq 0, c \neq 0$. Finally, the last property is trivially verified since $\|\cdot\|$ and $\theta$ are equivalent norms on $\R^{G\times V \times T}$.
%\end{proof}
\begin{proof}[Proof of Proposition~\ref{appendix:propOmega}]
We will assume, in a first phase, that the infimum in \eqref{eq:OmegaT} can be achieved for finite $r_1$ and $r_2$ (which is proved in the last point below) and drop the minimization in $r_1$ and $r_2$ to lighten the derivations.
\begin{enumerate}
\item First, since ${\theta}(\gamma,\psi,c) \geq 0 \ \forall (\gamma,\psi,c)$, we have that $\Omega_{\theta}(\u{X}) \geq 0 \ \forall \u{X}$. Then, the infimum $\Omega_{\theta}(0) = 0$ can be achieved by taking $(\Gamma,\Psi,\u{C}) = (0,0,0)$.  If $\u{X}=\u{C}\times_1\Gamma \times_2 \Psi$ and $\u{X} \neq 0$, we can write equivalently $\u{X}= \sum_{i=1}^{r_1}\sum_{j=1}^{r_2} \Gamma_i \otimes \Psi_j \otimes C_{i,j}$ and there exists $i_0,j_0$ such that $\Gamma_{i_0} \neq 0, \Psi_{j_0} \neq 0, C_{i_0,j_0} \neq 0$ and thus $\Omega_{\theta}(\u{X}) \geq \theta(\Gamma_{i_0},\Psi_{j_0},C_{i_0,j_0}) > 0$ thanks to the second property in Proposition \ref{appendix:propOmega}.
\item With substitution $(\bar{\Gamma},\bar{\Psi},\bar{\u{C}}) := (\alpha^{-1/3}\Gamma,\alpha^{-1/3}\Psi,\alpha^{-1/3}\u{C})$ and using the positive homogeneity of ${\theta}$,
\begin{align*}
\Omega_{\theta}(\alpha \u{X}) &= \inf_{\Gamma,\Psi,\u{C}} \sum_{i=1}^{r_1}\sum_{j=1}^{r_2} {\theta}(\Gamma_i,\Psi_j,C_{i,j})  \ \textnormal{s.t.}  \ \u{C}\times_1\Gamma \times_2 \Psi = \alpha \u{X}\\
&= \inf_{\Gamma,\Psi,\u{C}} \sum_{i=1}^{r_1}\sum_{j=1}^{r_2} {\theta}(\Gamma_i,\Psi_j,C_{i,j})  \ \textnormal{s.t.}  \ (\alpha^{-1/3})\u{C} \times_1 (\alpha^{-1/3})\Gamma \times_2 (\alpha^{-1/3})\Psi =  \u{X}\\
&= \inf_{\bar{\Gamma},\bar{\Psi},\bar{\u{C}}} \sum_{i=1}^{r_1}\sum_{j=1}^{r_2} {\theta}(\alpha^{1/3} \bar{\Gamma_i},\alpha^{1/3}\bar{\Psi}_j,\alpha^{1/3} \bar{C}_{i,j})  \ \textnormal{s.t.}  \ \bar{\u{C}}\times_1 \bar{\Gamma} \times_2 \bar{\Psi} =  \u{X}\\
&= \inf_{\bar{\Gamma},\bar{\Psi},\bar{\u{C}}} \alpha \sum_{i=1}^{r_1}\sum_{j=1}^{r_2} {\theta}(\bar{\Gamma_i},\bar{\Psi}_j, \bar{C}_{i,j})  \ \textnormal{s.t.}  \  \bar{\u{C}}\times_1 \bar{\Gamma} \times_2 \bar{\Psi} =  \u{X}\\
&= \alpha \Omega_{\theta}(\u{X}).
\end{align*}
\item Let $\u{X} = \u{C}^{X}\times_1\Gamma^{X} \times_2\Psi^{X}$ and $\u{Y} = \u{C}^{Y}\times_1\Gamma^{Y} \times_2\Psi^{Y}$ be two $\epsilon$-optimal factorizations, i.e. such that $\sum_{i=1}^{r_1}\sum_{j=1}^{r_2} {\theta}(\Gamma^{X}_{i},\Psi^{X}_{j},C^{X}_{i,j}) \leq \Omega_{\theta}(\u{X}) + \epsilon$ and a similar expression for $\u{Y}$. Now we construct $\Gamma = [\Gamma^{X}, \ \Gamma^{Y}]$, $\Psi = [\Psi^{X}, \ \Psi^{Y}]$, and $\u{C}$ such that for all $t=1,\ldots,T$, $C_t = 
\begin{bmatrix}
C^{X}_{t} & 0\\
0 & C^{Y}_{t}
\end{bmatrix}$. Then $\u{X}+\u{Y} = \u{C} \times_1\Gamma \times_2 \Psi$ and:
\begin{align*}
  \Omega_{\theta}(\u{X}+\u{Y}) &\leq  \sum_{i=1}^{r_1^X+r_1^Y}\sum_{j=1}^{r^X_2+r_Y^2} \theta(\Gamma_i,\Psi_{j},C_{i,j}) \\
  &= \sum_{i=1}^{r_1^X}\sum_{j=1}^{r^X_2} \theta(\Gamma^X_i,\Psi^X_{j},C^X_{i,j}) + \sum_{i=1}^{r_1^Y}\sum_{j=1}^{r^Y_2} \theta(\Gamma^Y_i,\Psi^Y_{j},C^Y_{i,j}) \\
  &\leq \Omega_{\theta}(\u{X}) + \Omega_{\theta}(\u{Y}) +2\epsilon
\end{align*}
where the second line equality results form the fact that $\theta(\Gamma^X_i,\Psi^Y_j,C_{i,j})= \theta(\Gamma^X_i,\Psi^Y_j,0)=0$ and similarly $\theta(\Gamma^Y_i,\Psi^X_j,C_{i,j})=0$. Taking $\epsilon \rightarrow 0$ completes the proof of the triangle inequality.
\item Assuming ${\theta}(-\gamma,\psi,c) = {\theta}(\gamma,\psi,c)$, (as is true for $-\psi$ or $-c$) and setting $\bar{\Gamma} := -\Gamma$,
\begin{align*}
\Omega_{\theta}(-\u{X}) &= \inf_{\Gamma,\Psi,\u{C}} \sum_{i=1}^{r_1}\sum_{j=1}^{r_2} {\theta}(\Gamma_i,\Psi_j,C_{i,j})  \ \textnormal{s.t.}  \ \u{C} \times_1 \Gamma \times_2 \Psi = -\u{X}\\
&= \inf_{\Gamma,\Psi,\u{C}} \sum_{i=1}^{r_1}\sum_{j=1}^{r_2} {\theta}(\Gamma_i,\Psi_j,C_{i,j})  \ \textnormal{s.t.}  \ \u{C} \times_1 -\Gamma \times_2 \Psi = \u{X}\\
&= \inf_{\bar{\Gamma},\Psi,\u{C}} \sum_{i=1}^{r_1}\sum_{j=1}^{r_2} {\theta}(-\bar{\Gamma}_i,\Psi_j,C_{i,j})  \ \textnormal{s.t.}  \ \u{C} \times_1 \bar{\Gamma} \times_2 \Psi  = \u{X}\\
&= \inf_{\bar{\Gamma},\Psi,\u{C}} \sum_{i=1}^{r_1}\sum_{j=1}^{r_2} {\theta}(\bar{\Gamma}_i,\Psi_j,C_{i,j})  \ \textnormal{s.t.}  \ \u{C} \times_1 \bar{\Gamma} \times_2 \Psi  = \u{X}\\
&=\Omega_{\theta}(\u{X}).
\end{align*}
\item In order to show that there exists a global minimum with finite $r_1$ and $r_2$ in the definition of $\Omega_{\theta}$, we start by introducing the function defined by:
\begin{equation}
    \label{eq:OmegaT_tilde}
 \widetilde{\Omega}_{\theta}(\underline{X}) := \inf_{r \in \mathbb{N}_+} \inf_{\substack{\Gamma \in \mathbb{R}^{G \times r}\\\Psi \in \mathbb{R}^{V \times r}\\\Lambda \in \mathbb{R}^{T \times r}}} \sum_{i=1}^{r} \theta(\Gamma_i,\Psi_i,\Lambda_i) \ \ \textnormal{s.t.} \ \  \sum_{i=1}^{r} \Gamma_i \otimes \Psi_i \otimes \Lambda_i = \underline{X}.   
\end{equation}
where $\Gamma_i,\Psi_i,\Lambda_i$ denote the i-th column of the respective matrices. This essentially corresponds to the same definition as $\Omega_{\theta}$ but with the additional constraints that $r_1=r_2=r$ and that $\underline{C}$ is a slice by slice diagonal tensor. In fact, it turns out that the two polar functions are equal, i.e. $\Omega_{\theta}(\u{X}) = \widetilde{\Omega}_{\theta}(\underline{X})$ for all $\u{X} \in \R^{G\times V \times T}$, as we show below.

First, we have $\Omega_{\theta}(\u{X}) \leq \widetilde{\Omega}_{\theta}(\underline{X})$. Indeed, if $\Gamma \in \R^{G\times r}$, $\Psi \in \R^{V \times r}$, $\Lambda \in \R^{r \times T}$ are such that $\sum_{i=1}^{r} \Gamma_i \otimes \Psi_i \otimes \Lambda_i = \u{X}$, we can define the tensor $\u{C} \in \R^{r \times r \times T}$ with, for any $t=1,\ldots,T$,
\begin{equation*}
    C_t = \begin{bmatrix}\Lambda_{t,1} & 0 & \cdots & 0 \\ 0 & \Lambda_{t,2} & \cdots & 0 \\ \vdots & \cdots & \ddots &\vdots \\ 0 & \cdots & \cdots & \Lambda_{t,r} \end{bmatrix}.
\end{equation*}
Then, we can see that for all $t=1,\ldots,T$ 
\begin{align*}
   \left(\u{C}\times_1\Gamma \times_2 \Psi \right)_t &= \Gamma C_t \Psi^T \\
   &= \sum_{i=1}^{r} \Lambda_{t,i} \Gamma_i \Psi_i^T \\
   &= \left(\sum_{i=1}^{r} \Gamma_i \otimes \Psi_i \otimes \Lambda_i \right)_{t}
\end{align*}
and therefore $\u{C}\times_1\Gamma \times_2 \Psi = \u{X}$. Furthermore $\sum_{i=1}^{r} \sum_{j=1}^{r} \theta(\Gamma_i,\Psi_j,C_{i,j}) = \sum_{i=1}^{r} \theta(\Gamma_i,\Psi_i,\Lambda_i)$ due to the fact that, by definition, $C_{i,j}=0$ for $i\neq j$. The inequality follows from the definition of $\Omega_{\theta}$.

Conversely, we show that $\Omega_{\theta}(\u{X}) \geq \widetilde{\Omega}_{\theta}(\underline{X})$. Let $\Gamma \in \R^{G\times r_1}$, $\Psi \in \R^{V \times r_2}$ and $\u{C} \in \R^{r_1 \times r_2 \times T}$ such that $\u{C}\times_1\Gamma \times_2 \Psi = \u{X}$. Define $r=r_1 r_2$ and the lexicographic ordering of pairs $l: \{1,\ldots,r_1\} \times \{1,\ldots,r_2\} \rightarrow \{1,\ldots,r\}$. We also set $\widetilde{\Gamma} \in \mathbb{R}^{G \times r}$ such that $\widetilde{\Gamma}_{l(i,j)} = \Gamma_i$, $\widetilde{\Psi} \in \mathbb{R}^{V \times r}$ such that $\widetilde{\Psi}_{l(i,j)} = \Psi_j$ and $\Lambda_{l(i,j),t}=c_{i,j,t}$. We then obtain for all $t=1,\ldots,T$,
\begin{align*}
   X_t &=(\u{C}\times_1\Gamma \times_2 \Psi)_t \\
   &=\sum_{i=1}^{r_1} \sum_{j=1}^{r_2} c_{i,j,t} \Gamma_i \otimes \Psi_j \\
   &=\sum_{l=1}^{r} \Lambda_{l,t} \widetilde{\Gamma}_l \otimes \widetilde{\Psi}_l \\
   &=\left(\sum_{l=1}^{r} \widetilde{\Gamma}_l \otimes \widetilde{\Psi}_l \otimes \Lambda_{l} \right)_t
\end{align*}
and consequently $\u{X} = \sum_{l=1}^{r} \widetilde{\Gamma}_l \otimes \widetilde{\Psi}_l \otimes \Lambda_{l}$. Now, by construction, we also get that $\sum_{l=1}^r \theta(\widetilde{\Gamma}_l,\widetilde{\Psi}_l,\Lambda_{l}) = \sum_{i=1}^{r_1} \sum_{j=1}^{r_2} \theta(\Gamma_i,\Psi_j,C_{i,j})$.  Consequently, any value of the minimization problem \eqref{eq:OmegaT} can be obtained by \eqref{eq:OmegaT_tilde} thanks to the previous transformation, giving the desired inequality. 

We now only need to show that a global minimum in \eqref{eq:OmegaT_tilde} can be achieved with a finite $r$, which will give a global minimum of \eqref{eq:OmegaT} with $r_1=r_2=r$. We can follow an argument similar to the one presented in \cite{Haeffele:PAMI2019} that we briefly recap. Let $\Theta \subset \mathbb{R}^{G \times V \times T}$ defined by $\Theta = \{\underline{X}: \ \exists (\gamma,\psi,\lambda)/ \ \underline{X} = \gamma \otimes \psi \otimes \lambda \ \text{and} \ \theta(\gamma,\psi,\lambda) \leq 1 \}$ which is a compact subset of $\mathbb{R}^{G \times V \times T}$ thanks to the third condition in Definition \ref{def:theta_T}. With the same reasoning as \cite{Haeffele:PAMI2019}, we know that $\widetilde{\Omega}_{\theta}$ is equivalent to the following gauge function on the convex hull of $\Theta$:
\begin{equation*}
    \widetilde{\Omega}_{\theta}(\underline{X}) = \inf \{\mu: \ \mu \geq 0, \underline{X} \in \mu \text{conv}(\Theta) \}.
\end{equation*}
Now since $\Theta$ and thus $\text{conv}(\Theta)$ are compact sets, the previous infimum over $\mu$ is achieved for a certain $\mu^* \geq 0$. Then $\underline{X} \in \mu^* \text{conv}(\Theta)$ and from Caratheodory's theorem, we know that any point in $\text{conv}(\Theta)$ can be written as a finite convex combination of a most $G \times V \times T$ elements in $\Theta$. In other words, there exist $(\Gamma_i,\Psi_i,\Lambda_i)_{i=1,\ldots,r}$ with $r \leq G \times V \times T$ and $\beta_1,\ldots,\beta_r \geq 0$ with $\beta_1+\ldots+\beta_r=1$ such that 
\begin{equation*}
\underline{X} = \mu^* \sum_{i=1}^{r} \beta_i \Gamma_i \otimes \Psi_i \otimes \Lambda_i = \sum_{i=1}^{r} \Gamma^*_i \otimes \Psi^*_i \otimes \Lambda^*_i,
\end{equation*}
with $\Gamma_i^* = \sqrt[3]{\beta_i\mu^*} \Gamma_i, \Psi_i^* = \sqrt[3]{\beta_i \mu^*} \Psi_i, \Lambda^*_i=\sqrt[3]{\beta_i \mu^*} \Lambda_i$ for all $i$. Now since $\mu^* = \widetilde{\Omega}_{\theta}(\underline{X})$ and by the positive homogeneity of $\theta$, we obtain: 
\begin{equation*}
 \sum_{i=1}^{r} \theta(\Gamma^*_i , \Psi^*_i , \Lambda^*_i) = \mu^* \sum_{i=1}^{r} \beta_i \theta(\Gamma_i , \Psi_i , \Lambda_i) \leq \widetilde{\Omega}_{\theta}(\underline{X}),
\end{equation*}
which implies that $(\Gamma^*_i , \Psi^*_i , \Lambda^*_i)$ is a global minimum of \eqref{eq:OmegaT_tilde} with finite $r$.
\end{enumerate}
\end{proof}

\section*{Appendix B}
\label{sec:appendixB}
\subsection*{Proximal Gradient Descent with Nesterov Acceleration}
Here in Algorithm~\ref{appendix:nesterovDL} we formalize the Proximal Gradient Descent Algorithm~\ref{alg:prox} with the additional process of Nesterov Acceleration to speed up the rate of convergence. We use Nesterov Acceleration within our current implementation.
%\begin{align}
%    \Gamma^{k+1}_i &= \textnormal{prox}_{\xi_i^k||\cdot||_2}( \Gamma^{k}_i - \xi_i^k [\nabla_{\Gamma^k} \ell]_i)\label{eq:updateGammaNes}\\
%    c_{i,j,t}^{k+1} &= \textnormal{prox}_{\kappa_{i,j}^k|\cdot|}( c_{i,j,t}^{k} - \kappa_{i,j}^k [\nabla_{C_t^k} \ell]_{i,j}) \label{eq:updateCtNes}\\
 %      \Psi^{k+1}_j &= \textnormal{prox}_{\pi_j^k||\cdot||_2}( \Psi^{k}_j - \pi_j^k [\nabla_{\Psi^k} \ell]_j).\label{eq:updatePsiNes}
%\end{align}

%\begin{algorithm}
%  \caption{Proximal Gradient Descent}
%  \begin{algorithmic}
%  \label{alg:prox}
%  \STATE Initialize: $k=0,\Gamma^0, \Psi^0, %\u{C}^0, \lambda, r_1, r_2.$
%  \WHILE{error $> \epsilon$}
%  \STATE Update $\Gamma^k$ via \eqref{eq:updateGammaNes}
%  \STATE Update $\u{C}^k$ via \eqref{eq:updateCtNes}
%  \STATE Update $\Psi^k$ via \eqref{eq:updatePsiNes}
%  \STATE Nesterov Acceleration via Algorithm~\ref{alg:nesterovDL}

 \begin{algorithm}
\caption{Proximal Gradient Descent with Nesterov Acceleration}
\begin{algorithmic}
\label{appendix:nesterovDL}
  \STATE Initialize: $k=0,\check{\Gamma}^0, \check{\Psi}^0, \u{\check{C}}^0, \lambda, r_1, r_2.$
  \WHILE{error $> \epsilon$}
    \STATE $\Gamma^{k+1}_i = \textnormal{prox}_{\xi_i^k||\cdot||_2}( \check{\Gamma}^{k}_i - \xi_i^k [\nabla_{\check{\Gamma}^k} \ell]_i)$\label{eq:updateGammaNes}
    \STATE $c_{i,j,t}^{k+1} = \textnormal{prox}_{\kappa_{i,j}^k|\cdot|}( \check{c}_{i,j,t}^{k} - \kappa_{i,j}^k [\nabla_{\check{C}_t^k} \ell]_{i,j})$ \label{eq:updateCtNes}
    \STATE $\Psi^{k+1}_j = \textnormal{prox}_{\pi_j^k||\cdot||_2}( \check{\Psi}^{k}_j - \pi_j^k [\nabla_{\check{\Psi}^k} \ell]_j).$\label{eq:updatePsiNes}
\IF{$f(\Gamma^k,\Psi^k,\u{C}^k) < f(\Gamma^{k-1},\Psi^{k-1},\u{C}^{k-1})$}
\STATE $s_k = (1 + \sqrt{1 + 4s_{k-1}^2})/2$
\STATE $\mu = (s_{k-1} -1)/2$
\STATE $\mu_\Gamma = \min(\mu, \sqrt{L_\Gamma^{k-1} / L_\Gamma^k})$
\STATE $\mu_{C_t} = \min(\mu, \sqrt{L_{C_t}^{k-1} / L_{C_t}^k}) \ \forall t$
\STATE $\mu_\Psi = \min(\mu, \sqrt{L_\Psi^{k-1} / L_\Psi^k})$
\STATE $\check{\Gamma}^{k+1} = \Gamma^k + \mu_\Gamma(\Gamma^k - \Gamma^{k-1})$
\STATE $\check{C}_t^{k+1} = C_t^k + \mu_{C_t}(C_t^k - C_t^{k-1})$
\STATE $\check{\Psi}^{k+1} = \Psi^k + \mu_\Psi(\Psi^k - \Psi^{k-1})$
\ELSE
\STATE $s_k = s_{k-1}$
\STATE $\check{\Gamma}^{k+1} = \Gamma^{k-1}$
\STATE $\check{C}_t^{k+1} = C_t^{k-1} \ \forall t$
\STATE $\check{\Psi}^{k+1} = \Psi^{k-1}$
\ENDIF
  \STATE $k \rightarrow k+1$
  \ENDWHILE
  \RETURN stationary point $(\tilde{\Gamma},\tilde{\Psi},\tilde{\u{C}})$
  %\end{algorithmic}
  %\end{algorithm}
\end{algorithmic}
\end{algorithm}

\end{document}